\documentclass[reqno,11pt]{amsart}

\usepackage{amscd}
\usepackage{amsmath}
\usepackage{amssymb}
\usepackage[american]{babel}
\usepackage{bbm}
\usepackage{bookmark}
\usepackage{cmap} 
\usepackage{dsfont}
\usepackage{enumerate}
\usepackage{epigraph}
\usepackage[mathscr]{euscript}
\usepackage[myheadings]{fullpage}
\usepackage{graphicx}
\usepackage[geometry]{ifsym}
\usepackage{mathabx}
\usepackage{accents}
\usepackage{mathrsfs}
\usepackage{mathtools}
\usepackage{refcount}
\usepackage{setspace}
\usepackage{stmaryrd}
\usepackage{thinsp}
\usepackage{verbatim}
\usepackage[all,2cell]{xy} \UseTwocells
\usepackage{xr}
\usepackage[multiple]{footmisc}
\usepackage{hyperref}

\usepackage{tikz}
\usetikzlibrary{decorations.pathmorphing}
\usetikzlibrary{cd}

\numberwithin{equation}{subsection}

\newtheorem{thm}{Theorem}[subsection]
\newtheorem*{thm*}{Theorem}
\newtheorem*{mainthm}{Main Theorem}
\newtheorem{cor}[thm]{Corollary}
\newtheorem*{cor*}{Corollary}
\newtheorem{lem}[thm]{Lemma}

\newtheorem{prop}[thm]{Proposition}
\newtheorem{prop-const}[thm]{Proposition-Construction}

\newtheorem{conjecture}{Conjecture}[subsection]
\newtheorem*{conjecture*}{Conjecture}

\newtheorem*{princ*}{Principle}

\theoremstyle{remark}
\newtheorem{rem}[thm]{Remark}
\newtheorem{example}[thm]{Example}

\newcounter{steps}[thm]
\newtheorem{step}[steps]{Step}

\newcommand{\xar}[1]{\xrightarrow{#1}}
\newcommand{\rar}[1]{\xar{#1}}
\newcommand{\isom}{\rar{\simeq}}

\newcommand{\into}{\hookrightarrow}
\newcommand{\onto}{\twoheadrightarrow}

\newcommand{\Loc}{\on{Loc}}

\newcommand{\bDelta}{\mathbf{\Delta}}

\newcommand{\bA}{{\mathbb A}}
\newcommand{\bB}{{\mathbb B}}

\newcommand{\bD}{{\mathbb D}}

\newcommand{\bG}{{\mathbb G}}

\newcommand{\bQ}{{\mathbb Q}}

\newcommand{\bV}{{\mathbb V}}

\newcommand{\cI}{{\mathcal I}}

\newcommand{\sA}{{\EuScript A}}

\newcommand{\sC}{{\EuScript C}}
\newcommand{\sD}{{\EuScript D}}
\newcommand{\sE}{{\EuScript E}}
\newcommand{\sF}{{\EuScript F}}
\newcommand{\sG}{{\EuScript G}}
\newcommand{\sH}{{\EuScript H}}

\newcommand{\sL}{{\EuScript L}}
\newcommand{\sM}{{\EuScript M}}
\newcommand{\sN}{{\EuScript N}}
\newcommand{\sO}{{\EuScript O}}
\newcommand{\sP}{{\EuScript P}}

\newcommand{\sS}{{\EuScript S}}

\newcommand{\sW}{{\EuScript W}}

\newcommand{\fG}{{\mathfrak G}}

\newcommand{\fZ}{{\mathfrak Z}}

\newcommand{\fb}{{\mathfrak b}}

\newcommand{\fg}{{\mathfrak g}}
\newcommand{\fh}{{\mathfrak h}}

\newcommand{\fk}{{\mathfrak k}}
\newcommand{\fl}{{\mathfrak l}}

\newcommand{\fn}{{\mathfrak n}}

\newcommand{\fs}{{\mathfrak s}}
\newcommand{\ft}{{\mathfrak t}}

\newcommand{\fz}{{\mathfrak z}}

\newcommand{\on}{\operatorname}

\newcommand{\ul}{\underline}

\newcommand{\mathendash}{\text{\textendash}}

\newcommand{\Ker}{\on{Ker}}
\newcommand{\Coker}{\on{Coker}}
\newcommand{\End}{\on{End}}
\newcommand{\Hom}{\on{Hom}}
\newcommand{\Ext}{\on{Ext}}
\newcommand{\Aut}{\on{Aut}}
\newcommand{\Spec}{\on{Spec}}
\newcommand{\Spf}{\on{Spf}}

\newcommand{\id}{\on{id}}

\newcommand{\Ad}{\on{Ad}}

\newcommand{\ind}{\on{ind}}

\newcommand{\Rep}{\mathsf{Rep}}

\newcommand{\act}{\on{act}}
\newcommand{\actson}{\curvearrowright}
\newcommand{\coact}{\on{coact}}

\renewcommand{\dot}{\bullet}

\newcommand{\Fun}{\on{Fun}}

\newcommand{\vph}{\varphi}

\newcommand{\Vect}{\mathsf{Vect}}

\newcommand{\Res}{\on{Res}}
\newcommand{\Gr}{\on{Gr}}

\newcommand{\Whit}{\mathsf{Whit}}

\newcommand{\LocSys}{\on{LocSys}}

\newcommand{\Op}{\on{Op}}

\renewcommand{\mod}{\mathendash\mathsf{mod}}

\newcommand{\bimod}{\mathendash\mathsf{bimod}}

\newcommand{\sinf}{\!\frac{\infty}{2}} 

\newcommand{\colim}{\on{colim}}
\newcommand{\Cat}{\mathsf{Cat}}
\newcommand{\DGCat}{\mathsf{DGCat}}
\newcommand{\ShvCat}{\mathsf{ShvCat}}

\renewcommand{\lim}{\on{lim}}
\newcommand{\Ind}{\mathsf{Ind}}
\newcommand{\Pro}{\mathsf{Pro}}

\newcommand{\TwoHom}{\mathsf{Hom}}
\newcommand{\TwoEnd}{\mathsf{End}}

\newcommand{\Tot}{\on{Tot}}
\newcommand{\heart}{\heartsuit}

\newcommand{\Oblv}{\on{Oblv}}
\newcommand{\Av}{\on{Av}}

\newcommand{\ld}{\check}
\newcommand{\Hecke}{\on{Hecke}}
\newcommand{\Heckez}{\Hecke_{\fz}}

\newcommand{\Perf}{\mathsf{Perf}}
\newcommand{\Coh}{\mathsf{Coh}}
\newcommand{\IndCoh}{\mathsf{IndCoh}}
\newcommand{\QCoh}{\mathsf{QCoh}}

\newcommand{\Alg}{\mathsf{Alg}}
\newcommand{\CoAlg}{\mathsf{CoAlg}}
\newcommand{\ComAlg}{\mathsf{ComAlg}}

\newcommand{\Lie}{\on{Lie}}

\renewcommand{\o}[1]{\mathring{#1}}

\renewcommand{\subset}{\subseteq}
\renewcommand{\supset}{\supseteq}

\renewcommand{\sl}{\fs\fl}

\newcommand{\rightrightrightarrows}{
\mathrel{\vcenter{\mathsurround0pt
\ialign{##\crcr
\noalign{\nointerlineskip}$\rightarrow$\crcr
\noalign{\nointerlineskip}$\rightarrow$\crcr
\noalign{\nointerlineskip}$\rightarrow$\crcr
}}}}

\newcommand{\ord}{\on{ord}}

\makeatletter
\newcommand{\biggg}{\bBigg@{4}}
\newcommand{\Biggg}{\bBigg@{5}}
\makeatother

\date{\today}

\begin{document}

\frenchspacing

\setlength{\epigraphwidth}{0.4\textwidth}
\renewcommand{\epigraphsize}{\footnotesize}

\begin{abstract}

We prove the rank 1 case of a conjecture of 
Frenkel-Gaitsgory: critical level Kac-Moody representations
with regular central characters 
localize onto the affine Grassmannian. 
The method uses an analogue in local geometric Langlands
of the existence of Whittaker models for most
representations of $GL_2$ over a non-Archimedean field.

\end{abstract}

\title{Affine Beilinson-Bernstein localization at the critical level for $GL_2$}

\author{Sam Raskin}

\address{The University of Texas at Austin, 
Department of Mathematics, 
RLM 8.100, 2515 Speedway Stop C1200, 
Austin, TX 78712}

\email{sraskin@math.utexas.edu}

\maketitle

\setcounter{tocdepth}{1}
\tableofcontents

\section{Introduction}

\subsection{}

More than a decade ago, Frenkel and Gaitsgory initiated an ambitious
program to relate geometric representation theory of
(untwisted) affine Kac-Moody algebras at critical level
to geometric Langlands, following Beilinson-Drinfeld 
\cite{hitchin} and \cite{chiral} and Feigin-Frenkel, e.g., 
\cite{ff-critical}.

We refer the reader to \cite{fg2}
for an introduction to this circle of ideas. 
The introduction to \cite{fg-spherical}
and the work \cite{quantum-langlands-summary} 
may be helpful supplements.

While Frenkel-Gaitsgory were extraordinarily successful
in developing representation theory at critical level
(highlights include \cite{fg1}, \cite{fg2}, 
\cite{fg-wakimoto}, \cite{fg-loc}, \cite{fg-spherical},
and \cite{dmod-aff-flag}),
their ambitious program left many open problems. 
Most of these problems 
are dreams that are not easy to formulate precisely. 

In contrast, their conjecture on critical level
localization for the affine Grassmannian is a concrete
representation theoretic problem. 
It remains the major such problem left
open by their work.
In this paper, we prove the
Frenkel-Gaitsgory localization conjecture for rank 1 groups.

Below, we recall the context for and statement of the 
Frenkel-Gaitsgory conjecture,
the progress that they made on it, and outline the argument
used in the present paper for $GL_2$.

\subsection{Notation}\label{ss:intro-notation}

In what follows, $G$ denotes a split reductive group over a
field $k$ of characteristic $0$. We fix $B \subset G$ a
Borel subgroup with unipotent radical $N$ and Cartan $T = B/N$.
We let $\check{G}$ denote the Langlands dual group to $G$, and similarly
$\check{B}$ and so on.

We let e.g. $G(K)$ denote the algebraic loop group of $G$, which 
is a group indscheme of ind-infinite type. We let $G(O) \subset G(K)$
denote its arc subgroup and $\Gr_G \coloneqq G(K)/G(O)$ the affine
Grassmannian. We refer to \cite{hitchin} for further discussion of these
spaces and \cite{dmod} for definitions of $D$-modules
in this context.

We follow the notational convention that all categories are assumed derived;
e.g., $A\mod$ denotes the DG (derived) category of $A$-modules. 
For $\sC$ a DG category with a given $t$-structure, we let
$\sC^{\heart}$ denote the corresponding abelian category.

\subsection{Affine Kac-Moody algebras}

Before recalling the Frenkel-Gaitsgory conjecture, 
we need to review the representation theory of affine Kac-Moody 
algebras at critical level.

\subsection{}

Recall that for a \emph{level} $\kappa$,
by which we mean an $\Ad$-invariant symmetric 
bilinear form on $\fg$, there is an associated 
central extension:

\[
0 \to k \to \widehat{\fg}_{\kappa} \to 
\fg((t)) \coloneqq \fg \underset{k}{\otimes} k((t)) \to 0.
\]

\noindent This extension is defined by a standard 2-cocycle
that vanishes on $\fg[[t]] \coloneqq \fg\otimes_k k[[t]]$; 
in particular, the embedding 
$\fg[[t]] \into \fg((t))$ canonically lifts to 
an embedding $\fg[[t]] \into \widehat{\fg}_{\kappa}$.

\subsection{}

By a \emph{representation} of $\widehat{\fg}_{\kappa}$ on a
(classical) vector space $V \in \Vect^{\heart}$,
we mean an action of the Lie algebra $\widehat{\fg}_{\kappa}$
such that every $v \in V$ is annihilated by
$t^N\fg[[t]]$ for $N \gg 0$ and such that 
$1 \in k \subset \widehat{\fg}_{\kappa}$ acts
by the identity. 

For instance, the \emph{vacuum} module
$\bV_{\kappa} \coloneqq \ind_{\fg[[t]]}^{\widehat{\fg}_{\kappa}}(k)$
is such a representation. Here $\ind$ denotes induction, and
we are abusing notation somewhat: 
we really mean to induce from $k \oplus \fg[[t]]$ the
module $k$ on which $k$ acts by the identity and $\fg[[t]]$ acts
trivially;
since we only consider representations on which 
$k \subset \widehat{\fg}_{\kappa}$ acts by the identity, we
expect this does not cause confusion.

We denote the abelian category of representations of 
$\widehat{\fg}_{\kappa}$ by $\widehat{\fg}_{\kappa}\mod^{\heart}$.
The appropriate DG category $\widehat{\fg}_{\kappa}\mod$ was defined in 
\cite{dmod-aff-flag} \S 23; see \cite{km-indcoh}, 
\cite{whit} Appendix A, or \cite{methods} for other expositions.

We recall the pitfall that the forgetful
functor $\Oblv:\widehat{\fg}_{\kappa}\mod \to \Vect$
is not conservative, i.e., it sends non-zero objects to 
zero.\footnote{See \cite{whit} \S 1.18 for some discussion of this point.}

But one key advantage of $\widehat{\fg}_{\kappa}\mod$ over other
possible ``derived categories" of $\widehat{\fg}_{\kappa}$-modules is that
it admits a \emph{level $\kappa$ action of $G(K)$}: see 
\cite{methods} \S 11 for the construction and definitions. 

\subsection{}

We let $U(\widehat{\fg}_{\kappa})$ denote the 
(twisted) topological enveloping algebra of $\widehat{\fg}_{\kappa}$
(with the central element $1 \in \widehat{\fg}_{\kappa}$ set to the
identity).
For our purposes, 
$U(\widehat{\fg}_{\kappa})$ is the pro-representation
of $\widehat{\fg}_{\kappa}$:

\[
\underset{n}{\lim} \, \ind_{t^n \fg[[t]]}^{\widehat{\fg}_{\kappa}}(k) \in 
\Pro(\widehat{\fg}_{\kappa}\mod^{\heart}).
\]

The underlying pro-vector space is naturally an 
$\overset{\rightarrow}{\otimes}$-algebra
algebra in the sense of \cite{methods} \S 3 
by construction,
its \emph{discrete} modules (in $\Vect^{\heart}$) 
are the same as (classical) representations of $\widehat{\fg}_{\kappa}$.

\subsection{}

Let $D_{\kappa}(\Gr_G)$ denote the DG category of $\kappa$-twisted
$D$-modules on $\Gr_G$. There is a global sections functor:

\[
\Gamma^{\IndCoh}(\Gr_G,-): D_{\kappa}(\Gr_G) \to \widehat{\fg}_{\kappa}\mod.
\]

\noindent This functor is a morphism of categories acted on by
$G(K)$ and sends the skyscraper $D$-module $\delta_1 \in D_{\kappa}(\Gr_G)$
to the vacuum module $\bV_{\kappa}$.

\subsection{Affine Beilinson-Bernstein localization?}\label{ss:bb-loc}

Recall the finite-dimensional \emph{Beilinson-Bernstein} localization theorem:

\begin{thm}[\cite{bb-loc}]\label{t:bb}

The functor:

\[
\Gamma(G/B,-):D(G/B) \to \fg\mod_0 
\]

\noindent is a $t$-exact equivalence of categories. Here
$D(G/B)$ is the DG category of $D$-modules and $\Gamma(G/B,-)$
is the left $D$-module global sections functor;
$\fg\mod_0$ is the DG category of modules
over $U(\fg) \otimes_{Z(\fg)} k$ for $Z(\fg)$ is the
center of $U(\fg)$ and $Z(\fg) \to k$ the restriction of
the augmentation $U(\fg) \to k$.

\end{thm}

Almost as soon as Beilinson and Bernstein proved their localization theorem, 
there was a desire for an affine analogue that would apply for
$\Gr_G$ or the affine flag variety. 
Results soon emerged in work of Kashiwara-Tanisaki, beginning with
\cite{kt1} for so-called \emph{negative} levels $\kappa$. 

The results of Kashiwara-Tanisaki suffice for applications to 
Kazhdan-Lusztig problems. 
However, their theorems are less
satisfying than Theorem \ref{t:bb}: they do not provide an equivalence
of categories, but only a fully faithful functor. Conceptually,
this is necessarily the case because for negative $\kappa$,
the center of $U(\widehat{\fg}_{\kappa})$ consists only
of scalars, so it is not possible to define an analogue
of the category $\fg\mod_0$.\footnote{However, see \cite{beilinson-heisenberg}
for some speculations; the suggestion is that $\widehat{\fg}_{\kappa}\mod$
should be considered not as decomposing over the spectrum of its
center but over a moduli of local systems on the punctured disc.}

As observed by Frenkel-Gaitsgory, this objection does not apply 
at \emph{critical} level, as we recall below.

\subsection{Critical level representation theory}\label{ss:intro-crit}

For the so-called \emph{critical} value of $\kappa$,
the representation theory of the Kac-Moody algebra behaves quite differently 
from other levels. For completeness, we recall that
critical level is $\frac{-1}{2}$ times the Killing form. 
We let $crit$ denote the corresponding symmetric bilinear 
form; in particular, we use $\widehat{\fg}_{crit}$ (resp. $\bV_{crit}$)
in place of $\widehat{\fg}_{\kappa}$ (resp. $\bV_{\kappa}$).

\begin{thm}[Feigin-Frenkel]\label{t:feigin-frenkel}

\begin{enumerate}

\item 

The (non-derived) center 
$\fZ$ of $U(\widehat{\fg}_{crit})$ is canonically isomorphic to the
commutative pro-algebra of functions on the 
ind-scheme $\Op_{\ld{G}}$ of \emph{opers (on the punctured disc)}
for the Langlands dual group $\ld{G}$:

\[
\Op_{\ld{G}} \coloneqq \big(\ld{f}+\ld{\fb}((t))\big)dt/\ld{N}(K)
\]

\noindent where $\ld{N}(K) \subset \ld{G}(K)$ acts on 
$\ld{\fg}((t))dt$ by gauge transformations and $\check{f}$ is
a principal nilpotent element with $[\check{\rho},\check{f}] = -\check{f}$.

We recall that, as for the Kostant slice, 
$\Op_{\ld{G}}$ is (somewhat non-canonically) isomorphic
to an affine space that is infinite-dimensional in both
ind and pro senses (like the affine space corresponding to the
$k$-vector space $k((t))$).

\item 

The natural map:

\[
\fZ \to \fz\coloneqq \End_{\widehat{\fg}_{crit}\mod^{\heart}}(\bV_{crit})
\]

\noindent is surjective and fits into a commutative diagram:

\[
\xymatrix{
\fZ \ar[r] \ar[d]^{\simeq} & \fz \ar[d]^{\simeq} \\
\Fun(\Op_{\ld{G}}) \ar[r] & \Fun(\Op_{\ld{G}}^{reg}).
}
\]

\noindent Here $\Op_{\ld{G}}^{reg} \coloneqq 
(\ld{f}+\ld{\fb}[[t]])dt/\ld{N}(O)$ is
the scheme of \emph{regular} opers, 
on the (non-punctured) disc; 
we recall that the natural
map $\Op_{\ld{G}}^{reg} \to \Op_{\ld{G}}$ is a closed embedding.

\end{enumerate}

\end{thm}

We refer to \cite{ff-critical} and \cite{frenkel-wakimoto} 
for proofs of most of these
statements; the only exception is that the map
$\Fun(\Op_{\ld{G}}) \into \fZ$ constructed using \cite{ff-critical}
is an isomorphism, which is shown as \cite{hitchin} 
Theorem 3.7.7.\footnote{In fact, the mere existence of this 
map (and its good properties) is all we really 
need. That the map is an isomorphism is nice, but not
strictly necessary.}

We refer to \cite{fg2} \S 1 and \cite{hitchin} \S 3 
for an introduction to opers. We highlight, as in 
\emph{loc. cit}., that
$\Op_{\ld{G}}$ (resp. $\Op_{\ld{G}}^{reg}$) is
a moduli space of de Rham $\ld{G}$-local systems
on the punctured (resp. non-punctured disc)
with extra structure.

\begin{rem}\label{r:opers-defin}

The definition of opers here is slightly different from the
original one used by Beilinson-Drinfeld and rather
follows the definition advocated by Gaitsgory. In this definition,
an isogeny of reductive groups induces an isomorphism on 
spaces of opers, unlike in \cite{hitchin}. For $\ld{G}$ semisimple,
the definition here coincides with the definition in \cite{hitchin}
for the associated adjoint group. We refer to \cite{barlev-notes}
for a more geometric discussion.

\end{rem}

\subsection{Localization at critical level}\label{ss:intro-loc-start}

The functor:

\[
\Gamma^{\IndCoh}: D_{crit}(\Gr_G) \to \widehat{\fg}_{crit}\mod
\]

\noindent fails to be an equivalence for two related reasons. 

First, recall that $\Gamma^{\IndCoh}(\delta_1) = \bV_{crit}$. As for
any skyscraper $D$-module, $\End(\delta_1) = k$,
while by Theorem \ref{t:feigin-frenkel}, $\bV_{crit}$ has a large
endomorphism algebra. Worse still, $\bV_{crit}$ has
large self-$\Ext$s by \cite{frenkel-teleman} and \cite{fg2} \S 8.

Moreover, there are central character restrictions on
the essential image of $\Gamma^{\IndCoh}$. Say 
$M \in \widehat{\fg}_{crit}\mod^{\heart}$
is \emph{regular} if $I \coloneqq \Ker(\fZ \to \fz)$
acts on $M$ trivially, and let 
$\widehat{\fg}_{crit}\mod_{reg}^{\heart} 
\subset \widehat{\fg}_{crit}\mod^{\heart}$ denote the corresponding
subcategory (which is not closed under extensions).
Then for any $\sF \in D_{crit}(\Gr_G)$,
the cohomology groups of $\Gamma^{\IndCoh}(\Gr_G,\sF) \in \widehat{\fg}_{crit}\mod$
will be regular, for the same reason as the analogous statement
in the finite-dimensional setting.

\subsection{}

In \cite{fg2}, Frenkel and Gaitsgory in effect proposed that these
are the only obstructions. We recall their conjecture now.

First, in \cite{dmod-aff-flag} \S 23, an appropriate
DG category $\widehat{\fg}_{crit}\mod_{reg}$
character was constructed:
we review the construction in \S \ref{s:central}.
There is a canonical action of the symmetric monoidal DG category
$\QCoh(\Op_{\ld{G}}^{reg})$
on $\widehat{\fg}_{crit}\mod_{reg}$ commuting with the critical level
$G(K)$-action.\footnote{There are actually important 
technical issues involving
this $G(K)$-action that should probably be 
overlooked at the level of an introduction;
we refer to \S \ref{ss:intro-eq-renorm} and \S \ref{ss:eq-renorm} 
for further discussion.}

Next, recall that geometric Satake \cite{mirkovic-vilonen}
gives an action of $\Rep(\ld{G}) = \QCoh(\bB \ld{G})$ on $D_{crit}(\Gr_G)$ by 
convolution. 

Moreover, $\Op_{\ld{G}}^{reg}$ has a canonical $\ld{G}$-bundle;
indeed, $\Op_{\ld{G}}^{reg}$ is the moduli of $\ld{G}$-local systems
on the formal disc $\sD = \Spec(k[[t]])$ with extra structure, giving a map:

\[
\Op_{\ld{G}}^{reg} \to \LocSys_{\ld{G}}(\sD) = \bB \ld{G}.
\]

\noindent In particular, there is a canonical symmetric monoidal
functor:

\[
\Rep(\ld{G}) \to \QCoh(\Op_{\ld{G}}^{reg}).
\] 

According to Beilinson-Drinfeld's \emph{birth of opers} theorem
from \cite{hitchin},
$\Gamma^{\IndCoh}$ is a canonically morphism of $(G(K),\Rep(\ld{G}))$-bimodule 
categories (c.f. \S \ref{s:localization}).

\begin{conjecture}[Frenkel-Gaitsgory, \cite{fg2} Main conjecture 8.5.2]\label{conj:intro-fg-loc}

The induced functor:

\[
\Gamma^{\Hecke}:
D_{crit}(\Gr_G) \underset{\Rep(\ld{G})}{\otimes}
\QCoh(\Op_{\ld{G}}^{reg}) \to 
\widehat{\fg}_{crit}\mod_{reg}
\]

\noindent is a $t$-exact equivalence of DG categories.

\end{conjecture}

We can now state:

\begin{mainthm}[Thm. \ref{t:main}]

Conjecture \ref{conj:intro-fg-loc} is true for $G$ of semisimple 
rank $1$.

\end{mainthm}

\begin{cor}\label{c:hecke-d-mod}

For $\chi \in \Op_{\ld{G}}^{reg}(k)$ a regular oper (defined over $k$), 
let $\widehat{\fg}_{crit}\mod_{\chi}^{\heart}$ denote the abelian
category of $\widehat{\fg}_{crit}$-modules on which
$\fZ$ acts through its quotient $\fZ \onto \fz \overset{\chi}{\onto} k$,
and let $\widehat{\fg}_{crit}\mod_{\chi}$ denote the appropriate
DG category. 

Then for $G = GL_2$, the functor:

\[
D_{crit}(\Gr_G) \underset{\Rep(\ld{G})}{\otimes} \Vect \to 
\widehat{\fg}_{crit}\mod_{\chi}
\]

\noindent induced by global sections 
is a $t$-exact equivalence, where $\Vect$ is a
$\Rep(\ld{G})$-module category via the map
$\Spec(k) \xar{\chi} \Op_{\ld{G}}^{reg} \to \bB \ld{G}$.

\end{cor}

\begin{cor}\label{c:ab-equiv}

Let $G = GL_2$ and let $\chi_1,\chi_2 \in \Op_{\ld{G}}^{reg}(k)$ be
two regular opers (defined over $k$). Then any isomorphism 
of the underlying $\ld{G}$-local systems of $\chi_1$ and $\chi_2$
gives rise to an equivalence of abelian categories:

\[
\widehat{\fg}_{crit}\mod_{\chi_1}^{\heart} \simeq 
\widehat{\fg}_{crit}\mod_{\chi_2}^{\heart}.
\]

\end{cor}

\begin{rem}

We highlight a wrong perspective on Corollary \ref{c:ab-equiv};
this remark may safely be skipped. 

For $G = PGL_2$, one can show that the group scheme
$\Aut$ of automorphisms of the formal disc acts transitively 
on $\Op_{\ld{G}}^{reg}$, giving rise to an elementary
construction of equivalences of categories
as in Corollary \ref{c:ab-equiv} in this case.

However, these are not the equivalences produced by 
Corollary \ref{c:ab-equiv}. First, at the level of
DG categories, the equivalences using the action of $\Aut$
are not $G(K)$-equivariant: the $G(K)$-actions differ
via the action of $\Aut$ on $G(K)$.
In contrast, 
the equivalences produced using 
Corollary \ref{c:hecke-d-mod} are 
manifestly $G(K)$-equivariant.

Concretely, this implies that for a $k$-point
$g\in G(K)$, if 
$g \cdot \bV_{crit} \coloneqq
\ind_{\Ad_g(\fg[[t]])}^{\widehat{\fg}_{crit}}(k)$ 
and $g \cdot \bV_{crit,\chi} \coloneqq 
(g \cdot \bV_{crit}) \otimes_{\fz,\chi} k$, then
Corollary \ref{c:ab-equiv} maps $g \cdot \bV_{crit,\chi_1}$ 
to $g \cdot \bV_{crit,\chi_2}$. For 
$\gamma \in \Aut$ and $\chi_2 = \gamma \cdot \chi_1$, the
resulting isomorphism produced \emph{using $\gamma$}
(not Corollary \ref{c:ab-equiv}) rather
sends $g \cdot \bV_{crit,\chi_1}$ to 
$\gamma(g) \cdot \bV_{crit,\chi_2}$.

In addition, one can see that the equivalences 
produced using the $\Aut$ action
depend on 
isomorphisms of underlying $\ld{G}_{\ld{B}}^{\wedge}$-bundles
of regular opers (in this $PGL_2$ case), 
not merely the underlying $\ld{G}$-bundles.

\end{rem}

\subsection{Viewpoints} 

We refer to the introduction and \S 2 of
\cite{fg-loc} for a discussion of Conjecture \ref{conj:intro-fg-loc} and
its consequences. We highlight some ways of thinking about it here.

\begin{itemize}

\item For the representation theorist, Theorem \ref{t:main} provides
an affine analogue of Beilinson-Bernstein similar to 
their original result, c.f. the discussion in \S \ref{ss:bb-loc}.
The equivalences of Corollary \ref{c:ab-equiv} provide analogues
of translation functors at critical level. By Theorem \ref{t:fg-summary},
the content of Theorem \ref{t:main} amounts to a structure theorem
for regular $\widehat{\fg}_{crit}$-modules (for $\fg = \sl_2$).

\item For the number theorist, Theorem \ref{t:main} provides
a first non-trivial test of Frenkel-Gaitsgory's proposal
\cite{fg2} for local geometric Langlands beyond Iwahori invariants.

Roughly, Frenkel-Gaitsgory propose that for $\sigma$ a 
de Rham $\ld{G}$-local system on the punctured disc, there
should be an associated DG category $\sC_{\sigma}$ with 
an action of $G(K)$.\footnote{Most invariantly, this action
should have critical level, which is (slightly non-canonically) equivalent
to level $0$.}
This construction should mirror the usual local Langlands correspondence,
leading to many expected properties of this assignment, c.f. 
\cite{quantum-langlands-summary}.

A striking part of their proposal does not have an arithmetic
counterpart. For $\chi$ an oper with underlying local system
$\sigma$, Frenkel-Gaitsgory propose $\sC_{\sigma} = 
\widehat{\fg}_{crit}\mod_{\chi} \in G(K)\mod_{crit}$, 
where we use similar notation to Corollary \ref{c:hecke-d-mod}.
We remark that Frenkel-Zhu \cite{frenkel-zhu} and Arinkin \cite{dima-opers}
have shown that any such $\sigma$ admits an oper structure
(assuming, to simplify the discussion, that $\sigma$ is a field-valued point). 

In particular, one expects equivalences as in Corollary \ref{c:ab-equiv},
at least on the corresponding derived categories. 

Our results provide the first verification 
beyond Iwahori invariants of their ideas. 

\end{itemize}

\begin{rem}

We have nothing to offer to the combinatorics of representations.
The previous results of Frenkel-Gaitsgory suffice\footnote{Up to mild
central character restrictions coming from \cite{dmod-aff-flag}.
These restrictions are understood among experts
to be inessential.} to treat problems of Kazhdan-Lusztig nature,
c.f. \cite{arakawa-fiebig}.

\end{rem}

\subsection{Previously known results}

Frenkel-Gaitsgory were able to show the following results, valid
for any reductive $G$.

We let $I = G(O) \times_G B$ be the Iwahori subgroup of
$G(O)$ and $\o{I} = G(O) \times_G N$ its prounipotent
radical.

\begin{thm}\label{t:fg-summary}

The functor $\Gamma^{\Hecke}$ is fully faithful, preserves
compact objects, and 
is an equivalence on $\o{I}$-equivariant categories.
Moreover, the restriction of $\Gamma^{\Hecke}$ to the
$\o{I}$-equivariant category 
$D_{crit}(\Gr_G)^{\o{I}} \underset{\Rep(\ld{G})}{\otimes}
\QCoh(\Op_{\ld{G}}^{reg})$ is $t$-exact.

\end{thm}

\begin{rem}

The fully faithfulness is \cite{fg2} Theorem 8.7.1; we
give a simpler proof of this result in Appendix \ref{a:ff}.
The existence of the continuous right adjoint $\Loc^{\Hecke}$
is proved as in \cite{dmod-aff-flag} \S 23.5-6.
The equivalence on $\o{I}$-equivariant
categories and $t$-exactness of the functor 
is Theorem 1.7 of \cite{fg-loc}.\footnote{The results
we cite here are not formulated in exactly the given form in
the cited works. For the purposes of the introduction,
we ignore this issue and address these gaps in the body of the paper.}

\end{rem}

\subsection{Methods}

Below, we outline the proof of Theorem \ref{t:main}. However,
to motivate this, we highlight a methodological point.

Across their works at critical level, Frenkel and Gaitsgory
use remarkably little about actual critical level representations.
Indeed, they rely primarily on Feigin and Frenkel's early results,
some basic properties of Wakimoto modules, and the Kac-Kazhdan 
theorem.

But using the action of $G(K)$ on $\widehat{\fg}_{crit}\mod$ and
constructions/results from geometric Langlands,
Frenkel and Gaitsgory were able to prove deep results about representations
at critical level; see e.g. \cite{dmod-aff-flag}. 

In other words, their works highlight an important
methodological point: the theory of group actions on categories
provides a bridge:

\[
\Bigg\{
\begin{tabular}{l}
Geometry and higher \\
representation
theory of \\
groups
\end{tabular}
\Bigg\} 
\begin{tikzpicture}
\node at (0,0) (v0) {};
\node at (3,0) (v1) {};
\draw [->,
line join=round,
decorate, decoration={
coil,
segment length=5,
amplitude=1,
post length=0pt
}]
    (v0)--(v1);
\node at (1.5,.75) 
{\footnotesize $\begin{tabular}{c}
Group actions \\ on categories
\end{tabular}$};
\end{tikzpicture}
\Bigg\{
\begin{tabular}{l}
Representation theory \\
of Lie algebras
\end{tabular}
\Bigg\} 
\] 

For loop groups in particular, a great deal
was known at the time about $G(O)$ and Iwahori 
invariants: see e.g. \cite{mirkovic-vilonen}, 
\cite{arkhipov-bezrukavnikov}, 
\cite{arkhipov-bezrukavnikov-ginzburg}, and \cite{abbgm}.

More recently, 
Whittaker invariants have been added to the list: see
\cite{whit}. These can be used to simplify many arguments
from Frenkel-Gaitsgory, as e.g. in Appendix \ref{a:ff}.

\subsection{}

As we outline below, our methods are in keeping with the
above. The main new idea and starting point of the present paper, 
Theorem \ref{t:generic},
is exactly about the higher representation theory of
$PGL_2(K)$.

\subsection{}

Group actions on categories inherently involve derived
categories. Therefore, one has the striking fact that
although Corollary \ref{c:ab-equiv} is about
abelian categories (of modules!), the proof we give
involves sophisticated homological methods and careful
analysis of objects in degree $-\infty$ in various
DG categories.

\subsection{Sketch of the proof}\label{ss:sketch}

We now give the Platonic ideal of the proof of the main theorem.

\subsection{}

First, one readily reduces to proving Conjecture \ref{conj:intro-fg-loc} 
for any $G$ of semisimple rank $1$; for us, it is convenient to
focus on $G = PGL_2$.

\subsection{}\label{ss:plato}

The following result is one of the key new ideas of this paper:

\begin{thm*}[Thm. \ref{t:generic}]

Let $G = PGL_2$ and let $\sC$ be acted on by $G(K)$ (perhaps 
with level $\kappa$).

Then $\sC$ is generated under the action of $G(K)$ by its Whittaker category
$\Whit(\sC) \coloneqq \sC^{N(K),\psi}$ 
and its $\o{I}$-equivariant category 
$\sC^{\o{I}}$.

\end{thm*}

The relation to the equivalence part of the 
Frenkel-Gaitsgory conjecture is immediate:
By fully faithfulness of $\Gamma^{\Hecke}$ (Theorem \ref{t:fg-summary}),
Theorem \ref{t:main} is reduced
to showing essential surjectivity.
Applying Theorem \ref{t:generic} to the essential image of
$\Gamma^{\Hecke}$, one immediately obtains Theorem \ref{t:main}
from Theorem \ref{t:fg-summary} and:

\begin{thm*}[Thm. \ref{t:whit-equiv}]

For any reductive $G$, the 
functor $\Gamma^{\Hecke}$ induces an equivalence on Whittaker 
categories.

\end{thm*}

\noindent The latter result is an essentially immediate 
consequence of the affine Skryabin theorem from \cite{whit}
and the classical work \cite{fgv}.

\subsection{}

Theorem \ref{t:generic} warrants some further discussion.

First, this result mirrors the fact that for 
$PGL_2$ over a local, non-Archmidean
field, irreducible representations admit Whittaker models, or 
else are one of the two 1-dimensional characters trivial
on the image of $SL_2$.

We now give an intentionally informal heuristic for Theorem \ref{t:generic}
that may safely be skipped.

For general reductive
$G$ and $\sC \in G(K)\mod_{crit}$, let $\sC^{\prime} \subset \sC$
be the subcategory 
generated under the $G(K)$-action 
by $\Whit(\sC)$. 

Assuming some form of local geometric Langlands, one expects
that the local Langlands parameters of $\sC/\sC^{\prime}$ to consist
only of those $\sigma \in \LocSys_{\ld{G}}(\o{\sD})$ that
lift to a point of $\LocSys_{\ld{P}}(\o{\sD})$
at which the map $\LocSys_{\ld{P}}(\o{\sD}) \to \LocSys_{\ld{G}}(\o{\sD})$
is singular; here $\ld{P}$ is some parabolic subgroup of $\ld{G}$
and $\o{\sD} = \Spec(k((t)))$ is the formal punctured disc.

For $\ld{G} = SL_2$, the only parabolic we need to consider is
the Borel $\ld{B}$. Then $\sigma \in \LocSys_{\ld{B}}$ is
the data of an extension: 

\[
0 \to (\sL,\nabla) \to (\sE,\nabla) \to 
(\sL^{\vee},\nabla) \to 0
\]

\noindent  
$(\sL,\nabla)$ a line bundle
with connection on the punctured disc (and $\sL^{\vee}$ equipped with the
dual connection to that of $\sL$). At such a point, the
cokernel of the map of tangent spaces induced
by $\LocSys_{\ld{B}}(\o{\sD}) \to 
\LocSys_{\ld{G}}(\o{\sD})$ is:

\[
H_{dR}^1(\o{\sD},(\ld{\fg}/\ld{\fb})_{\sigma}) = 
H_{dR}^1(\o{\sD},(\sL^{\vee},\nabla)^{\otimes 2}).
\]

\noindent This group will vanish unless $(\sL,\nabla)^{\otimes 2}$
is trivial, i.e., unless $(\sE,\nabla) \in \LocSys_{\ld{G}}(\o{\sD})$
or its quadratic twist has unipotent monodromy. 

It is expected that for $\sD \in G(K)\mod_{crit}$ with local Langlands
parameters having unipotent monodromy (resp. up to twist by a 1-dimensional
character) is generated under the $G(K)$-action 
by its Iwahori invariants (resp. 
its Iwahori invariants twisted by a suitable 
character of $I$ trivial on $\o{I}$).

This justifies that for $G = PGL_2$, one should expect
$\sC/\sC^{\prime}$ above to be generated by its $\o{I}$-invariants.
(And in fact, following the above reasoning, 
one can refine Theorem \ref{t:generic} to show 
that $\sC$ is even generated by $\Whit(\sC)$ and
its invariants with respect to the Iwahori subgroup of $SL_2(K)$.) 

\subsection{}\label{ss:intro-generic-pf}

The argument we provide for Theorem \ref{t:generic} is novel.
Its decategorified version gives a new proof of the corresponding
result in usual harmonic analysis.

We use the perspective of \cite{whit} on Whittaker categories,
which allows us to study the Whittaker construction
via (finite-dimensional!) algebraic groups. (We summarize
the most relevant parts of \cite{whit} in 
\S \ref{ss:ad-whit-review}.)

In \S \ref{s:heisenberg}, we introduce a new technique in the 
finite-dimensional settings suggested by \cite{whit}, which
we call \emph{Whittaker inflation}. In that context,
Theorem \ref{t:cons} 
shows that subcategories generated under 
group actions by Whittaker invariants
are large in a suitable sense. 
These ideas apply for a general reductive
group $G$, and have counterparts in the decategorified setting.

In \S \ref{s:adolescent}, we introduce a method of \emph{descent}
that is specific for $PGL_2$. Combined with the results of
\S \ref{s:heisenberg}, descent immediately gives Theorem 
\ref{t:generic}.

\subsection{}\label{ss:intro-eq-renorm}

For clarity, we highlight that there is one technical issue
in the above argument: there is not an a priori
$G(K)$-action on $\widehat{\fg}_{crit}\mod_{reg}$,
so the above argument does not apply as is.
Instead, there is a closely related but inequivalent category, 
$\widehat{\fg}_{crit}\mod_{reg,naive}$, with an evident $G(K)$-action
(coming from \cite{methods}).
We refer to \S \ref{ss:eq-renorm} for a more technical discussion
of this point.

This distinction makes the second half of the paper more technical,
and requires finer analysis than was suggested in \S \ref{ss:plato}.

\subsection{}

The $t$-exactness in Theorem \ref{t:main} is proved by another
instance of the descent argument highlighted above.
The details are in \S \ref{s:exactness}, 
with some auxiliary support in \S \ref{s:bddness}. 

\subsection{}

Finally, we highlight that the vast majority of the intermediate
results in this paper apply to general reductive groups $G$.
In particular, this includes the results of
\S \ref{s:heisenberg}, which are a key ingredient in the proof
of Theorem \ref{t:main}.

The descent arguments discussed above are where
we use that $G = PGL_2$; here the key input is that
every element in the Lie algebra $\fg = \sl_2$
is either regular or $0$. The situation
strongly suggests that there should be some (more complicated)
generalization of the descent method that applies 
for higher rank groups as well.

\subsection{Structure of this paper}

The first part of the paper is purely geometric, primarily
involving monoidal categories of $D$-modules on algebraic groups.

In \S \ref{s:heisenberg}, we introduce the \emph{inflation}
method discussed above. In \S \ref{s:conv}, we provide
some refinements of these ideas that are needed later
in the paper; this section includes some results
on Whittaker models for the finite-dimensional
group $G$ that are of independent interest.

In \S \ref{s:adolescent}, we prove our theorem on
the existence of Whittaker models for 
most categorical representations of $PGL_2(K)$ and introduce
the \emph{descent} argument discussed above.

\subsection{}

The second part of the paper applies the 
above material to critical level Kac-Moody representations.

In \S \ref{s:central}, we introduce the
DG category $\widehat{\fg}_{crit}\mod_{reg}$ following
\cite{dmod-aff-flag}. To study this DG category using
group actions, we import the main results from 
\cite{methods} here.

In \S \ref{s:localization}, we recall in detail
the key constructions from the formulation of
Conjecture \ref{conj:intro-fg-loc}. We formulate three
lemmas from which we deduce our main result, Theorem \ref{t:main}. 

The proofs of these lemmas occupy 
\S \ref{s:essential-surjectivity}-\ref{s:bddness}.
Roughly, \S \ref{s:essential-surjectivity} is devoted
to showing that the functor $\Gamma^{\Hecke}$ 
is essentially surjective,
while \S \ref{s:exactness} is devoted to showing that
it is $t$-exact. The final section, \S \ref{s:bddness},
provides additional technical support 
related to the distinction between
$\widehat{\fg}_{crit}\mod_{reg}$ and
$\widehat{\fg}_{crit}\mod_{reg,naive}$. 

Finally, \S \ref{s:iwahori-and-whittaker} collects
results on the behavior of $\Gamma^{\Hecke}$ on
Iwahori and Whittaker equivariant categories;
the former results are due to Frenkel-Gaitsgory \cite{fg-loc},
while the latter are original. 

\subsection{} There are two appendices. 

In Appendix \ref{a:gamma}, we compare our construction 
of the global sections functor to the more classical
one used by Kashiwara-Tanisaki, 
Beilinson-Drinfeld and Frenkel-Gaitsgory.

In Appendix \ref{a:ff}, we reprove the 
Frenkel-Gaitsgory theorem that $\Gamma^{\Hecke}$
is fully faithful.
 
\subsection{Acknowledgements}

We thank Dima Arinkin, Sasha Beilinson, Dario Beraldo, 
David Ben-Zvi, Roman Bezrukavnikov, Justin Campbell,
Vladimir Drinfeld, Gurbir Dhillon, Ivan Mirkovic, and David Yang 
for their encouragement and
for helpful conversations related to this material. 

We especially thank Dennis Gaitsgory for 
sharing many inspiring ideas on 
Kac-Moody algebras and loop group actions over the
years. In particular,
the crucial idea of using Heisenberg groups to
prove Theorem \ref{t:cons} was inspired by his work
\cite{lesjours} \S 2.

\section{Preliminary material}\label{s:notation}

\subsection{} 

In this section, we collect some notation
and constructions that will be used throughout the paper.

\subsection{}

As in \S \ref{ss:intro-notation}, we always work over a field
$k$ of characteristic $0$. 

\subsection{Reductive groups}

Throughout the paper, $G$ denotes a split reductive group,
$B$ denotes a fixed Borel with unipotent radical $N$
and Cartan $T = B/N$.

We let $\Lambda = \Hom(T,\bG_m)$ be the lattice of
weights of $T$ and $\Hom(\bG_m,T)$ the lattice of
coweights. Let $\rho \in \Lambda\otimes \bQ $ be the half sum of
positive roots and $\check{\rho} \in \check{\Lambda} \otimes \bQ$
be the half sum of positive coroots. We denote the pairing
between $\Lambda$ and $\check{\Lambda}$ by $(-,-)$.

We let e.g. $\fg$ denote the Lie algebra of $G$,
$\fb$ the Lie algebra of $B$, and so on.

We let $\ld{G}$ denote the Langlands dual group to $G$,
considered as an algebraic group over $k$. It naturally
comes equipped with a choice Borel $\ld{B}$ with 
radical $\ld{N}$ and Cartan $\ld{T} = \ld{B}/\ld{N}$.

\subsection{Higher categories}

Following standard conventions in the area, we freely
use Lurie's theory \cite{htt} \cite{higheralgebra} of
higher category theory. To simplify the terminology,
we use \emph{category} to mean \emph{$(\infty,1)$-category}.

\subsection{DG categories}

We let $\DGCat_{cont}$ denote the symmetric monoidal
category of presentable (in particular, cocomplete) 
DG categories, referring
to \cite{grbook} Chapter I for more details.
As in \emph{loc. cit}., the binary product underlying this
symmetric monoidal structure is denoted $\otimes$.
We recall that $\Vect \in \DGCat_{cont}$ is the
unit for this tensor product.

\subsection{}

For $\sA \in \Alg(\DGCat_{cont})$ an algebra
in this symmetric monoidal category, we typically write
$\sA\mod$ for $\sA\mod(\DGCat_{cont})$, i.e., 
the category of modules for $\sA$ in $\DGCat_{cont}$.

\subsection{}

For $\sC$ a DG category and $\sF,\sG \in \sC$, 
we use the notation $\ul{\Hom}_{\sC}(\sF,\sG)$ 
to denote the corresponding object of $\Vect$,
as distinguished from the 
corresponding $\infty$-groupoid 
$\Hom_{\sC}(\sF,\sG) = \Omega^{\infty} \ul{\Hom}_{\sC}(\sF,\sG)$.

\subsection{}

For $\sC,\sD$ objects of a 2-category (i.e., 
(meaning: $(\infty,2)$-category) $\mathsf{C}$,
we use the notation $\TwoHom_{\mathsf{C}}(\sC,\sD) \in \Cat$ 
to denote the corresponding category of maps.

When $\mathsf{C}$ is enriched over $\DGCat_{cont}$,
we use the same notation for the DG category of maps.
E.g., this applies for $\mathsf{C} = \DGCat_{cont}$ 
or $\mathsf{C} = \sA\mod$ for
$\sA$ as above.

\subsection{}

We use the notation $(-)^{\vee}$ to denote duals of
dualizable objects in symmetric monoidal categories.
In particular, for $\sC \in \DGCat_{cont}$ dualizable
in the sense of \cite{grbook}, we let 
$\sC^{\vee} \in \DGCat_{cont}$ denote
the corresponding dual category.

\subsection{}

For a DG category $\sC$ with $t$-structure, we use
cohomological notation: $\sC^{\leq 0}$ denotes the connective
objects and $\sC^{\geq 0}$ denotes the coconnective objects.
We let $\sC^{\heart} = \sC^{\leq 0} \cap \sC^{\geq 0}$ denote the
heart of the $t$-structure. 

\subsection{Classical objects}

Where we wish to say that an object lives
in some traditional $(1,1)$-category, we often refer to it
as \emph{classical}. So e.g., a \emph{classical vector
space} refers to an object of $\Vect^{\heart}$,
while a \emph{classical (ind)scheme} is being distinguished 
from a DG (ind)scheme.

\subsection{$D$-modules}

For an indscheme $S$ of ind-finite type, we let
$D(S)$ denote the DG category of $D$-modules on $S$ as
defined in \cite{grbook}. For a map $f:S \to T$, we let
$f^!$ and $f_{*,dR}$ denote the corresponding $D$-module functors.

We recall that for $S$ an indscheme of possibly infinite type,
there are two categories of $D$-modules, denoted
$D^*(S)$ and $D^!(S)$. We refer to \cite{dmod} for the definitions
in this setting.

\subsection{Group actions on categories}

We briefly recall some constructions from the
theory of group actions on categories.

\subsection{}

Suppose $H$ is a Tate group indscheme in the
sense of \cite{methods} \S 7, i.e., $H$ is
a group indscheme that admits a group subscheme 
$K \subset H$ such that $H/K$ is an indscheme of ind-finite type.

We recall from \cite{dmod} the category
$D^*(H)$ is canonically monoidal. 
By definition, we let $H\mod$ denote the category 
$D^*(H)\mod$ and refer to objects of this
category as categories with a \emph{strong $H$-action}.
We typically omit the adjective \emph{strong}; 
where we refer only to an $H$-action, we mean a strong $H$-action.

For $\sC \in H\mod$ and $\sF \in D^*(H)$, we let
$\sF \star -: \sC \to \sC$ denote the (\emph{convolution}) 
functor defined by the action.

\subsection{}

For $\sC \in H\mod$, we have the \emph{invariants
category} and \emph{coinvariants} categories:

\[
\sC^H \coloneqq \ul{\Hom}_{H\mod}(\Vect,\sC), \hspace{.5cm}
\sC_H \coloneqq \Vect \underset{D^*(H)}{\otimes} \sC.
\]

\noindent Here $\Vect$ is given the \emph{trivial} $H$-action.

We let $\Oblv:\sC^H \to \sC$ denote the forgetful functor.
Recall from \cite{beraldo-*/!} \S 2 and \S 4 
that if $H$ is a group scheme with 
prounipotent tail, then $\Oblv:\sC^H \to \sC$ admits
a continuous right adjoint $\Av_* = \Av_*^H$ that is
functorial in $\sC$. The composition
$\Oblv\Av_*:\sC \to \sC$ is given by convolution with the
constant $D$-module $k_H \in D^*(H)$. 

More generally, as in \cite{beraldo-*/!} \S 2.5.4, for
any character $\psi:H \to \bG_a$, we may form the
twisted invariants and coinvariants categories:

\[
\sC^{H,\psi}, \sC_{H,\psi}.
\] 

\noindent We use similar notation to the above,
though (for $H$ a group scheme) 
we often write $\Av_*^{H,\psi} = \Av_*^{\psi}$ to emphasize
the character.

\subsection{}

For $\sC$ with a right $H$-action and $\sD$ with a left
$H$-action, we let $\sC \overset{H}{\otimes} \sD$
denote the $H$-invariants for the induced diagonal
action on $\sC \otimes \sD$.

\subsection{}

Given a central extension
$\widehat{H}$ of $H$ by a torus $T$ and an element
$\lambda \in \ft^{\vee}$, we have a category
$H\mod_{\lambda}$ of categories acted on by
$H$ with \emph{level} $\lambda$, and such that
for $\lambda = 0$ we have $H\mod_0 = H\mod$.
We refer to \cite{methods} \S 11.3 and \cite{whit} \S 1.30
for definitions.

For $H = G(K)$ the loop group, ad-invariant
symmetric bilinear forms $\kappa:\fg \otimes \fg \to k$ 
define the above data, c.f. \emph{loc. cit}. In particular,
we obtain $G(K)\mod_{\kappa}$ for any $\kappa$.

In the presence of a level, we can form
invariants and coinvariants for group indschemes
$H^{\prime}$ equipped with a map $H^{\prime} \to H$
and a trivialization\footnote{One can do better: the
important thing is to have a specified action of $H^{\prime}$
on $\Vect$ with the given level.} of the corresponding central
extension of $H^{\prime}$.

For instance, for $H = G(K)$, this applies to 
$N(K)$ and $G(O)$, or any subgroup of either. 
Indeed, the Kac-Moody extension is canonically trivialized
over each of these subgroups.

Where the level is obviously implied, we sometimes allow ourselves
simply to refer to \emph{$H$-actions}, \emph{$H$-equivariant functors},
and so on.

\subsection{}

We recall from \cite{methods} that for $H$ as above,
there is a canonical category $\fh\mod$ of
modules for the Lie algebra $\fh$ of $H$
and a canonical action of $H$ on $\fh\mod$. 
We remind that if $H$ is not of finite type, the forgetful functor
$\fh\mod \to \Vect$ is not conservative.

One has similar reasoning in the presence of a level.
For instance, we have a canonical object
$\widehat{\fg}_{\kappa}\mod \in G(K)\mod_{\kappa}$.
We refer to \cite{methods} \S 11 for further discussion.

\subsection{}

We will sometimes reference the
theory of weak actions of Tate group indschemes.
We let $H\mod_{weak}$ denote the category of DG categories
with weak $H$-actions, defined as in \cite{methods} \S 7.
We use the notation $\sC \mapsto \sC^{H,w},\sC_{H,w}$ to
denote weak invariants and coinvariants functors.

\subsection{}

We will frequently reference compatibilities between
$t$-structures and group actions. We refer to \cite{whit} Appendix B
and \cite{methods} \S 10 for definitions and basic results.

\subsection{}

Finally, we end with informal remarks. 

The theory of loop group actions
on DG categories, especially weak actions,
is somewhat involved to set up, c.f. \cite{methods}.
With that said, as a black box, the theory is fairly intuitive
to use and provides quite useful insights. 

Therefore, we hope that the sometimes frequent references to 
\cite{methods} and the more formal parts (e.g., Appendix B)
of \cite{whit} will not cause the reader too much
indigestion. 

\section{Whittaker inflation}\label{s:heisenberg}

\subsection{}

The main result of this section is Theorem \ref{t:cons},
which is one of the key innovations of this paper.
For higher jet groups $G_n$ (see below) of a reductive group $G$, 
this result precisely measures how much information is
lost by the corresponding analogues of the Whittaker model.

The proof uses some constructions with Heisenberg group actions
on categories, which we recall here. This material
is a categorical version of the usual representation 
theory of Heisenberg groups over finite fields. Similar ideas
were used in \cite{lesjours}, though the
application was of different nature there.

\subsection{}\label{ss:gn-notation}

For $H$ an algebraic group and $n \geq 1$, we let
$H_n$ denote the algebraic group of maps from 
$\Spec(k[[t]]/t^n)$ to $H$. In particular, $H_1 = H$.

Let $\{e_i \in \fn\}_{i \in \cI_G}$ be Chevalley generators 
of $\fn$ indexed
by $\cI_G$ the set of simple roots. 
Let $\psi:N_n \to \bG_a$ be defined as the composition:

\[
N_n \to N_n/[N,N]_n = \prod_{i \in \cI_G} (\bG_a)_n \cdot e_i \xar{\on{sum}}
(\bG_a)_n = (\bG_a) \underset{k}{\otimes} k[[t]]/t^n \to \bG_a
\]

\noindent where the last map is induced by the functional:

\[
\begin{gathered}
k[[t]]/t^n \to k \\
\sum a_i t^i \mapsto a_{n-1}.
\end{gathered}
\]

For the remainder of this section, we assume that 
$n$ is at least $2$. 
The main result of this section answers the question:
for $\sC \in G_n\mod$, how much information do the invariants
$\sC^{N_n,\psi}$ remember about $\sC$? 

\subsection{}\label{ss:fourier-decomp}

As $n \geq 2$, we have a homomorphism:

\begin{equation}\label{eq:gn-lower-central}
\begin{gathered}
\fg \otimes \bG_a \to G_n \\
(\xi \in \fg) \mapsto \exp(t^{n-1} \xi). 
\end{gathered}
\end{equation}

\noindent This map realizes
$\fg \otimes \bG_a$ as a normal subgroup of $G_n$.
Note that the adjoint action of $G_n$ on this
normal subgroup is given by:

\[
G_n \xar{\on{ev}} G \overset{\text{adjoint}}{\actson} \fg.
\]

If $\sC$ is acted on by $G_n$, it is thus acted on by
$\fg \otimes \bG_a$ by restriction, or using
Fourier transform, by $D(\fg^{\vee})$ equipped
with the $\overset{!}{\otimes}$-tensor product.
(We omit the tensoring with $\bG_a$ because we are not
concerned with the additive structure on $\fg^{\vee}$ here.)

Fix a symmetric, linear
$G$-equivariant identification $\kappa:\fg \simeq \fg^{\vee}$
for the remainder of this section. Therefore, $\sC$ is acted
on by $D(\fg)$ with its $\overset{!}{\otimes}$-monoidal structure.
In particular, for $S$ a scheme mapping to $\fg$, we may
form $\sC|_S \coloneqq \sC \otimes_{D(\fg)} D(S)$.

Define $\sC_{reg}$ as $\sC|_{\fg_{reg}}$
where $\fg_{reg} \subset \fg$ is the subset of regular elements.
We have adjoint functors:

\[
j^!:\sC \rightleftarrows \sC_{reg}:j_{*,dR}
\]

\noindent with the right adjoint $j_{*,dR}$ being
fully faithful: indeed, these properties are inherited from
the corresponding situation $j^!:D(\fg) \rightleftarrows D(\fg_{reg}):j_{*,dR}$
for $j:\fg_{reg} \into \fg$ the embedding.

Because $\fg_{reg} \subset \fg$ is closed under the
adjoint action of $G$, and since $G_n$ acts on 
$\fg^{\vee} \simeq \fg$ through the adjoint action of $G$,
it follows that $\sC_{reg}$ is acted on by $G_n$
so that the comparison functors with $\sC$ are $G_n$-equivariant.

\subsection{Main theorem}

We have $\fg \otimes \bG_a \cap N_n = \fn \otimes \bG_a$,
and under the Fourier transform picture above,
we have:

\[
\sC^{\fn \otimes \bG_a,\psi|_{\fn \otimes \bG_a}} \simeq 
\sC|_{f+\fb}.
\]

\noindent Here $f$ a principal nilpotent
whose image in $\fg/\fb \simeq \fn^{\vee}$ is 
$\psi|_{\fn \otimes \bG_a}$.

In particular, because $f+\fb \subset \fg_{reg}$, it follows
that $\sC^{N_n,\psi} \simeq \sC_{reg}^{N_n,\psi}$.
The following result states that this is the only
loss in $(N_n,\psi)$-invariants.

\begin{thm}\label{t:cons}

The functor:

\[
G_n\mod_{reg} \xar{\sC \mapsto \sC^{N_n,\psi}} \DGCat_{cont}
\]

\noindent is conservative, where $G_n\mod_{reg} \subset G\mod$
is the full subcategory consisting of $\sC$ with $\sC_{reg} = \sC$.

\end{thm}

Here are some consequences.

\begin{cor}\label{c:hecke}

For every $\sC \in G_n\mod$, the convolution functor:

\[
D(G_n)^{N_n,-\psi} \underset{\sH_{N_n,\psi}}{\otimes} \sC^{N_n,
\psi} \to \sC
\]

\noindent is fully faithful with essential image $\sC_{reg}$.
Here $\sH_{N_n,\psi} = D(G_n)^{N_n \times N_n,(\psi,-\psi)}$
is the appropriate Hecke category for the pair $(G_n,(N_n,\psi))$.

\end{cor}

\begin{proof}

Note that this functor is $G_n$-equivariant and that its essential
image factors through $G_n$ (by the above analysis).
Therefore, by Theorem \ref{t:cons}, it suffices to show that
it is an equivalence on $(N_n,\psi)$-invariants, which is clear.

\end{proof}

\begin{cor}

Observe that $D(G_n)_{reg}$ admits a unique monoidal structure
such that the localization functor 
$D(G_n) \to D(G_n)_{reg}$ is monoidal.

Then $D(G_n)_{reg}$ and $\sH_{N_n,\psi}$ (as defined in the previous corollary)
are Morita equivalent, with bimodule $D(G_n)^{N_n,\psi}$ defining
this equivalence.

\end{cor}

The remainder of this section is devoted to the proof of 
Theorem \ref{t:cons}.

\subsection{Example: $n = 2$ case}\label{ss:n=2}

First, we prove Theorem \ref{t:cons} in the $n = 2$ case.
This case is simpler than the general case, and contains one of
the main ideas in the proof of the general case.

Note that by Fourier transform along $\fg \otimes \bG_a \subset G_2$,
an action of $G_2$ on $\sC$ is equivalent to the datum of 
$G$ on $\sC$, and an action of $(D(\fg),\overset{!}{\otimes})$
on $\sC$ as an object of $G\mod$ (where $G$ acts on $D(\fg)$
by the adjoint action). In the sheaf of categories language
\cite{shvcat}, we obtain:

\[
G_2\mod \simeq \ShvCat_{/(\fg/G)_{dR}}.
\]

The functor of $(N_2,\psi)$-invariants then corresponds
to global sections of the sheaf of categories
over $(f+\fb/N)_{dR}$, i.e., the de Rham space of the Kostant slice.
Recall that the Kostant slice $f+\fb/N$ is an affine scheme
and maps smoothly to $\fg/G$ with image $\fg_{reg}/G$.

As the Kostant slice is a scheme (not a stack), 
\cite{shvcat} Theorem 2.6.3 implies that
$(f+\fb/N)_{dR}$ is 1-affine. In particular, its global
sections functor is conservative.

Therefore, it suffices to note that pullback of sheaves of categories
along the map $(f+\fb/N)_{dR} \to (\fg_{reg}/G)_{dR}$ is conservative.
However, in the diagram:

\[
\xymatrix{
f+\fb/N \ar[r] \ar[d] & \fg_{reg}/G \ar[d] \\
(f+\fb/N)_{dR} \ar[r] & (\fg_{reg}/G)_{dR}
}
\]

\noindent pullback for sheaves of categories along the vertical
maps is conservative for formal reasons (e.g., write
de Rham as the quotient by the infinitesimal groupoid),
and conservativeness of pullback along the upper arrow follows from 
descent of sheaves of categories 
along smooth (or more generally fppf) covers, 
c.f. \cite{shvcat} Theorem 1.5.2.
This implies that pullback along 
the bottom arrow is conservative as well.

\begin{rem}

It follows from the above analysis that 
the Hecke algebra $\sH_2$ (in the notation of Corollary \ref{c:hecke})
is equivalent to $D$-modules on the group scheme
of regular centralizers.

\end{rem}

\subsection{Heisenberg groups}\label{ss:heis-defin}

We will deduce the general case of 
Theorem \ref{t:cons} from the representation 
theory of Heisenberg groups, which we digress to discuss now.

Let $V$ be a finite-dimensional vector space. In the following
discussion, we do not distinguish
between $V$ and the additive group scheme $V \otimes_k \bG_a$.

Let $H = H(V)$ denote the corresponding \emph{Heisenberg group};
by definition, $H$ is the semidirect product:

\[
V \ltimes (V^{\vee} \times \bG_a)
\]

\noindent where $V$ acts on $V^{\vee} \times \bG_a$ via:

\[
v \cdot (\lambda,c) = (\lambda,c+\lambda(v)), 
\hspace{.5cm} (v,\lambda,c) \in V \times V^{\vee} \times \bG_a.
\]

\begin{rem}\label{r:heis-sympl}

Note that 
$H$ only depends on the symplectic vector space $W = V \times V^{\vee}$, not on
the choice of polarization $V \subset W$. But the above presentation
is convenient for our purposes.

\end{rem}

\subsection{}\label{ss:heis-reg-defin}

Observe that $\bG_a \subset H$ is central.
In particular, $D(\bA^1) \overset{\on{Fourier}}{\simeq} D(\bG_a)$
maps centrally to $H$, where we use $D(\bA^1)$ to indicate
that we consider the $\overset{!}{\otimes}$-monoidal structure
and $D(\bG_a)$ to indicate the convolution monoidal structure.

Let $H\mod_{reg} \subset H\mod$ denote the subcategory
where $D(\bA^1)$ acts through its localization
$D(\bA^1\setminus 0)$, i.e., where all Fourier coefficients
are non-zero. 

\begin{thm}\label{t:heisenberg}

The functor:

\[
H\mod_{reg} \xar{\sC \mapsto \sC^V} D(\bA^1\setminus 0)\mod
\]

\noindent is an equivalence.

\end{thm} 

\begin{cor}\label{c:heisenberg}

The functor of $V$-invariants is conservative on $H\mod_{reg}$.

\end{cor}

\begin{proof}[Proof of Theorem \ref{t:heisenberg}]

Note that by duality, $V$ acts on $V \times \bA^1$; 
explicitly, this is given by the formula:

\[
v \cdot (w,c) \coloneqq (w-c\cdot v,c)
\]

By Fourier transform along $V^{\vee} \times \bG_a \subset H$,
we see that an $H$-action on $\sC$ is equivalent
to giving a $V$-action on $\sC$ (where $V$ is given its natural
additive structure), and
an additional
$(D(V \times \bA^1),\overset{!}{\otimes})$-action on $\sC$ in the
category $V\mod$. 

Using the sheaf of categories
language \cite{shvcat},
this is equivalent to the data
of a sheaf of categories on $(V_{dR} \times \bA_{dR}^1)/V_{dR}$,
where we are quotienting using the above action.
The corresponding object of $H\mod$ lies in $H\mod_{reg}$
if and only if the sheaf of categories is pushed forward
from:

\[
(V_{dR} \times \bA_{dR}^1\setminus 0)/V_{dR} = \bA_{dR}^1\setminus 0.
\]

Therefore, we obtain an equivalence of the above type. Geometrically,
this equivalence is given by taking global sections of a
sheaf of categories, which for 
$(V_{dR} \times \bA_{dR}^1)/V_{dR}$
corresponds to taking (strong) $V$-invariants for the corresponding
$H$-module category.

\end{proof}

\subsection{Proof of Theorem \ref{t:cons}}

We now return to the setting of Theorem \ref{t:cons}.
The remainder of this section is devoted to
the proof of this result.

In what follows, for 
$\fh$ a nilpotent Lie algebra, we let $\exp(\fh)$ denote
the corresponding unipotent algebraic group.

Let $N_n^m = \exp(t^{n-m}\fn[[t]]/t^n\fn[[t]]) \subset N_n$
for $1 \leq m \leq n$. For example, for $m = 1$ we recover
the group $\fn \otimes \bG_a \subset G_n$.

We will show by induction on $m$ that the functor
of $(N_n^m,\psi)$-invariants is conservative on $G_n\mod_{reg}$.

\subsection{}

As a base case, we first show the claim for $m = 1$. 

Here the assertion follows 
by the argument of \S \ref{ss:n=2}.
Indeed, we have a homomorphism $G_2 \to G_n$
which identifies $G \subset G_2,G_n$ and
$\fg \otimes \bG_a \subset G_2,G_n$.
Restricting along this homomorphism, we obtain
that $\psi$-invariants for 
$N \cdot N_n^1 \subset N_n$ is conservative,
and a fortiori, $(N_n^1,\psi)$-invariants is as well.

\subsection{}\label{ss:n/2}

We\footnote{The arguments in \S \ref{ss:n/2} and 
\S \ref{ss:2m+1} are not needed in the case $\fg = \sl_2$,
which is what we use for our application to the
localization theorem. Indeed, for $\fg = \sl_2$, 
in the argument in \S \ref{ss:cons-pf}, one only
needs to consider (in the notation of \emph{loc. cit}.) 
$r = 1$, in which case $\fg_{1-r} = \ft$ is abelian,
hence the last equation in \eqref{eq:h0} holds for 
trivial reasons. Given that equality, the rest of
the argument goes through for $m \geq 2$.

In other words, the reader who is only interested in 
Theorem \ref{t:main} can safely skip \S \ref{ss:n/2} 
and \S \ref{ss:2m+1}.}
now observe that the above argument extends
to treat any $m \leq \frac{n}{2}$.

In this case, the subalgebra 
$t^{n-m} \fg[[t]]/t^n\fg[[t]] \subset 
\fg[[t]]/t^n\fg[[t]] = \Lie(G_n)$ is abelian.
Clearly this subalgebra is normal; the adjoint
action of $G_n$ on it is given via the representation:

\[
G_n \to G_m \actson \Lie(G_m) = \fg[[t]]/t^m \fg[[t]]
\overset{t^m \cdot -}{\simeq} 
t^{n-m} \fg[[t]]/t^n\fg[[t]].
\]

We therefore have a homomorphism:

\[
G_n \ltimes t^{n-m} \fg[[t]]/t^n\fg[[t]] \otimes \bG_a
\to G_n
\]

\noindent whose restriction to $G_n$
is the identity and whose restriction
to $t^{n-m} \fg[[t]]/t^n\fg[[t]] \otimes \bG_a$
is the exponential of the embedding
$t^{n-m} \fg[[t]]/t^n\fg[[t]] \into \fg[[t]]/t^n\fg[[t]]$.

Considering $\sC$ as a category acted on by
$G_n \ltimes t^{n-m} \fg[[t]]/t^n\fg[[t]]$ via the
above map and Fourier transforming as in
Example \ref{ss:n=2}, we can view this action 
as the data of making $\sC$ into a sheaf of categories on
$(\fg[[t]]/t^m\fg[[t]])_{dR}/G_{n,dR}$. 
Here we have identified the dual of
$(t^{n-m} \fg[[t]]/t^n\fg[[t]])$ with
$\fg[[t]]/t^m\fg[[t]]$ via the pairing
$(\xi_1,\xi_2) \mapsto \Res(t^{-n} \kappa(\xi_1,\xi_2) dt)$
(for $\kappa$ as above).

Define: 

\[
(\fg[[t]]/t^m\fg[[t]])_{reg} = 
\fg[[t]]/t^m\fg[[t]] \times_{\fg} \fg_{reg}.
\]

\noindent By the regularity assumption on $\sC$, the 
above sheaf of categories is pushed-forward
from:

\[
(\fg[[t]]/t^m\fg[[t]])_{reg,dR}/G_{n,dR}.
\]

Then $(N_n^m,\psi)$-invariants 
correspond to global sections of
$(f+\fb[[t]]/t^m\fb[[t]])_{dR}$ with 
coefficients in the above sheaf of categories.
As the map:

\[
(f+\fb[[t]]/t^m\fb[[t]]) \to 
(\fg[[t]]/t^m\fg[[t]])_{reg}/G_m
\]

\noindent is a smooth cover (as it is obtained by applying
jets to a smooth cover), the same is true of:

\[
(f+\fb[[t]]/t^m\fb[[t]]) \to 
(\fg[[t]]/t^m\fg[[t]])_{reg}/G_n.
\]

\noindent As
$f+\fb[[t]]/t^m\fb[[t]]$ is a scheme, the reasoning
of \S \ref{ss:n=2} gives us the desired result.

\subsection{}

In \S \ref{ss:2m+1}, we will give a separate
argument to treat the case $n = 2m-1$; of course, this
is only possible for $n$ odd. The argument
is not complicated, but a little involved to set up, so 
we postpone the argument for the moment.

Combined with \S \ref{ss:n/2}, this gives the result
for all $m \leq \frac{n+1}{2}$.

\subsection{}

We now perform the induction; 
we assume the conservativeness for $m - 1$ and 
show it for our given $m \leq n$. By the inductive
hypothesis as established above (though postponed in one
case to \S \ref{ss:2m+1}), we may assume 
$m \geq \frac{n+2}{2}$.

We will give the argument here by another inductive argument. 
As above, let $\fg = \oplus_s \fg_s$
be the principal grading defined by the coweight
$\check{\rho}:\bG_m \to G^{ad}$ of the adjoint
group $G^{ad}$ of $G$. So for example, $e_i \in \fg_1$
and $\fn = \oplus_{s \geq 1} \fg_s$. For $r \geq 1$, let
$\fn_{\geq r} \coloneqq \oplus_{s \geq r} \fg_s$.

Now define: 

\[
N_n^{m,r} \coloneqq 
\exp\big(t^{n-m+1}\fn[[t]]+t^{n-m} \fn_{\geq r}[[t]]/t^n\fn[[t]]\big) \subset N_n^m \subset N_n.
\]

\noindent We will show by descending induction on $r \geq 1$ that 
$(N_n^{m,r},\psi)$-invariants is conservative. Note that this
result is clear from our hypothesis on $m$ 
for $r \gg 0$, since then $\fn_{\geq r} = 0$ and $N_n^{m,r} = N_n^{m-1}$.
Moreover, a proof for all $r$ implies the next step in the induction
with respect to $m$, since $N_n^{m,1} = N_n^m$, so would complete
the proof of Theorem \ref{t:cons}.

\subsection{}\label{ss:cons-pf}

For $r \geq 1$, assume the conservativeness (in the regular setting) of
$(N_n^{m,r+1},\psi)$-invariants; we will deduce it for
$N_n^{m,r}$. The idea is to make a Heisenberg group act
on $(N_n^{m,r+1},\psi)$-invariants so that invariants with respect
to a Lagrangian gives $(N_n^{m,r},\psi)$-invariants.

\setcounter{steps}{0}
\step 

Define $\fh_0 \subset \Lie(G_n) = \fg[[t]]/t^n\fg[[t]]$ as:

\[
t^{m-1} \fg_{1-r} \oplus \Lie(N_n^{m,r}).
\] 

Observe that $\fh_0$ is a Lie subalgebra. 
Indeed: 

\begin{equation}\label{eq:h0}
\begin{gathered} 
[t^{m-1} \fg_{1-r},t^{n-m+1} \fg[[t]]] \subset t^n\fg[[t]],  \\
[t^{m-1} \fg_{1-r},t^{n-m}\fn_{\geq r}]\subset t^{n-1} \fn,
\text{ and } \\
[t^{m-1} \fg_{1-r},t^{m-1} \fg_{1-r}] \subset 
t^{2m-2} \fg[[t]] \subset t^n\fg[[t]]
\end{gathered}
\end{equation}

\noindent where the last embedding uses 
the assumption $m\geq \frac{n+2}{2}$.

In the same way, we see that\footnote{The same is true for
$r$ instead of $r+1$, but the statement with the character is not.}
 $\Lie(N_n^{m,r+1})$
is a normal Lie subalgebra of $\fh_0$, and that
for $\xi \in \fh_0$ and $\vph \in \Lie(N_n^{m,r+1})$,
$\psi([\xi,\vph]) = 0$.

Moreover, $\fh_0$ is nilpotent, so exponentiates
to a group $H_0 \subset G_n$.\footnote{The embedding exponentiates 
because $\fh_0 \subset \fn + t \fg[[t]]/t^n\fg[[t]]$, i.e.,
the Lie algebra of a unipotent subgroup of $G_n$. Here we use
that $m \geq 2$.}
Combining this with the above, we see that 
$H_0$ acts on $(N_n^{m,r+1},\psi)$-invariants for any category
with an action of $G_n$.

\step

Let $\fg_{1-r}^{\prime} \subset \fg_{1-r}$ denote
$\Ad_f^{2r-1}(\fg_r)$.
Observe that the pairing:

\begin{equation}\label{eq:psi-pairing}
\psi([-,-]):\fg_r \otimes \fg_{1-r} \to k
\end{equation} 

\noindent induces a perfect pairing between
$\fg_r$ and $\fg_{1-r}^{\prime}$. Indeed,
the diagram:

\[
\xymatrix{
\fg_r \otimes \fg_r \ar[r]^{\id \otimes \Ad_f^{2r-1}} & 
\fg_r \otimes \fg_{1-r} \ar[r]^{\id \otimes \Ad_f} \ar[dr]^{\psi([-,-])} &
\fg_r \otimes \fg_{-r} \ar[d]^{-\kappa(-,-)} \\
& & k
}
\]

\noindent commutes\footnote{Proof: write $\psi(-)$ as $\kappa(f,-)$
and use $\Ad$-invariance of $\kappa$.}, and
$\Ad_f^{2r}:\fg_r \to \fg_{-r}$ is an isomorphism by
$\sl_2$-representation theory.

Define $\fh_0^{\prime} \subset \fh_0$ as:

\[
t^{m-1} \fg_{1-r}^{\prime} \oplus \Lie(N_n^{m,r}).
\]

Again, $\fh_0^{\prime}$ integrates to a group $H_0^{\prime}$.

\step 

Finally, recall that the adjoint action 
of $H_0$ fixes $N_n^{m,r+1} \subset H_0$ and preserves
its character $\psi$ to $\bG_a$. Let $K \subset N_n^{m,r+1}$ be the
kernel of $\psi$: clearly $K$ is normal in $H_0$. 

One immediately observes that $H \coloneqq H_0^{\prime}/K$ is a Heisenberg
group. The central $\bG_a$ is induced by the map:

\[
\bG_a = N_n^{m,r+1}/K \to H_0^{\prime}/K = H.
\]

\noindent The vector space defining the Heisenberg group
is $t^{n-m} \fg_r$, and its dual is embedded as
$t^{m-1}\fg_{1-r}^{\prime} = H_0^{\prime}/K$.

Now observe that our Heisenberg group $H$ acts on
$\sC^{N_n^{m,r+1},\psi}$ for any $\sC$ acted on by $G_n$,
with its central $\bG_a$ acting through the exponential character.
Now the result follows from Corollary \ref{c:heisenberg}.

\subsection{}\label{ss:2m+1}

As above, it remains to show the
result in the special case that
$n = 2m-1$ for some $m \geq 2$.
We do so below.

\setcounter{steps}{0}
\step 

We need some auxiliary constructions.

Let $\xi \in \fg_{reg}$ be a $k$-point (i.e., a regular
element of $\fg$ in the usual sense).
Let $\fz_{\xi} \subset \fg$ denote the centralizer
of $\xi$. 

Then $\fg/\fz_{\xi}$ carries an alternating form:

\[
(\vph_1,\vph_2)_{\xi} \coloneqq \kappa(\xi,[\vph_1,\vph_2]) = 
\kappa([\xi,\vph_1],\vph_2).
\]

\noindent The second equality holds as $\kappa$ is 
$G$-invariant, and shows that $(-,-)_{\xi}$ descends
to $\fg/\fz_{\xi}$. Moreover, as $\kappa$ is non-degenerate,
we see from the last expression 
that $(-,-)_{\xi}$ is non-degenerate on 
$\fg/\fz_{\xi}$, hence symplectic.

\step In the above setting, suppose that 
$\xi$ lies in the Kostant slice $f+\fb$.

In this case, we claim that the composition 
$\fn \into \fg \onto \fg/\fz_{\xi}$ is injective,
and that $\fn \subset \fg/\fz_{\xi}$ is Lagrangian
with respect to the symplectic form $(-,-)_{\xi}$.

Indeed, it is standard that $\fz_{\xi} \cap \fn = 0$
(this is the infinitesimal version of the freeness
of the action of $N$ on $f+\fb$), giving the
injectivity. 

We now claim that $\fn$ is isotropic for the above form.
For $\vph_1,\vph_2 \in \fn$, we have:

\[
(\vph_1,\vph_2)_{\xi} = 
\kappa(\xi,[\vph_1,\vph_2])
\]

\noindent by definition; we claim this inner
product is zero. Let $\fg = \oplus_s \fg_s$ 
be the \emph{principal grading} of $\fg$, i.e., the grading
defined by the coweight
$\check{\rho}:\bG_m \to G^{ad}$.
Then $[\vph_1,\vph_2] \in 
[\fn,\fn] = \oplus_{s \geq 2} \fg_s$, while
$\xi \in f+\fb \subset \oplus_{s \geq -1} \fg_s$.
By invariance of $\kappa$, for 
$\widetilde{\xi} \in \fg_s,\widetilde{\vph} \in \fg_r$,
we have $\kappa(\widetilde{\xi},\widetilde{\vph}) = 0$
unless $r+s = 0$, giving the claim.

Finally, $2\dim(\fn)+\dim(\fz_{\xi}) = 
\dim(\fn)+\dim(\fn^-)+\dim(\ft) = \dim(\fg)$, so
$\fn \subset \fg/\fz_{\xi}$ is in fact Lagrangian.

\step\label{st:reg-symplectic-bundle}

Next, we observe that the above generalizes to the 
scheme-theoretic situation in which we allow
$\xi$ to vary.

More precisely, let $\widetilde{\sW} 
= \fg \otimes \sO_{\fg_{reg}}$
be the constant vector bundle on $\fg_{reg}$
with fiber $\fg$. This bundle carries
a subbundle $\fz \subset \widetilde{W}$ 
of regular centralizers; e.g., the fiber
of $\fz$ at $\xi \in \fg(k)$ is $\fz_{\xi}$. 

The quotient:

\[
\sW \coloneqq 
\widetilde{\sW}/\fz
\]

\noindent is a
vector bundle on $\fg_{reg}$. Our earlier
construction defines a symplectic form on
$\sW$. Moreover, after pulling back along 
the embedding $i:f+\fb \into \fg_{reg}$, 
the constant bundle with fiber $\fn$ defines
a Lagrangian subbundle of the vector bundle $i^*(\sW)$.

\step 

\newcounter{steps-aux}
\setcounter{steps-aux}{\value{steps}}

We now record some general results in the above setting.

Let $S$ be a scheme of finite type and let 
$\sW$ be a symplectic vector bundle on $S$.
We denote the total space of $\sW$ by the same
notation.

Define the \emph{Heisenberg group scheme} 
$\sH = \sH(\sW)$ over $S$ as the extension:

\[
0 \to \bG_{a,S} \to \sH \to \sW \to 0
\]

\noindent where $\sH = \sW \times_S \bG_{a,S}$ as 
a scheme, and the group law is given by the formula:

\[
(w_1,\lambda_1) \cdot (w_2,\lambda_2) =
(w_1+w_2,\lambda_1+\lambda_2+\frac{1}{2}(w_1,w_2)),
\hspace{.5cm} (w_i,\lambda_i) \in \sW \times_S \bG_{a,S}
\]

\noindent where the term $(w_1,w_2)$
denotes the symplectic pairing.

\begin{example}

For example, if $S = \Spec(k)$ and 
$\sW = W = V \times V^{\vee}$ with the evident symplectic
form, then the above recovers the Heisenberg group
denoted $H(V)$ earlier. 

\end{example}

In the general setting above, let $\bB_S \sH = S/\sH$ denote
classifying space of the group scheme $\sH$.
By a (strong) action of $\sH(\sW)$ on a category, we mean
a sheaf of categories on $(\bB_S \sH)_{dR}$; 
by 1-affineness of $S_{dR}$ and of the morphism
$\sH_{dR} \to S_{dR}$ (\cite{shvcat} Theorem 2.6.3),
this data is equivalent to that of a module category for
$D(\sH) \in \Alg(D(S)\mod)$ 
with its natural convolution monoidal structure.
We denote the corresponding 2-category by $\sH\mod$.

As when working over a point, we have a
subcategory $\sH\mod_{reg} \subset \sH\mod$:
Fourier transform for the central 
$\bG_{a,S} \subset \sH$ makes any object of
$\sH\mod$ into a 
$(D(S \times \bA^1),\overset{!}{\otimes})$ module category,
and we ask that this action factors
through $D(S \times (\bA^1 \setminus 0))$.

\begin{lem}\label{l:heis-bundle}

Suppose $\sN \subset \sW$ is a Lagrangian subbundle.
Then the functor of strong $\sN$-invariants defines
an equivalence:

\[
\sH\mod_{reg} \isom D(S \times (\bA^1\setminus 0))\mod.
\]

\end{lem}

\begin{proof}

In the case where $\sW$ admits a Lagrangian splitting
$\sW = \sN \times \sN^{\vee}$, the same argument 
as over a point applies. 

\'Etale locally, such a splitting exists: indeed,
\'etale locally, $\sW$ admits Darboux coordinates
(as a torsor for a smooth group scheme is \'etale locally
trivial), and then by the Bruhat decomposition for
the Lagrangian Grassmannian, $\sN$ admits a
complement after a further Zariski localization.

Therefore, we obtain the result by \'etale descent for
sheaves of categories on $S_{dR}$, 
see \cite{shvcat} Corollary 1.5.4.

\end{proof}

We also need a mild extension of the above.

Suppose we are given a vector bundle 
$\widetilde{\sW}$ on $S$ equipped with 
an epimorphism $\pi:\widetilde{\sW} \onto \sW$.
We form the group scheme $\widetilde{\sH} \coloneqq
\sH \times_{\sW} \widetilde{\sW}$, i.e.,
the pullback of the extension $\sH$ of $\sW$ to 
$\widetilde{\sW}$. 
We can again speak of (strong) $\widetilde{\sH}$-actions;
we define regularity as for $\sH$, i.e., with
respect to the central $\bG_a$.

\begin{lem}\label{l:heis-bundle-2}

Suppose $\sN \subset \sW$ is a Lagrangian subbundle,
and suppose we are given a lift $\sN \into \widetilde{\sW}$
of this embedding over $\pi$. In particular,
we obtain an embedding of the additive group scheme
$\sN$ into $\widetilde{\sH}$.

Then the functor of (strong) $\sN$-invariants is
conservative on $\widetilde{\sH}\mod_{reg}$.

\end{lem}

\begin{proof}

As in the proof of Lemma \ref{l:heis-bundle}, 
we are reduced by Zariski descent to the case where
$S$ is affine.

In this case, the embedding
$\sN \into \widetilde{\sW}$ extends
to a map $\sW \to \widetilde{\sW}$ splitting
the projection (because $\sW/\sN$ is a vector bundle).
This gives a map
$\sH \to \widetilde{\sH}$ splitting the canonical projection
that is the identity on the 
centrally embedded $\bG_{a,S}$, and which is compatible
with embeddings from $\sN$. Therefore, the result
in this case follows from Lemma \ref{l:heis-bundle}.

\end{proof}

We remark that $\widetilde{\sW}$ inherits an 
alternating form from $\sW$, and $\widetilde{\sH}$
may be interpreted as a degenerate version of a Heisenberg
group scheme. 

\setcounter{steps}{\value{steps-aux}}
\step 

We can now conclude the argument. We remind that
we have assumed $n = 2m-1$ for some $m \geq 2$.

We have the following extension of Lie algebras,
which is between abelian Lie algebras:

\[
\xymatrix{
0 \ar[r] &  t^{n-m+1} \fg[[t]]/t^n\fg[[t]] 
\ar@{=}[d] \ar[r] &
t^{n-m} \fg[[t]]/t^n \fg[[t]] \ar@{=}[d] \ar[r] & 
t^m\fg[[t]]/t^{m+1}\fg[[t]] \ar[d]^{\simeq} \ar[r] & 
0. \\
& t^m \fg[[t]]/t^n \fg[[t]] &
t^{m-1} \fg[[t]]/t^n\fg[[t]] &
\ul{\fg} 
}
\]

\noindent Here we write $\ul{\fg}$ to emphasize we
are considering the \emph{abelian} Lie algebra with vector
space $\fg$.

As an extension of vector spaces, the above 
has an obvious splitting
$(\xi \in \ul{\fg}) \mapsto t^{n-m} \xi$,
so we see that the corresponding
Lie algebra is a Heisenberg Lie algebra
for the degenerate alternating form:

\[
\begin{gathered}
\fg \otimes \fg \to t^{n-1}\fg[[t]]/t^n\fg[[t] \subset 
t^m \fg[[t]]/t^n\fg[[t]] \\
(\xi_1,\xi_2) \mapsto [t^{m-1} \xi_1,t^{m-1} \xi_2].
\end{gathered}
\]

Passing to algebraic groups, we see that an 
action of $\exp(t^{m-1} \fg[[t]]/t^n \fg[[t]])$
on $\sC$ amounts to the following data.
First, performing a Fourier transform
along the central $\exp(t^m \fg[[t]]/t^n \fg[[t]]) = 
t^m \fg[[t]]/t^n \fg[[t]] \otimes \bG_a$,
we obtain a sheaf of categories on:

\[
((t^m \fg[[t]]/t^n \fg[[t]])^{\vee})_{dR} \simeq 
(\fg[[t]]/t^{n-m} \fg[[t]])_{dR} = 
(\fg[[t]]/t^{m-1} \fg[[t]])_{dR}
\]

\noindent where the $\simeq$ is constructed 
as in \S \ref{ss:n/2}; we denote the sheaf of 
categories corresponding to $\sC$ by $\mathsf{C}$.
The remaining data encoding the full 
$\exp(t^{m-1} \fg[[t]]/t^n \fg[[t]])$-action amounts
to an action of a degenerate Heisenberg group 
$\widetilde{\sH}$ on $\mathsf{C}$. In detail:
form a constant vector bundle 
on $(\fg[[t]]/t^{m-1} \fg[[t]])$
with fiber $\fg$, and equip it with the
(degenerate) alternating form whose fiber
at $\xi \in (\fg[[t]]/t^{m-1} \fg[[t]])$ is:

\[
(\vph_1,\vph_2) \in \fg \times \fg \mapsto 
\kappa([\xi(0),\vph_1],\vph_2)
\]

\noindent where $\xi(0)$ indicates the image
of $\xi$ in $\fg$ obtained by $t \mapsto 0$.
The corresponding Heisenberg group scheme 
$\widetilde{\sH}$ defined by this
data acts strongly on $\mathsf{C}$.

In these terms, $\sC^{N_n^{m-1},\psi}$ 
is calculated as global sections of $\mathsf{C}$
on $(f+\fb[[t]]/t^{m-1} \fb[[t]])_{dR}$;
by \S \ref{ss:n/2}, the assignment
$(\sC \in G_n\mod_{reg}) \mapsto 
\sC^{N_n^{m-1},\psi}$ is conservative.

Now observe that the constant vector bundle $\sN$ on 
$f+\fb[[t]]/t^{m-1}\fb[[t]]$ with fiber
$\fn$ satisfies the assumptions of 
Lemma \ref{l:heis-bundle-2} by 
Step \ref{st:reg-symplectic-bundle}, where the
notation of Step \ref{st:reg-symplectic-bundle}
matches that of Lemma \ref{l:heis-bundle-2}
(up to pulling back from $\fg_{reg}$ or $f+\fb$).
We obtain $\sC^{N_n^m,\psi}$ by passing to 
invariants for this Lagrangian subbundle; by Lemma 
\ref{l:heis-bundle-2}, that functor is
conservative, giving the claim.

\section{Convolution for finite Whittaker categories}\label{s:conv}

\subsection{} In this section, we extend the results from
\S \ref{s:heisenberg}. These extensions are given in \S \ref{ss:reg-res}. 
This material plays technical roles in 
\S \ref{s:exactness} and \S \ref{s:bddness}.
The reader may safely skip this section on 
a first read and refer back where necessary.

Key roles are played by Theorems \ref{t:whit-conv} and \ref{t:bbm-2-adolescent}.
The author finds these results to be of independent interest.\footnote{For
instance, using Theorem \ref{t:bbm-2-adolescent} and standard
arguments (relying on \cite{b-fn}), one obtains geometric proofs of
\cite{ginzburg-whittaker-hecke} Theorem 1.6.3 (similarly, Proposition 3.1.2).
In particular, these arguments show that the $t$-exactness from
\emph{loc. cit}. Theorem 1.6.3 applies as well in the $\ell$-adic
context in characteristic $p$ (using Artin-Schreier sheaves
instead of exponential $D$-modules, and needing no special reference
to \cite{b-fn} because ``non-holonomic" objects are meaningless here).} 

\subsection{Main result} 

The first main result of this section is the following:

\begin{thm}\label{t:whit-conv}

For any $n \geq 1$ and any $\sC \in G_n\mod$, the convolution functor:

\[
D(G_n)^{N_n,-\psi} \otimes \sC^{N_n,\psi} \to \sC
\]

\noindent admits a left adjoint. Here $D(G_n)^{N_n,-\psi}$ is the
equivariant category for the action of $N_n$ on $G_n$ on the right.

Moreover, this left adjoint is isomorphic to the composition:

\[
\sC \xar{\coact[-2\dim G_n]} D(G_n) \overset{N_n}{\otimes} \sC 
\xar{\Av_*^{N_n,-\psi} \otimes \id_{\sC}}
D(G_n)^{N_n,-\psi} \otimes \sC^{N_n,\psi}.
\]

\noindent (Because of the diagonal $N_n$-equivariance and by
unipotence of $N_n$, the functor 
$\Av_*^{N_n,-\psi} \otimes \id_{\sC}[2\dim N_n]$ may be replaced
by $\Av_*^{N_n,-\psi} \otimes \Av_*^{N_n,-\psi}[2\dim N_n]$
or $\id_{D(G_n)} \otimes \Av_*^{N_n,-\psi}[2\dim N_n]$.)

\end{thm}

The proof bifurcates into the cases $n \geq 2$ and $n = 1$. In the
former case, the argument is quite similar to the proof of 
Theorem \ref{t:cons}.

\subsection{Reformulation}

First, we begin with a somewhat more convenient formulation of
Theorem \ref{t:whit-conv}.

\begin{thm}\label{t:bbm-2-adolescent}

Let $n \geq 1$ and let $\sC \in G_n \times G_n\mod$.
Then the left adjoint to:\footnote{Note that $(\psi,-\psi)$
restricted to the diagonal $\Delta N_n$ is the trivial character.}

\[
\sC^{N_n \times N_n,(\psi,-\psi)} \xar{\Oblv} \sC^{\Delta N_n} \xar{\Av_*^{\Delta G_n}}
\sC^{\Delta G_n}
\]

\noindent is defined, where $\Delta:G_n \to G_n \times G_n$ is the diagonal 
embedding. For convenience, we denote this left adjoint
by $\Av_!^{\psi,-\psi}$.

Moreover, the canonical natural transformation:

\[
\Av_!^{\psi,-\psi} \to \Av_*^{\psi,-\psi}[2\dim N_n] \in 
\TwoHom_{\DGCat_{cont}}(\sC^{\Delta G}, \sC^{N \times N,(\psi,-\psi)})
\]

\noindent is an equivalence.

\end{thm}

\begin{rem}

In the case $n = 1$ and $\sC = D(G) \in G \times G\mod$,
Theorem \ref{t:bbm-2-adolescent} is \cite{bbm} Theorem 1.5 (2).
However, even in the $n = 1$ case, the result is new 
e.g. for $\sC = D(G \otimes G)$.

\end{rem}

\begin{rem}

In \S \ref{s:central}, we will only need the $n>1$ case 
of Theorem \ref{t:whit-conv}. We include the proof in the $n = 1$
case only for the sake of completeness.

\end{rem}

\begin{proof}[Proof that Theorem \ref{t:bbm-2-adolescent} implies 
Theorem \ref{t:whit-conv}]

Suppose $\sC \in G_n\mod$ is given. We form 
$D(G_n) \otimes \sC \in G_n \times G_n \mod$. By 
Theorem \ref{t:bbm-2-adolescent} (and changing $\psi$ by a sign), 
the map:

\[
\big(D(G_n) \otimes \sC\big)^{N_n \times N_n,(-\psi,\psi)} =
D(G_n)^{N_n,-\psi} \otimes \sC^{N_n,\psi} 
 \xar{\Oblv} D(G_n) \overset{N_n}{\otimes} \sC
\xar{\Av_*^{\Delta G_n}}
D(G_n) \overset{G_n}{\otimes} \sC \isom \sC.
\]

\noindent By definition, the resulting functor is the convolution
functor, so that convolution functor admits a left adjoint.
We similarly obtain the formula for the left adjoint in 
Theorem \ref{t:whit-conv}.

\end{proof}

Below we prove Theorem \ref{t:bbm-2-adolescent}, splitting
it up into different cases. 

\subsection{Proof of Theorem \ref{t:bbm-2-adolescent} for $n = 2$}

We freely use the notation and observations from \S \ref{ss:n=2}. 

As in \emph{loc. cit}., we have:

\[
G_2 \times G_2\mod \simeq \ShvCat_{/\fg_{dR}/G_{dR} \times \fg_{dR}/G_{dR}}.
\]

\noindent Let $\sC \in G_2\times G_2\mod$, and let $\mathsf{C}$ denote
the corresponding sheaf of categories 
on $\fg_{dR}/G_{dR} \times \fg_{dR}/G_{dR}$. The following
commutative diagram provides a dictionary between these two perspectives:

\[
\begin{gathered}
\sC^{N_2 \times N_2,(\psi,-\psi)} \simeq  
\Gamma((f+\fb)_{dR}/N_{dR} \times (-f+\fb)_{dR}/N_{dR},\mathsf{C}) \\
\sC^{\Delta N_2} \simeq 
\Gamma((\fb\times \fb+\Delta^-\fg)_{dR}/N_{dR},\mathsf{C}) \\
\sC^{\Delta G_2} \simeq
\Gamma(\Delta^-\fg_{dR}/G_{dR},\mathsf{C}).
\end{gathered}
\]

\noindent The averaging functor 
$\sC^{N_2 \times N_2,(\psi,-\psi)} \to \sC^{\Delta G_2}$
corresponds to $!$-pullback and then $*$-pushforward 
(in the $D$-module sense, which tautologically 
adapts to sheaves of categories on 
de Rham stacks) along the correspondence:

\[
\xymatrix{
& (f+\fb)/N \ar[dr] \ar[dl]_{\Delta^-} & \\
(f+\fb)/N \times (-f+\fb)/N & & \fg/G.
}
\]

\noindent The left map $\Delta^-$ is a closed embedding because the Kostant
slice $(f+\fb)/N$ is an affine scheme, so $!$-pullback along it
admits a left adjoint.
The right map is smooth, so $!$-pullback along it equals $*$-pullback up 
to shift; in particular, the relevant $*$-pushforward admits a left adjoint.

This shows that our $*$-averaging functor admits a left adjoint in 
this case. That the comparison map 
$\Av_!^{\psi,-\psi} \to \Av_*^{\psi,-\psi}[2\dim N_2]$ effects this
isomorphism follows from the above analysis.

\subsection{Proof of Theorem \ref{t:bbm-2-adolescent} for $n>2$}

The argument proceeds as in the proof of Theorem \ref{t:cons};
we use the notation from that proof in what follows.

First, observe that it is equivalent to show that the left adjoint
$\Av_!^{\psi} = \Av_!^{N_n,\psi}$ to 
$\Av_*^{\Delta(G_n)}:\sC^{N_n \times 1,\psi} \to \sC^{\Delta(G_n)}$ is 
defined, with the natural map $\Av_!^{\psi} \to \Av_*^{\psi}[2\dim N_n]$
being an isomorphism; indeed, $\Av_*^{\Delta G_n}$ factors as:

\[
\sC^{N_n \times 1,\psi} \xar{\Av_*^{\Delta N_n}} 
(\sC^{N_n \times 1,\psi})^{\Delta N_n} = 
\sC^{N_n\times N_n,(\psi,-\psi)} \xar{\Av_*^{\Delta G_n}} \sC^{\Delta G_n}
\]

\noindent and the first functor admits the fully faithful left adjoint
$\Oblv$.

By induction on $m$, we will show that the appropriate left adjoint
$\Av_!^{N_n^m,\psi}:\sC^{\Delta G_n}\to \sC^{N_n^m \times 1,\psi}$
is defined, and that the natural map $\Av_!^{N_n^m,\psi} \to 
\Av_*^{N_n^m,\psi}[2\dim N_n^m]$ is an equivalence.

As in the proof of Theorem \ref{t:cons}, the base case $m = 1$
is a consequence of the $n = 2$ case proved in \S \ref{ss:n=2}.
Moreover, as in \S \ref{ss:n/2}, essentially the 
same argument applies for 
$m \leq \frac{n}{2}$. 
As in \S \ref{ss:2m+1}, the natural generalization
of Lemma \ref{l:av-lagr} vector bundles with alternating
bilinear forms allows us to 
deduce the special case where $n = 2m-1$; we
omit the details, which are quite similar to \S \ref{ss:2m+1}.

Now in what follows, we assume $m \geq \frac{n+2}{2}$.
By descending induction on $r$, 
we will show that the appropriate left adjoint
$\Av_!^{N_n^{m,r},\psi}:\sC^{\Delta G_n}\to \sC^{N_n^{m,r} \times 1,\psi}$
is defined, and that the natural map $\Av_!^{N_n^{m,r},\psi} \to 
\Av_*^{N_n^{m,r},\psi}[2\dim N_n^{m,r}]$ is an equivalence.
The base case $r \gg 0$ amounts to the inductive hypothesis for
$m-1$.

To perform the induction, 
we use the following observation.

\begin{lem}\label{l:av-lagr}

Let $V$ be a finite-dimensional 
vector space over $k$ and let
$H = H(V)$ be the associated Heisenberg group, as in \S \ref{ss:heis-defin}.

Let $\sC \in H\mod_{reg}$. Then the functor
$\Av_*^V:\sC^{V^{\vee}} \to \sC^{V}$ is an equivalence.

Moreover, if we (appropriately) denote the inverse functor
$\Av_!^{V^{\vee}}$, then the natural map $\Av_!^{V^{\vee}} \to 
\Av_*^{V^{\vee}}[2\dim V]$ is an equivalence.

\end{lem}

\begin{proof}

Immediate from the proof of Theorem \ref{t:heisenberg}.

\end{proof}

The relevant Heisenberg group is constructed as follows.
Here we use notation parallel to the proof of 
Theorem \ref{t:cons},
but the meanings are different in the present context.

Define $\fh_0$ as 
$\Lie(N_n^{m,r} \times 1)+ \Delta(t^{m-1} \fg_{1-r})
\subset \Lie(G_n \times G_n)$. Define $\fh_0^{\prime}$ 
similarly, but with $\fg_{1-r}^{\prime}$ in place of
$\fg_{1-r}$ (in the notation of \S \ref{ss:cons-pf}).

As in the proof of Theorem \ref{t:cons},
these are nilpotent Lie subalgebras of $\Lie(G_n \times G_n)$,
and there are associated unipotent subgroups
$H_0^{\prime} \subset H_0 \subset G_n \times G_n$.
And again by the same argument as in \emph{loc. cit}.,
$(N_n^{m,r+1} \times 1) \subset H_0$ is normal, and its character
is stabilized by the adjoint action of $H_0$.
We again let $K \subset (N_n^{m,r+1} \times 1)$ denote the kernel
of the character and $H \coloneqq H_0^{\prime}/K$; again,
$H$ is a Heisenberg group.

By induction, we have a $!$-averaging functor:

\[
\Av_!^{N_n^{m,r+1},\psi} = \Av_*^{N_n^{m,r+1}}[2\dim N_n^{m,r+1}]:\sC^{\Delta G_n}
\to \sC^{N_n^{m,r+1} \times 1,\psi}
\]

\noindent which evidently lifts to invariants for the additive
subgroup $\Delta(\fg_{1-r}^{\prime}) \subset H$.
By Lemma \ref{l:av-lagr}, we can $!$-average 
$(\sC^{N_n^{m,r+1} \times 1,\psi})^{\Delta(\fg_{1-r}^{\prime})} \to 
\sC^{N_n^{m,r} \times 1,\psi}$, and this coincides with $*$-averaging
up to suitable shift (and moreover, the resulting functor gives an 
equivalence $(\sC^{N_n^{m,r+1} \times 1,\psi})^{\Delta(\fg_{1-r}^{\prime})} 
\isom \sC^{N_n^{m,r} \times 1,\psi}$). This gives the claim.

\subsection{Proof of Theorem \ref{t:bbm-2-adolescent} for $n = 1$}\label{ss:bmm-pf-n=1}

Let $B^-$ be a Borel opposed to $B$ with
radical $N^-$.  

\setcounter{steps}{0}
\step\label{st:spr}

We have the functor:
 
\[
\Psi:\sC^{\Delta G} \xar{\Oblv} \sC^{\Delta B^-} 
\xar{\Av_*}
\sC^{N^- \times N^- \cdot \Delta T}.
\]

\noindent This functor admits the left adjoint:

\[
\Xi:\sC^{N^- \times N^- \cdot \Delta T} \xar{\Oblv} \sC^{\Delta B^-}
\xar{\Av_!} \sC^{\Delta G}
\]

\noindent with $\Av_! = \Av_*[2\dim G/B^-]$ by
properness of $G/B^-$.

Recall from \cite{mirkovic-vilonen-but-not-that-one} 
that the counit map $\Xi \Psi \to \id$ splits.
Indeed, as in \emph{loc. cit}., $\Xi \Psi$ is computed
as convolution with the Springer sheaf
in $D(G)^{\Ad G} = D(\Delta G\backslash (G \times G)/\Delta G)$,
and by an argument in \emph{loc. cit}. using the decomposition theorem, 
the Springer sheaf admits the skyscraper sheaf at $1 \in G\overset{\Ad}{/} G$ 
as a summand.

In particular, every $\sF \in \sC^{\Delta G}$ is a summand of
an object of the form $\Xi(\sF^{\prime})$.

\step 

Next, we recall a key result of \cite{bbm}. Theorem 1.1 (1)
of \emph{loc. cit}. implies that we can 
$!$-average $N^-$-equivariant objects to
be $(N,\psi)$-equivariant, and this $!$-average coincides
with the $*$-averaging after shift by $2\dim N$.
(Note that the authors work in the setting of perverse sheaves,
but their argument works in this generality: c.f. 
the proof of \cite{whit} Theorem 2.7.1.) 

Applying this for $G \times G$ instead, 
we see that for $\sF \in \sC^{N^- \times N^-}$
(or $\sF \in \sC^{N^- \times N^- \cdot \Delta T}$),
we can form $\Av_!^{(\psi,-\psi)} \sF \in \sC^{N \times N,(\psi,-\psi)}$,
and the natural map:

\[
\Av_!^{(\psi,-\psi)} \sF \to \Av_*^{(\psi,-\psi)} \sF[4\dim N]
\]

\noindent is an isomorphism.

\step 

Now suppose that $\sF \in \sC^{N^- \times N^- \cdot \Delta T}$.
We claim that $\Av_!^{\psi,-\psi} \sF$ coincides with
$\Av_!^{\psi,-\psi} \Xi(\sF)$; in particular, the latter term is defined.

By base-change, $\Av_!^{\psi,-\psi} \Xi(\sF)$ should be computed
as follows. 
We have a functor:
 
\[
\Av_! = \Av_*[2\dim G/B^-]: D(N\backslash G) \overset{B^-}{\otimes} \sC \to
D(N\backslash G) \overset{G}{\otimes} \sC = \sC^{\Delta N}.
\]

\noindent Also, $\sF$ defines an object 
$\widetilde{\sF}$ (i.e., $\omega_{N\backslash G} \widetilde{\boxtimes} \sF$) 
in $D(N\backslash G) \overset{B^-}{\otimes} \sC$. Finally, the recipe
says that to compute
$\Av_!^{\psi,-\psi} \Xi(\sF)$, we should 
form $\Av_!(\widetilde{\sF}) \in \sC^{\Delta N}$ and 
then further $!$-average to $\sC^{N \times N,(\psi,-\psi)}$.

Observe that $\widetilde{\sF}$ carries a canonical Bruhat filtration.
More precisely, for $w$ an element of the Weyl group $W$, let
$i_w$ denote the locally closed embedding
$N\backslash NwB^- \into N\backslash G$. Let 
${\sF^w \in D(N \backslash N w B^-) \overset{B^-}{\otimes} \sC}$ 
be the object induced by $\sF$, 
so $\widetilde{\sF}$ is filtered with subquotients
$i_{w,*,dR}(\sF^w)$.

Let $N^w = N \cap \Ad_w(B^-)$.
Then $D(N \backslash N w B^-) \overset{B^-}{\otimes} \sC 
\simeq \sC^{\Delta N^w}$,
since $N w B^- = N \overset{N^w}{\times} B^-$,
where $N^w$ maps to $B^-$ via $\Ad_{w^{-1}}$.
The object $\sF^w$ is then\footnote{Here
$g \cdot \sF$ is by definition $\delta_g \convolution \sF$,
and we are using the diagonal action of $G$ on $\sC$.}
$w \cdot \sF$, which
we note is equivariant for 
$\Ad_{\Delta w} (N^- \times N^- \cdot \Delta T) \supset 
N^w \times N^w \supset \Delta N^w$.

Then observe that up to cohomological shift,
$\Av_! i_{w,*,dR}(\sF^w) \in \sC^{\Delta N}$
is obtained by $*$-averaging
$w \cdot \sF$ from $\Delta N^w$ to 
$\Delta N$, since $\Av_!$ is $!$-averaging
from $B^-$ to $G$, and therefore coincides with $*$-averaging
up to shift.

Now for $w \neq 1$, recall that the character $\psi$ is non-trivial
on $N \cap \Ad_w(N^-)$. Therefore, $!$-averaging to
$(N\times N,(\psi,-\psi))$-equivariance vanishes
on $\sC^{N^w \times N^w}$. In particular, this $!$-averaging is defined.
(The same applies for $*$-averaging.)

This vanishing implies:

\[
\Av_!^{N \times N,(\psi,-\psi)} \Xi(\sF) = 
\Av_!^{N \times N,(\psi,-\psi)} \sF^1.
\]

\noindent (Here $1 \in W$ is the unit in the Weyl group.) We
note that $\sF^1 = \Oblv \sF \in \sC = D(NB^-) \overset{B^-}{\otimes} \sC$.
Since this last $!$-averaging is defined by \cite{bbm} Theorem 1.1 (1),
we obtain the result.

\step We have now shown $\Av_!^{\psi,-\psi} \sF$ is defined
for $\sF \in \sC^{\Delta G}$.
All that is left is to check that the natural map:

\[
\Av_!^{\psi,\psi} \sF \to \Av_*^{\psi,\psi} \sF [2\dim N]
\]

\noindent is an isomorphism. 

We may assume $\sF = \Xi \sG$ for $\sG \in \sC^{N^- \times N^- \cdot \Delta T}$.
In this case, the assertion is a straightforward verification 
in the above argument.

\subsection{Application: construction of resolutions}\label{ss:reg-res}

For the remainder of the section, we assume $n \geq 2$. 

For $\sC \in G_n\mod$, let 
$j^!:\sC \rightleftarrows \sC_{reg}:j_{*,dR}$
be as in \S \ref{ss:fourier-decomp}.

For $\sC = D(G_n)$, let $\delta_1 \in D(G_n)$ be the skyscraper
$D$-module at the identity, and let 
$\delta_1^{reg} \coloneqq j_{*,dR}j^!(\delta_1)$. Note that
for any $\sC \in G_n\mod$, 
the convolution functor $\delta_1^{reg} \star -$ is isomorphic to 
$j_{*,dR}j^!$ as endofunctors of $\sC$.

\begin{lem}\label{l:delta-reg}

$\delta_1^{reg}$ lies in the full subcategory of $D(G_n)$
generated by the essential image of 
the functor:

\[
D(G_n)^{N_n,-\psi,+} \times D(G_n)^{N_n,\psi,+} \to 
D(G_n)^{N_n,-\psi} \otimes D(G_n)^{N_n,\psi} \to 
D(G_n)
\]

\noindent under finite colimits and direct summands. Here the first
factor $D(G_n)^{N_n,-\psi}$ has invariants taken on the right,
$D(G_n)^{N_n,\psi}$ has invariants on the left, and both
terms are considered with their natural $t$-structures.

\end{lem}

\begin{proof}

Suppose $\sC \in G_n\mod$. 
By Theorem \ref{t:whit-conv}, the convolution functor:

\[
D(G_n)^{N_n,-\psi} \otimes \sC^{N_n,\psi} \to \sC
\]

\noindent admits a left adjoint. Moreover, this left adjoint
is a morphism in $G_n\mod$ (where a priori, it is lax). 
Passing to $(N_n,\psi)$-invariants, we see that the functor:

\[
\sH_{N_n,\psi} \otimes \sC^{N_n,\psi} \to \sC^{N_n,\psi}
\] 

\noindent admits a left adjoint that is a morphism of
$\sH_{N_n,\psi}$-module categories (for $\sH_{N_n,\psi}$ as in
Corollary \ref{c:hecke}). 

By the above remarks and \cite{shvcat} Corollary C.2.3, 
the morphism:

\[
D(G_n)^{N_n,-\psi} \otimes \sC^{N_n,\psi} \to 
D(G_n)^{N_n,-\psi} \underset{\sH_{N_n,\psi}}{\otimes} \sC^{N_n,\psi}
\]

\noindent admits a monadic (discontinuous!) right adjoint. 
By Corollary \ref{c:hecke}, the
right hand side maps isomorphically onto $\sC_{reg}$.

Let $\on{conv}:D(G_n)^{N_n,-\psi} \otimes \sC^{N_n,\psi} \to \sC$
denote the convolution functor, let $\on{conv}^R$ denote
its (discontinuous!) right adjoint, and let 
$T = \on{conv} \circ \on{conv}^R:\sC \to \sC$ denote the corresponding
monad. Clearly $\on{conv}$ factors through $\sC_{reg}$, and
$\on{conv}^R \circ j_{*,dR}$ is the right adjoint to the corresponding
functor $D(G_n)^{N_n,-\psi} \otimes \sC^{N_n,\psi} \to \sC_{reg}$.

Therefore, the monadic conclusion above shows that for any
$\sF \in \sC_{reg} \overset{j_{*,dR}}{\subset} \sC$, the geometric realization 
$|T^{\dot}(\sF)| \in \sC$ maps isomorphically onto $\sF$.

We now specialize to the case $\sC = D(G_n)$ and $\sF = \delta_1^{reg}$.
Note that $\delta_1^{reg}$ is holonomic in $D(G_n)$ and therefore
compact. Therefore, as:

\[
\delta_1^{reg} = |T^{\dot}(\delta_1^{reg})| = 
\underset{r}{\colim} \, |T^{\dot}(\delta_1^{reg})|_{\leq r}
\]

\noindent (for $|-|_{\leq r}$ the usual partial geometric realization,
i.e., the colimit over $\bDelta_{\leq r}^{op}$), we obtain
that $\delta_1^{reg}$ is a direct summand
of $|T^{\dot}(\delta_1^{reg})|_{\leq r}$ for some $r$.

We conclude in noting that $T$ is left $t$-exact up to shift as
$\on{conv}$ is both left and right $t$-exact up to shift.
Any object of $D(G_n)^{N_n,-\psi} \otimes D(G_n)^{N_n,\psi} = 
D(G_n \times G_n)^{N_n \times N_n,(-\psi,\psi)}$ bounded
cohomologically bounded from below lies in the full
subcategory generated by the image of
$D(G_n)^{N_n,-\psi,+} \times D(G_n)^{N_n,\psi,+}$, so we obtain the
claim.

\end{proof}

We obtain the following result, which is a sort of effective
version of Theorem \ref{t:cons}.

\begin{cor}\label{c:reg-eff}

Suppose that $n \geq 2$ and $\sC \in G_n\mod$. Then for
any $\sF \in \sC_{reg}$, $\sF$ lies in the full subcategory
of $\sC$ generated under finite colimits and direct summands by 
the essential image of the convolution functor:

\[
D(G_n)^{N_n,-\psi} \otimes \sC^{N_n,\psi} \to \sC.
\] 

Moreover, if $\sC$ has a $t$-structure compatible 
with the
action of $G_n$ on it, and if $\sF \in \sC_{reg} \cap \sC^+$, then 
$\sF$ lies in the full subcategory of $\sC$ generated
under finite colimits and direct summands by
the essential image of the convolution functor:

\[
D(G_n)^{N_n,-\psi,+} \times \sC^{N_n,\psi,+} \to 
D(G_n)^{N_n,-\psi} \otimes \sC^{N_n,\psi} \to \sC.
\] 

\end{cor}

\begin{proof}

Suppose $\sG_1\in D(G_n)^{N_n,-\psi,+}$ and $\sG_2 \in 
D(G_n)^{N_n,\psi,+}$, with conventions for the actions as in 
Lemma \ref{l:delta-reg}.
Then $\sG_2 \star \sF \in \sC^{N_n,\psi}$, so 
$\sG_1 \star \sG_2 \star \sF \in \sC$ lies in the essential
image of the convolution functor. 

Moreover, in the presence of a $t$-structure on $\sC$ as in the 
second part of the assertion, 
$\sG_2 \star \sF \in \sC^{N_n,\psi,+}$ and
$\sG_1 \star \sG_2 \star \sF \in \sC^+$ lies in the essential
image of the functor considered in the second part.

Now we obtain the result by Lemma \ref{l:delta-reg}.

\end{proof}

\begin{cor}\label{c:av-!-cons}

For any $\sC \in G_n\mod$, the functor 
$\Av_!^{\psi,-\psi}:\sC \to D(G_n)^{N_n,\psi} \otimes \sC$
restricts to a conservative functor on $\sC_{reg}$.

\end{cor}

\begin{proof}

Let $\sF \in \sC_{reg}$, and assume $\sF$ is non-zero. 
We need to show that $\Av_!^{\psi,-\psi}(\sF) \neq 0$.

By Corollary \ref{c:reg-eff}, there exists
$\sG \in D(G_n)^{N_n,\psi}$ with $\sG \star \sF \neq 0$
in $\sC^{N_n,\psi}$. As $D(G_n)^{N_n,\psi}$ is compactly
generated, we may assume that $\sG$ is compact.

Note that $D(G_n)^{N_n,\psi}$ is canonically dual as a DG category
to $D(G_n)^{N_n,-\psi}$. Let $\bD \sG:D(G_n)^{N_n,-\psi} \to \Vect$
denote the functor dual to the compact object $\sG$ (explicitly,
this functor is given as $\Hom$ out of the Verdier dual to $\sG$). 

Then the convolution $\sG \star \sF$ may be calculated
by forming $\Av_*^{\psi,-\psi}(\sF) \in 
D(G_n)^{N_n,-\psi} \otimes \sC^{N_n,\psi}$
and then applying $\bD \sG \otimes \id_{\sC^{N_n,\psi}}$.
In particular, we deduce that $\Av_*^{\psi,-\psi}(\sF)$ is non-zero.
As $\Av_*^{\psi,-\psi}(\sF)$ coincides with 
$\Av_!^{\psi,-\psi}(\sF)$ up to shift, we obtain the claim.

\end{proof}

\section{Most $PGL_2$-representations are generic}\label{s:adolescent}

\subsection{}

We now prove the following result.

\begin{thm}\label{t:generic}

Let $G = PGL_2$ and let $\sC$ be acted on by $G(K)$ with level $\kappa$.
Then $\sC$ is generated under the action of $G(K)$ by 
$\Whit(\sC)$ and $\sC^{\o{I}}$ where $\o{I} \subset G(K)$ 
is the radical of the Iwahori subgroup. That is, any
subcategory of $\sC$ that is closed under colimits,
contains $\Whit(\sC)$ and
$\sC^{\o{I}}$, and is closed under the $G(K)$ action is $\sC$ itself.

\end{thm}

\begin{rem}

This result is reminiscent of the existence of Whittaker
models for those irreducible smooth representations
of $GL_2$ over a locally compact non-Archimedean field with
non-trivial restriction to $SL_2$.

However, in Theorem \ref{t:generic}, $\o{I}$ cannot\footnote{However,
$\o{I}$ can be strengthened somewhat: one can take
invariants with respect to the Iwahori subgroup of $SL_2(K)$,
i.e., the canonical degree $2$ cover of Iwahori.}
be replaced by $G(K)$:
this can be seen by applying Bezrukavnikov's theory \cite{bezrukavnikov-hecke}
to local systems with non-trivial unipotent monodromy
(c.f. with the ideas of \cite{arinkin-gaitsgory} in the spherical setting).
Note that such local systems are outside the scope of 
arithmetic Langlands
because they are not semisimple.

\end{rem}

\subsection{Review of adolescent Whittaker theory}\label{ss:ad-whit-review}

We prove Theorem \ref{t:generic} using the theory of \cite{whit} \S 2.
For convenience, we review this here.

Let $G$ be an adjoint\footnote{This is only for the convenience
of using the action of $\check{\rho}(t) \in G(K)$ on $\sC$.
In fact, \cite{whit} uses different indexing conventions than we use here,
and which are better adapted to a general reductive group.}
group and let $\sC \in G(K)\mod_{\kappa}$ be acted on by
$G(K)$ with some level $\kappa$. 
We use the notation of \S \ref{s:heisenberg}.
Let $K_n \subset G(O) \subset G(K)$ denote the $n$th congruence subgroup
and observe that $G_n$ acts on $\sC^{K_n}$.

For $n>0$, define $\Whit^{\leq n}(\sC) \coloneqq (\sC^{K_n})^{N_n,\psi}$.
There is a natural functor $\Whit^{\leq n+1}(\sC) \to \Whit^{\leq n}(\sC)$:

\[
\sF \mapsto \Av_*^{K_n} (-\check{\rho}(t) \star \sF)
\]

\noindent and which is denoted $\iota_{n,n+1}^!$ in \emph{loc. cit}.

\begin{thm}[\cite{whit} Theorem 2.7.1]\label{t:whit}

The functor $\iota_{n,n+1}^!$ admits a left adjoint
$\iota_{n,n+1,!}$. This left adjoint is given by convolution with
some $D$-module on $G(K)$. 

Moreover, there is a natural equivalence:

\[
\underset{n,\iota_{n,n+1,!}}{\colim} \, \Whit^{\leq n}(\sC) \isom 
\Whit(\sC)
\in \DGCat_{cont}.
\]

\noindent The structural functors $\Whit^{\leq n}(\sC) \to \Whit(\sC)$
are left adjoint to the natural functors
$\Av_*^{K_n} \circ (\delta_{n\check{\rho}(t)} \star -):
\Whit(\sC) \to \Whit^{\leq n}(\sC)$.
In particular,
every object $\sF \in \Whit(\sC)$ is canonically a colimit (in $\sC$)
of objects $\sF_n$ with $\delta_{\check{\rho}(t^n)} \star \sF \in 
\Whit^{\leq n}(\sC)$.

\end{thm}

\subsection{Proof of Theorem \ref{t:generic}}\label{ss:descent}

We can now prove the main theorem of this section. Below, $G = PGL_2$.

Let $\sC^{\prime} \subset \sC$ be a $G(K)$-subcategory containing
$\Whit(\sC)$ and $\sC^{\o{I}}$, and we
wish to show that $\sC^{\prime} = \sC$.

Recall that $\sC = \colim_n \sC^{K_n} \in \DGCat_{cont}$. Therefore, 
it suffices to show that $\sC^{\prime}$ contains $\sC^{K_n}$ for all $n \geq 1$.
We do this by induction on $n$.

In the base case $n = 1$, recall that for any $\sD$ acted on by\footnote{Finite-dimensional, and here arbitrary reductive is fine.} $G$, $\sD$ is
the minimal cocomplete subcategory of itself closed under the $G$-action
and containing $\sD^N$; indeed, this follows from the main
theorem of \cite{bzgo}.\footnote{Or it follows from usual Beilinson-Bernstein
localization theory: by reduction to the case $\sD = D(G)$, one finds
that $\sD^{G,w}$ is a colocalization of 
$(\sD^N)^{T,w}$, and then use conservativeness of weak invariants
(\cite{shvcat} Theorem 2.2.2).}
Applying this to $\sD = \sC^{K_1}$, we find that 
$\sC^{K_1}$ can be generated from $\sC^{\o{I}}$ using the
action of $G \subset G(K)$.

Now suppose the claim is true for $n$, and let us show it for $n+1$.
Note that $n+1 \geq 2$, so we may apply the methods
of \S \ref{s:heisenberg} to $\sC^{K_{n+1}}$ with
its canonical $G_{n+1}$-action.
In the notation of \emph{loc. cit}., we have adjoint functors:

\[
j^!:\sC^{K_{n+1}} \rightleftarrows (\sC^{K_{n+1}})_{reg}:j_{*,dR}.
\]

\noindent Note that 
$(\sC^{K_{n+1}})_{reg} \subset \sC^{\prime}$
by Corollary \ref{c:heisenberg}, as
$\Whit^{\leq n+1}(\sC) = (\sC^{K_{n+1}})^{N_{n+1},\psi} 
\subset \sC^{\prime}$ by hypothesis on $\sC^{\prime}$
(and Theorem \ref{t:whit}).

Therefore, it suffices to show that $\Ker(j^!) \subset \sC^{\prime}$.
Then we observe that $\fg_{reg} = \fg\setminus 0$ for $\fg = \sl_2$,
so (in the notation of \S \ref{ss:fourier-decomp}),
$\Ker(j^!) = \sC^{K_{n+1}}|_0 = \sC^{K_n}$ as we have the short
exact sequence:

\[
1 \to \fg \otimes \bG_a \to K_{n+1} \to K_n \to 1.
\]

\noindent But $\sC^{K_n} \subset \sC^{\prime}$ by induction.

\begin{rem}

The above is the \emph{descent} method discussed in the 
introduction. As this argument plays a key role in the 
paper, we reiterate the idea: with notation as above,
for $\sC \in G(K)\mod_{\kappa}$, 
$\Ker(\Av_*^{K_n}:\sC^{K_{n+1}} \to \sC^{K_n})$ 
is the category $(\sC^{K_{n+1}})_{reg}$, understood
in the sense of \S \ref{ss:fourier-decomp} for the
corresponding $G_{n+1}$-action. By 
Theorem \ref{t:cons} and Theorem \ref{t:whit}, 
this kernel may therefore be functorially described in terms of
the Whittaker model for $\sC$. 

One can then try to verify some property of objects
of $\sC$ as follows:

\begin{enumerate}

\item Reduce to showing the property for objects
in $\sC^{K_n}$ for some $n$.

\item Use the Whittaker model and the above observations
to inductively reduce to the $n = 1$ case.

\item Use \cite{bzgo} to reduce the $n = 1$ case to
a property of objects in $\sC^{\o{I}}$.

\end{enumerate}

\end{rem}

\section{Kac-Moody modules with central character}\label{s:central}

\subsection{}

In this section, we study categories of 
critical level Kac-Moody representations
with central character restrictions. We refer back to \S \ref{ss:intro-crit}
for a review of standard notation at critical level. 

First, for 
reductive $G$ and any $n \geq 0$, we will introduce a certain category:

\[
\widehat{\fg}_{crit}\mod_{\ord_n,naive} \in \DGCat_{cont}
\]

\noindent with a critical level $G(K)$-action. 

In the above, the subscript
$\ord_n$ indicates that we look at $\widehat{\fg}_{crit}$-modules
on which the center $\fZ$ acts through 
a certain standard quotient $\fZ \onto \fz_n$, and in a suitable derived
sense. Equivalently, under Feigin-Frenkel, these can be thought
of as representations \emph{scheme-theoretically} supported on 
$\Op_{\ld{G}}^{\leq n} \subset \Op_{\ld{G}}$, where 
$\Op_{\ld{G}}^{\leq n}$ are opers with singularity of 
order $\leq n$, c.f. \cite{hitchin} \S 3.8 or \cite{fg2} \S 1.

For $n = 0$, $\fz_0 = \fz$. Here the
central character condition is the regularity assumption from 
\S \ref{ss:intro-loc-start}, so we use the notation $reg$ in place
of $\ord_0$.

In the spirit of \cite{methods}, the subscript $naive$ indicates that
this is not the best derived category to consider. 
For instance, 
$\widehat{\fg}_{crit}\mod_{\ord_n,naive}$ is not compactly
generated. And for $n = 0$, the analogue of Conjecture \ref{conj:intro-fg-loc}
fails for it.

Following \cite{dmod-aff-flag} \S 23, we introduce a somewhat
better \emph{renormalized}
category $\widehat{\fg}_{crit}\mod_{\ord_n}$. This category
will have a forgetful functor:

\[
\widehat{\fg}_{crit}\mod_{\ord_n} \to 
\widehat{\fg}_{crit}\mod_{\ord_n,naive}
\]

\noindent that is $t$-exact for suitable $t$-structures and an 
equivalence on eventually coconnective subcategories.

However, this renormalization procedure is somewhat subtle,
and there are many basic questions about
$\widehat{\fg}_{crit}\mod_{\ord_n}$ that I do not know how
to answer. For instance, I cannot generally
show that there is a $G(K)$-action on 
$\widehat{\fg}_{crit}\mod_{\ord_n}$ compatible with the
forgetful functor above. We refer to \S \ref{ss:eq-renorm} for
further discussion.

The material of this section is technical. 
Proposition \ref{p:central-naive} and Lemma \ref{l:ren} are the key points.
After understanding the statements of these results, the reader
should be equipped to move on to future sections.

Finally, we highlight that the material of this section relies
on \cite{methods} \S 11 and extends the material
from \emph{loc. cit}.

\subsection{Notation at critical level}

As in \cite{methods} \S 11, we use the following notation.
We refer to \cite{fg2} \S 1 for background on opers.

First, $\Op_{\ld{G}}$ denotes the indscheme of $\ld{G}$-opers
on the punctured disc. We let $\Op_{\ld{G}}^{\leq n} \subset
\Op_{\ld{G}}$ denote the subscheme of opers with singularities
of order $\leq n$.

We remind that $\Op_{\ld{G}}^{\leq n}$ is affine for every
$n$; we let $\fz_n$ denote the corresponding algebra of
functions, so $\Op_{\ld{G}}^{\leq n} = \Spec(\fz_n)$. 
We remind that $\fz_n$ is a polynomial algebra in infinitely
many variables.

We let $\fZ$ denote the commutative $\overset{!}{\otimes}$-algebra
$\lim_n \fz_n \in \Pro\Vect^{\heart}$, 
the limit being taken in $\Pro\Vect^{\heart}$;
we refer to \cite{methods} for the terminology on topological
algebras used here. We remark that $\Op_{\ld{G}} = \Spf(\fZ)$.

By Feigin-Frenkel (see \cite{ff-critical} and \cite{hitchin} \S 3),
$\fZ$ naturally identifies with $U(\widehat{\fg}_{crit})$,
the twisted topological enveloping algebra of 
$\widehat{\fg}_{crit}$.

We let 
$\bV_{crit,n} \coloneqq 
\ind_{t^n\fg[[t]]}^{\widehat{\fg}_{crit}}(k) \in 
\widehat{\fg}_{crit}\mod^{\heart}$.

\subsection{Naive categories}

We begin with some preliminary notations.

First, if $\sA \in \CoAlg(\DGCat_{cont})$ and 
$\sM$ (resp. $\sN$) is a right (resp. left) comodule for $\sA$,
we let:

\[
\sM \overset{\sA}{\otimes} \sN \in \DGCat_{cont}
\] 

\noindent denote the cotensor product of these comodules.
By definition, this means we regard $\sA$ as an algebra in the opposite category
$\DGCat_{cont}^{op}$ and form the usual tensor product there.
This cotensor product may be calculated as a totalization in $\DGCat_{cont}$:

\[
\sM \overset{\sA}{\otimes} \sN  = 
\Tot \Big( \sM \otimes \sN \rightrightarrows 
\sM \otimes \sA \otimes \sN \rightrightrightarrows \ldots \Big).
\]

Next, for $S$ a reasonable indscheme in the sense of
\cite{methods} \S 6, recall that we have the compactly generated
DG category $\IndCoh^*(S) \in \DGCat_{cont}$.
This construction is 
covariantly functorial in $S$. In particular, if
$S$ is a reasonable indscheme that is 
\emph{strict},\footnote{See \emph{loc. cit}. for the definition.
The relevance here is that this condition implies e.g. that
the natural functor $\IndCoh^*(S) \otimes \IndCoh^*(S) 
\to \IndCoh^*(S \times S)$ is an equivalence.}
$\IndCoh^*(S)$ is canonically a cocommutative coalgebra
in $\DGCat_{cont}$. 

\subsection{}\label{ss:naive-defin}

Note that $\Op_{\ld{G}}$ is a strict, reasonable indscheme.
By \cite{methods} Theorem 11.18.1,  
$\widehat{\fg}_{crit}\mod \in G(K)\mod_{crit}$ is canonically an 
$\IndCoh^*(\Op_{\ld{G}})$-comodule (in $G(K)\mod_{crit}$). 

For $n \geq 0$, define:

\[
\widehat{\fg}_{crit}\mod_{\ord_n,naive} \coloneqq
\IndCoh^*(\Op_{\ld{G}}^{\leq n}) 
\overset{\IndCoh^*(\Op_{\ld{G}})}{\otimes} 
\widehat{\fg}_{crit}\mod \in G(K)\mod_{crit}.
\]

Let $i_n$ denote the embedding $\Op_{\ld{G}}^{\leq n} \to
\Op_{\ld{G}}$. We abuse notation in letting
$i_{n,*}: \widehat{\fg}_{crit}\mod_{\ord_n,naive} \to 
\widehat{\fg}_{crit}\mod$ denote the functor 
$i_{n,*}^{\IndCoh} \overset{\IndCoh^*(\Op_{\ld{G}})}{\otimes}
\id_{\widehat{\fg}_{crit}\mod}$. By \cite{methods} Lemma 6.17.1-2,
this functor admits a continuous right adjoint
$i_n^! \overset{\IndCoh^*(\Op_{\ld{G}})}{\otimes}
\id_{\widehat{\fg}_{crit}\mod}$, which we also denote
$i_n^!$. 
Note that $i_{n,*}$ and $i_n^!$ are (by construction) morphisms
of $\IndCoh^*(\Op_{\ld{G}})$-module categories.

Similarly, for $m \geq n$, we have a natural adjunction:

\[
i_{n,m,*}:
\widehat{\fg}_{crit}\mod_{\ord_n,naive} \to 
\widehat{\fg}_{crit}\mod_{\ord_m,naive}:
i_{n,m}^!
\]

\noindent with $i_{n,*} = i_{m,*} \circ i_{n,m,*}$. 
Note that $i_{n,m,*}$ actually admits a \emph{left} adjoint
$i_{n,m}^*$ as well as a right adjoint; this follows because
the closed embedding $i_{n,m}:\Op_{\ld{G}}^{\leq n} \into 
\Op_{\ld{G}}^{\leq m}$ is a finitely presented regular embedding.

\begin{rem}

For a reasonable indscheme $S$, we let $\IndCoh^!(S)$ denote the 
dual DG category to $\IndCoh^*(S)$; this construction is contravariantly 
functorial
in $S$. For strict $S$, $\IndCoh^!(S)$ is therefore a 
symmetric monoidal category.

In these terms, we can reformulate the above definition
(to use monoidal categories instead of ``comonoidal" categories):

\[
\widehat{\fg}_{crit}\mod_{\ord_n,naive} = 
\TwoHom_{\IndCoh^!(\Op_{\ld{G}})\mod}(\IndCoh^!(\Op_{\ld{G}}^{\leq n}),
\widehat{\fg}_{crit}\mod).
\]

\end{rem}

\subsection{}\label{ss:op-action}

We record what symmetries the above construction provides.

As indicated above, there is an evident critical level
$G(K)$-action on $\widehat{\fg}_{crit}\mod_{\ord_n,naive}$.

Moreover, $\widehat{\fg}_{crit}\mod_{\ord_n,naive}$
is an $\IndCoh^*(\Op_{\ld{G}}^{\leq n})$-comodule category,
or equivalently, an $\IndCoh^!(\Op_{\ld{G}}^{\leq n})$-module
category. Because $\Op_{\ld{G}}^{\leq n}$ is the spectrum
of a polynomial algebra (on infinitely many generators),
the natural symmetric monoidal functor
$\QCoh(\Op_{\ld{G}}^{\leq n}) \to 
\IndCoh^!(\Op_{\ld{G}}^{\leq n})$ is an equivalence.
Therefore, we may as well
regard $\widehat{\fg}_{crit}\mod_{\ord_n,naive}$ as equipped
with a $\QCoh(\Op_{\ld{G}}^{\leq n})$-action commuting
with the critical level $G(K)$-action.

In our notation, we regard $G(K)$ as acting
on the left on $\widehat{\fg}_{crit}\mod_{\ord_n,naive}$
by convolution $-\star -$, and we regard
$\QCoh(\Op_{\ld{G}}^{\leq n})$ as acting on the right
by an action functor:

\[
\widehat{\fg}_{crit}\mod_{\ord_n,naive} \otimes 
\QCoh(\Op_{\ld{G}}^{\leq n}) 
\xar{ - \underset{\Op_{\ld{G}}^{\leq n}}{\otimes} - }
\widehat{\fg}_{crit}\mod_{\ord_n,naive}.
\]

\subsection{}

The following result summarizes the basic properties of the above
construction.

\begin{prop}\label{p:central-naive}

\begin{enumerate}

\item\label{i:naive-1} The functor $i_{n,*}: \widehat{\fg}_{crit}\mod_{\ord_n,naive} \to 
\widehat{\fg}_{crit}\mod$ is comonadic, and in particular,
conservative.

\item\label{i:naive-2} $\widehat{\fg}_{crit}\mod_{\ord_n,naive}$
admits a unique $t$-structure for which 
$i_{n,*}$ is $t$-exact.

\item\label{i:naive-3} The natural map:

\[
\underset{m \geq n}{\colim} \, i_{n,m}^! i_{n,m,*} \to 
i_n^!i_{n,*}
\]

\noindent is an isomorphism.

\item\label{i:naive-4} The natural functor:

\[
\underset{n,i_{n,m,*}}{\colim} \, \widehat{\fg}_{crit}\mod_{\ord_n,naive}
\to \widehat{\fg}_{crit}\mod \in \DGCat_{cont}
\]

\noindent is an equivalence.

\end{enumerate}

\end{prop}

\begin{proof}

Let $\bA^{\infty} \coloneqq \colim_r \bA^r$, i.e., the
ind-finite type indscheme version of infinite-dimensional affine space.
Using standard choices of coordinates on $\Op_{\ld{G}}$, one
find an isomorphism $\Op_{\ld{G}} = \Op_{\ld{G}}^{\leq n} \times \bA^{\infty}$
so that the diagram:

\[
\xymatrix{
\Op_{\ld{G}}^{\leq n} \ar[d]^{i_n} \ar[drr]^{\id \times 0} \\
\Op_{\ld{G}}
\ar[rr]^{\simeq} & &
\Op_{\ld{G}}^{\leq n} \times \bA^{\infty} 
}
\]

\noindent commutes.

We then have:\footnote{Of course, $\IndCoh^!(\bA^{\infty})$ and $\IndCoh^*(\bA^{\infty})$
coincide with usual $\IndCoh$ as $\bA^{\infty}$ is locally of finite type.
We include the notation to clarify whether this category is being
viewed as an algebra or coalgebra in $\DGCat_{cont}$.}

\[
\TwoHom_{\IndCoh^!(\bA^{\infty})\mod}(\Vect,
\widehat{\fg}_{crit}\mod) =
\Vect
\overset{\IndCoh^*(\bA^{\infty})}{\otimes} 
\widehat{\fg}_{crit}\mod \isom 
\widehat{\fg}_{crit}\mod_{\ord_n,naive} \in G(K)\mod_{crit}.
\]

Take $\sA \coloneqq \IndCoh^!(\bA^{\infty})$ as a monoidal category.
Note that the monoidal product:

\[
-\overset{!}{\otimes}-:
\sA \otimes \sA \isom 
\IndCoh^!(\bA^{\infty} \times \bA^{\infty}) \xar{\Delta^!}
\IndCoh^!(\bA^{\infty}) = \sA
\]

\noindent admits a left adjoint $\Delta_*^{\IndCoh}$ that is a morphism
of $\sA$-bimodule categories (by the projection formula).
It is easy to see in this setting that for any 
$\sA$-module category $\sM$, the action functor:

\[
\act:\sA \otimes \sM \to \sM
\]

\noindent admits a continuous left adjoint $\act^L$ that is a morphism
of $\sA$-module categories, where the left hand side is regarded
as an $\sA$-module via the action on the first factor.
It follows that for any pair of $\sA$-module categories
$\sM,\sN$, the cosimplicial category:

\[
\TwoHom_{\DGCat_{cont}}(\sM,\sN) \rightrightarrows
\TwoHom_{\DGCat_{cont}}(\sA \otimes \sM,\sN) \rightrightrightarrows \ldots
\]

\noindent satisfies the comonadic Beck-Chevalley conditions.\footnote{See 
\cite{higheralgebra} \S 4.7.6 or \cite{shvcat} \S C for background
on the Beck-Chevalley theory; our terminology here
is taken from the latter source. We especially note \cite{shvcat}
Lemma C.2.2, which is essentially dual to the present assertion.}
Applying this for $\sM = \Vect$ and $\sN = \widehat{\fg}_{crit}\mod$,
we obtain \eqref{i:naive-1}.

Next, we show \eqref{i:naive-4}. We calculate:

\[
\begin{gathered}
\underset{n,i_{n,m,*}}{\colim} \, \widehat{\fg}_{crit}\mod_{\ord_n,naive}
= 
\underset{n,i_{n,m}^!}{\lim} \, \widehat{\fg}_{crit}\mod_{\ord_n,naive} =
\\
\underset{n}{\lim} \, \Big( \IndCoh^*(\Op_{\ld{G}}^{\leq n}) 
\overset{\IndCoh^*(\Op_{\ld{G}})}{\otimes} 
\widehat{\fg}_{crit}\mod\Big) = \\
\underset{n}{\lim} \, \Tot\Big(
\IndCoh^*(\Op_{\ld{G}}^{\leq n}) \otimes \IndCoh^*(\Op_{\ld{G}})^{\otimes \dot}
\otimes \widehat{\fg}_{crit}\mod\Big) = \\
 \Tot\underset{n}{\lim} \Big(
\IndCoh^*(\Op_{\ld{G}}^{\leq n}) \otimes \IndCoh^*(\Op_{\ld{G}})^{\otimes \dot}
\otimes \widehat{\fg}_{crit}\mod\Big) \overset{\star}{=} \\
\Tot \Big(
\big(\underset{n}{\lim} \, \IndCoh^*(\Op_{\ld{G}}^{\leq n})\big) 
\otimes \IndCoh^*(\Op_{\ld{G}})^{\otimes \dot}
\otimes \widehat{\fg}_{crit}\mod\big) = \\
\Tot \Big(
\IndCoh^*(\Op_{\ld{G}})\big) 
\otimes \IndCoh^*(\Op_{\ld{G}})^{\otimes \dot}
\otimes \widehat{\fg}_{crit}\mod\Big) = \\
\IndCoh^*(\Op_{\ld{G}}) 
\overset{\IndCoh^*(\Op_{\ld{G}})}{\otimes} 
\widehat{\fg}_{crit}\mod = \widehat{\fg}_{crit}\mod
\end{gathered}
\]

\noindent as desired; here the only non-trivial manipulations
are the first, which expresses that a colimit in $\DGCat_{cont}$ under left 
adjoints is canonically isomorphic to the limit under right adjoints,
and the one labeled $\star$, where the limit past tensor
products is justified because we are tensoring with compactly 
generated, hence dualizable, DG categories.

We deduce \eqref{i:naive-3} immediately from \eqref{i:naive-4} and
\cite{dgcat} Lemma 1.3.6. 

It remains to show \eqref{i:naive-2}. Given \eqref{i:naive-1},
a standard argument reduces us to checking that
$i_{n,*}i_n^!$ is left $t$-exact. 

By the above Beck-Chevalley analysis,
$i_{n,*}i_n^!$ may be calculated by applying the composition:

\[
\widehat{\fg}_{crit}\mod \xar{\coact}
\IndCoh^*(\Op_{\ld{G}}) \otimes \widehat{\fg}_{crit}\mod 
\xar{\pi_*^{\IndCoh} \otimes \id}
\IndCoh^*(\bA^{\infty}) \otimes \widehat{\fg}_{crit}\mod
\]

\noindent and then applying the right adjoint to this composition;
here $\pi:\Op_{\ld{G}} \to \bA^{\infty}$ denotes the projection.
It suffices to show the composition is
$t$-exact (for the tensor product $t$-structure on the right hand side);
we will show each of the functors appearing here is $t$-exact.
The functor $\coact$ is $t$-exact by 
\cite{methods} Lemma 11.13.1. 
Then $\pi_*^{\IndCoh}$ is $t$-exact because $\pi$ is affine,
and similarly for
$\pi_*^{\IndCoh} \otimes \id$ by \cite{whit} Lemma B.6.2.

\end{proof}

\subsection{} 

We continue our study of $\widehat{\fg}_{crit}\mod_{\ord_n,naive}$.

\begin{lem}\label{l:reg-ab}

Suppose $\sF \in \widehat{\fg}_{crit}\mod^{\heart}$. Then the adjunction map
$H^0(i_{n,*}i_n^!(\sF)) \to \sF \in \widehat{\fg}_{crit}\mod^{\heart}$
is a monomorphism with image the maximal submodule of
$\sF$ on which $\fZ$ acts through $\fz_n$.

\end{lem}

\begin{proof}

The forgetful functor $\widehat{\fg}_{crit}\mod \to \Vect$ admits
a unique lift $\widehat{\fg}_{crit}\mod \xar{\Oblv^{enh}} 
\IndCoh^*(\Op_{\ld{G}}) = \fZ\mod_{ren} \xar{\Oblv} \Vect$ 
with $\Oblv^{enh}$ a morphism of $\IndCoh^*(\Op_{\ld{G}})$-comodule
categories. By \cite{methods} Lemma 11.13.1, $\Oblv^{enh}$ is 
$t$-exact, and on the hearts of the $t$-structure corresponds
to restriction of modules along the homomorphism 
$\fZ \into U(\widehat{\fg}_{crit})$.

As $\Oblv^{enh}$ is a map of $\IndCoh^*(\Op_{\ld{G}})$-comodule categories,
it intertwines $i_{n,*}i_n^!$ with the similar
functor on $\IndCoh^*(\Op_{\ld{G}})$. It is clear that
$H^0$ of that functor extracts the maximal submodule on which
$\fZ$ acts through $\fz_n$, giving the claim.

\end{proof}

\begin{cor}\label{c:reg-ab}

The map $\widehat{\fg}_{crit}\mod_{\ord_n,naive}^{\heart}
\to \widehat{\fg}_{crit}\mod^{\heart}$ is fully faithful.
Its essential image is the full subcategory of the target
consisting of modules on which $\fZ$ acts through
$\fz_n$.

\end{cor}

\begin{proof}

Immediate from Lemma \ref{l:reg-ab} and 
Proposition \ref{p:central-naive} \eqref{i:naive-1}.

\end{proof}

\subsection{}

We use the notation: 

\[
\sP \coloneqq
\underset{m}{\lim} \, \bV_{crit,m} 
\in \Pro(\widehat{\fg}_{crit}\mod^{\heart}) \subset
\Pro(\widehat{\fg}_{crit}).
\]

\noindent Here the limit is over the natural structure
maps $\bV_{crit,m+1} \to \bV_{crit,m}$, and we emphasize that the 
limit occurs in the pro-category (or rather, in either pro-category).
We remark that the pro-object $\sP$ corepresents the forgetful
functor $\Oblv:\widehat{\fg}_{crit}\mod \to \Vect$: this is
clear of its restriction
to $\widehat{\fg}_{crit}\mod^+$, and then
the claim follows generally as the objects $\bV_{crit,m}$ are compact
in $\widehat{\fg}_{crit}\mod$. Clearly $\Oblv(\sP) \in \Pro(\Vect^{\heart})$
is $U(\widehat{\fg}_{crit})$; its $\overset{\rightarrow}{\otimes}$-algebra
structure may be seen using \cite{methods} Proposition 3.7.1. 

For $m \geq 0$, let $\bV_{\ord_n,m} \in 
\widehat{\fg}_{crit}\mod_{\ord_n,naive}^{\heart}$
denote the minimal quotient of $\bV_{crit,m}$ lying in
$\widehat{\fg}_{crit}\mod_{\ord_n,naive}^{\heart} \subset
\widehat{\fg}_{crit}\mod^{\heart}$, i.e.,
$\bV_{\ord_n,m} = \bV_{crit,m}/I_n$.

Define: 

\[
\sP_{\ord_n} \coloneqq
\underset{m}{\lim} \, \bV_{\ord_n,m} \in 
\Pro(\widehat{\fg}_{crit}\mod_{\ord_n,naive}^{\heart}) \subset 
\Pro(\widehat{\fg}_{crit}\mod_{\ord_n,naive})
\]

\noindent to be the corresponding pro-object; we again emphasize that the
displayed limit occurs in the pro-category.

There is an evident canonical morphism:

\[
\pi:\sP \to i_{n,*}\sP_{\ord_n} \in 
\Pro(\widehat{\fg}_{crit}\mod^{\heart}) \subset
\Pro(\widehat{\fg}_{crit}).
\]

\begin{lem}\label{l:pro-reg}

As an object of $\Pro(\widehat{\fg}_{crit}\mod_{\ord_n,naive}^+)$,
$\sP_{\ord_n}$ corepresents the composition:

\[
\widehat{\fg}_{crit}\mod_{\ord_n,naive}^+ \xar{i_{n,*}}
\widehat{\fg}_{crit}\mod^+ \xar{\Oblv} \Vect.
\]

More precisely, for $\sF \in \widehat{\fg}_{crit}\mod_{\ord_n,naive}^+$,
the composite map:

\[
\begin{gathered}
\ul{\Hom}_{\Pro(\widehat{\fg}_{crit}\mod_{\ord_n,naive})}
(\sP_{\ord_n},\sF) \to 
\ul{\Hom}_{\Pro(\widehat{\fg}_{crit}\mod)}
(i_{n,*}\sP_{\ord_n},i_{n,*}\sF) \xar{-\circ\pi} \\
\ul{\Hom}_{\Pro(\widehat{\fg}_{crit}\mod)}
(\sP,i_{n,*}\sF) \simeq \Oblv(i_{n,*}\sF) 
\end{gathered}
\]

\noindent is an isomorphism.

\end{lem}

\begin{proof}

\step\label{st:pro-reg-1}

First, suppose $\sG \in \widehat{\fg}_{crit}\mod_{\ord_n,naive}^{\geq 0}$
has the property that $i_{n,*}\sG$ is compact in 
$\widehat{\fg}_{crit}\mod$. Then we claim that for any $r \geq 0$,
$\sG$ is compact as an 
object of the category 
$\widehat{\fg}_{crit}\mod_{\ord_n,naive}^{\geq -r}$.

Indeed, this is standard from Proposition \ref{p:central-naive} 
\eqref{i:naive-1}-\eqref{i:naive-2}: 
see the proof of \cite{methods} Lemma 6.11.2.

\step\label{st:pro-reg-2}

Suppose $\sG$ as above, and let $\sF \in 
\widehat{\fg}_{crit}\mod_{\ord_n,naive}^+$. Then we claim that
the natural map:

\begin{equation}\label{eq:hom-colim-ord_n}
\underset{m \geq n}{\colim} \, 
\ul{\Hom}_{\widehat{\fg}_{crit}\mod_{\ord_m,naive}}
(i_{n,m,*}\sG,i_{n,m,*}\sF) \to 
\ul{\Hom}_{\widehat{\fg}_{crit}\mod}
(i_{n,*}\sG,i_{n,*}\sF)  
\end{equation}

\noindent is an isomorphism.

Indeed, we have:

\[
\begin{gathered}
\underset{m \geq n}{\colim} \, 
\ul{\Hom}_{\widehat{\fg}_{crit}\mod_{\ord_m,naive}}
(i_{n,m,*}\sG,i_{n,m,*}\sF) = \\
\underset{m \geq n}{\colim} \, 
\ul{\Hom}_{\widehat{\fg}_{crit}\mod_{\ord_n,naive}}
(\sG,i_{n,m}^!i_{n,m,*}\sF) \overset{\text{Step \ref{st:pro-reg-1}}}{=}
\\
\ul{\Hom}_{\widehat{\fg}_{crit}\mod_{\ord_n,naive}}
(\sG,\underset{m \geq n}{\colim} \, i_{n,m}^!i_{n,m,*}\sF) 
\overset{\text{Prop. \ref{p:central-naive} \eqref{i:naive-3}}}
{=} \\
\ul{\Hom}_{\widehat{\fg}_{crit}\mod_{\ord_n,naive}}
(\sG,i_n^!i_{n,*}\sF) =  
\ul{\Hom}_{\widehat{\fg}_{crit}\mod}
(i_{n,*}\sG,i_{n,*}\sF).  
\end{gathered}
\]

\noindent We remark that if 
$\sF$ is in cohomological degrees $\geq -r$,
then each $i_{n,m}^!i_{n,m,*}(\sF)$ is as well (because the functors
$i_{n,m,*}$ are $t$-exact); this justifies the reference
to Step \ref{st:pro-reg-1}.
We also note that the composite identification here is easily seen to be
given by the map considered above.

\step\label{st:pro-reg-3}

Next, recall the functors $i_{n,m}^*$ from \S \ref{ss:naive-defin}.
We claim that $i_{n,m}^*(\bV_{crit,m}) = \bV_{\ord_n,m}$.
Clearly the right hand side is the top (= degree $0$) 
cohomology of the left hand side, so this amounts to arguing that
the lower cohomology groups vanish.

As in the argument for Lemma \ref{l:reg-ab}, we have a commutative 
diagram:\footnote{To be explicit, we remind that by the definition
from \cite{methods} \S 6, 
$\IndCoh^*(\Op_{\ld{G}}^{\leq n})$ is $\Ind(Coh(\Op_{\ld{G}}^{\leq n}))$.
As $\Op_{\ld{G}}^{\leq n}$ is the spectrum of a (infinitely generated)
polynomial algebra,
$\Coh(\Op_{\ld{G}}^{\leq n}) = \Perf(\Op_{\ld{G}}^{\leq n})$.
Therefore, $\IndCoh^*$ in the bottom row may be replaced by the more familiar
$\QCoh$. The functor $i_{n,m}^*$ in that bottom row is then the standard
pullback functor.}

\[
\xymatrix{
\widehat{\fg}_{crit}\mod_{\ord_m,naive} \ar[d] \ar[rr]^{i_{n,m}^*} && 
\widehat{\fg}_{crit}\mod_{\ord_n,naive} \ar[d] \\
\IndCoh^*(\Op_{\ld{G}}^{\leq m}) \ar[rr]^{i_{n,m}^*} &&
\IndCoh^*(\Op_{\ld{G}}^{\leq n}).
}
\]

\noindent The vertical arrows are the natural restriction maps,
and arise from $\Oblv^{enh}$ (from the proof of Lemma \ref{l:reg-ab}) 
and the evident identification
$\IndCoh^*(\Op_{\ld{G}}^{\leq n}) = 
\IndCoh^*(\Op_{\ld{G}}^{\leq n}) \overset{\IndCoh^*(\Op_{\ld{G}})}{\otimes}
\IndCoh^*(\Op_{\ld{G}})$, and similarly for $m$. 
These vertical arrows are $t$-exact
and conservative on bounded below subcategories as this
is true for $\Oblv^{enh}$.

The functor $i_{n,m}^*:\widehat{\fg}_{crit}\mod_{\ord_m,naive} \to 
\widehat{\fg}_{crit}\mod_{\ord_n,naive}$ is easily\footnote{For
one, it is (non-canonically) isomorphic to $i_{n,m}^!$ up to shift.
Alternatively, $i_{n,m,*}i_{n,m}^*$ is calculated as the composition:

\[
\begin{gathered}
\widehat{\fg}_{crit}\mod_{\ord_m,naive} \xar{\coact}
\IndCoh^*(\Op_{\ld{G}}^{\leq m}) 
\otimes \widehat{\fg}_{crit}\mod_{\ord_m,naive} 
\xar{i_{n,m,*}i_{n,m}^* \otimes \id} \\
\IndCoh^*(\Op_{\ld{G}}^{\leq m}) 
\otimes \widehat{\fg}_{crit}\mod_{\ord_m,naive} 
\xar{\Gamma^{\IndCoh}(\Op_{\ld{G}}^{\leq m},-) \otimes \id}
\widehat{\fg}_{crit}\mod_{\ord_m,naive}
\end{gathered}
\]

\noindent giving the claim by considering the standard
finite Koszul filtration on the endofunctor
$i_{n,m,*}i_{n,m}^*$ of $\IndCoh^*(\Op_{\ld{G}}^{\leq m})$.}
seen to be left $t$-exact up to shift. Therefore, it suffices
to see that the underlying object of
$\IndCoh^*(\Op_{\ld{G}}^{\leq n})$ 
defined by $i_{n,m}^*(\bV_{crit,m})$ lies in cohomological
degree $0$. 

This follows from the commutativity of the above
diagram and the fact that $\bV_{crit,m} \in 
\IndCoh^*(\Op_{\ld{G}}^{\leq m})^{\heart}$ defines a 
\emph{flat} sheaf by \cite{fg2} Lemma 7.2.2 (which is 
based on \cite{frenkel-eisenbud}).

\step 

We now deduce the claim.

In what follows, we consider $\bV_{crit,m}$ as an object
of $\widehat{\fg}_{crit}\mod_{\ord_m,naive}^{\heart}$; 
we let $i_{m,*} \bV_{crit,m}$ denote the corresponding
object of $\widehat{\fg}_{crit}\mod^{\heart}$.

For $\sF \in \widehat{\fg}_{crit}\mod_{\ord_n,naive}^+$, we calculate:

\[
\begin{gathered}
\ul{\Hom}_{\Pro(\widehat{\fg}_{crit}\mod_{\ord_n,naive})}
(\sP_{\ord_n},\sF) = 
\underset{m \geq n}{\colim} \,
\ul{\Hom}_{\widehat{\fg}_{crit}\mod_{\ord_n,naive}}
(\bV_{\ord_n,m},\sF) \overset{\text{Step \ref{st:pro-reg-3}}}{=} \\
\underset{m \geq n}{\colim} \,
\ul{\Hom}_{\widehat{\fg}_{crit}\mod_{\ord_n,naive}}
(i_{n,m}^*\bV_{crit,m},\sF) = 
\underset{m \geq n}{\colim} \,
\ul{\Hom}_{\widehat{\fg}_{crit}\mod_{\ord_m,naive}}
(\bV_{crit,m},i_{n,m,*}\sF) = \\
\underset{m \geq n}{\colim} \, \underset{r \geq m}{\colim} \,
\ul{\Hom}_{\widehat{\fg}_{crit}\mod_{\ord_r,naive}}
(i_{m,r,*}\bV_{crit,m},i_{n,r,*}\sF) 
\overset{\text{Step \ref{st:pro-reg-2}}}{=} \\
\underset{m \geq n}{\colim} \,
\ul{\Hom}_{\widehat{\fg}_{crit}\mod}
(i_{m,*}\bV_{crit,m},i_{n,*}\sF) =
\ul{\Hom}_{\Pro(\widehat{\fg}_{crit}\mod)}
(\sP,i_{n,*}\sF)
\end{gathered}
\]

\noindent as desired.

\end{proof}

In what follows, we let
$\Oblv:\widehat{\fg}_{crit}\mod_{\ord_n,naive} \to \Vect$
denote the forgetful functor considered above, i.e.,
$\Oblv i_{n,*}$.

\begin{cor}

The (non-cocomplete) 
DG category $\widehat{\fg}_{crit}\mod_{\ord_n,naive}^+$
is the bounded below derived category of its heart.

\end{cor}

\begin{proof}

Note that $U(\widehat{\fg}_{crit})_{\ord_n} \coloneqq 
\Oblv(\sP_{\ord_n}) \in \Pro(\Vect^{\heart})$ by 
construction. Therefore, the result follows from 
\cite{methods} Proposition 3.7.1. 

\end{proof}

It follows that $\widehat{\fg}_{crit}\mod_{\ord_n,naive}^+$
identifies with the similar category considered in the works
of Frenkel-Gaitsgory, e.g. in \cite{fg2} \S 23.

\subsection{Renormalization}

We now introduce a renormalized version of the 
above categories following \cite{dmod-aff-flag} \S 23.

Define $\widehat{\fg}_{crit}\mod_{\ord_n}^c \subset
\widehat{\fg}_{crit}\mod_{\ord_n,naive}$ as the full subcategory
of objects $\sF$ such that
$i_{n,*}(\sF)$ is compact in $\widehat{\fg}_{crit}\mod$.
By Proposition \ref{p:central-naive} and the similar fact
for $\widehat{\fg}_{crit}\mod$,
$\widehat{\fg}_{crit}\mod_{\ord_n}^c \subset 
\widehat{\fg}_{crit}\mod_{\ord_n,naive}^+$.

\begin{example}

For $m \geq n$, Koszul resolutions for the
finitely presented regular embedding
$\Op_{\ld{G}}^{\leq n} \into \Op_{\ld{G}}^{\leq m}$
imply that the functors
$i_{n,m}^!$ and $i_{n,m}^*$ map
$\widehat{\fg}_{crit}\mod_{\ord_m}^c$
to $\widehat{\fg}_{crit}\mod_{\ord_n}^c$.

\end{example}

\begin{example}

The objects $\bV_{\ord_n,m}$ lie in
$\widehat{\fg}_{crit}\mod_{\ord_n}^c$.
Indeed, for $0 \leq m \leq n$, 
$i_{n,*}\bV_{\ord_n,m} = \bV_{crit,m}$, clearly
giving the claim in this case. In general,
for $m \geq n$, we have $i_{n,m}^*\bV_{crit,m} = 
\bV_{\ord_n,m}$ as in Step \ref{st:pro-reg-3} from
the proof of Lemma \ref{l:pro-reg}, clearly
giving the claim.

\end{example}

Define $\widehat{\fg}_{crit}\mod_{\ord_n} = 
\Ind(\widehat{\fg}_{crit}\mod_{\ord_n}^c)$.
Define a $t$-structure on 
$\widehat{\fg}_{crit}\mod_{\ord_n}$ by taking
$\widehat{\fg}_{crit}\mod_{\ord_n}^{\leq 0}$ to be generated
under colimits by objects in 
$\widehat{\fg}_{crit}\mod_{\ord_n}^c \cap 
\widehat{\fg}_{crit}\mod_{\ord_n,naive}^{\leq 0}$.

We have a canonical functor
$\rho:\widehat{\fg}_{crit}\mod_{\ord_n} \to 
\widehat{\fg}_{crit}\mod_{\ord_n,naive}$: this is the unique
continuous functor with $\rho|_{\widehat{\fg}_{crit}\mod_{\ord_n}^c}$
the canonical embedding.

\begin{lem}[C.f. \cite{dmod-aff-flag} \S 23.2.2]\label{l:ren}

The functor $\rho$ is $t$-exact and induces an equivalence
on eventually coconnective subcategories.

\end{lem}

\begin{proof}

\step\label{st:ren-1} 

We collect some observations we will
need later.

Note that for any $m$,
$\bV_{\ord_n,m} \in \widehat{\fg}_{crit}\mod_{\ord_n}^c \subset
\widehat{\fg}_{crit}\mod_{\ord_n}$ lies in the heart of the $t$-structure;
indeed, it is connective by definition, and it is clear that
any object in $\widehat{\fg}_{crit}\mod_{\ord_n}^c$ 
that is coconnective in $\widehat{\fg}_{crit}\mod_{\ord_n,naive}$
is also coconnective in $\widehat{\fg}_{crit}\mod_{\ord_n}$.

In addition, the canonical map $\bV_{\ord_n,m+1} \to \bV_{\ord_n,m}$ is
an epimorphism in $\widehat{\fg}_{crit}\mod_{\ord_n}^{\heart}$.
Indeed, it suffices to show that the (homotopy) kernel of this
map is in cohomological degree $0$, and the above logic applies just
as well to see this.

\step\label{st:ren-2} Define $\Oblv:\widehat{\fg}_{crit}\mod_{\ord_n} \to \Vect$
as $\Oblv\circ \rho$. We claim that 
$\Oblv|_{\widehat{\fg}_{crit}\mod_{\ord_n}^+}$ is conservative and $t$-exact.

Suppose 
$\sF \in \widehat{\fg}_{crit}\mod_{\ord_n}^{\geq 0}$ with
$\Oblv(\sF) = 0$; it suffices 
to show that $H^0(\sF) = 0$. 
To this end, it suffices to 
show that any morphism 
$\eta:\sG \to \sF$ is nullhomotopic for
a connective object
$\sG \in \widehat{\fg}_{crit}\mod_{\ord_n}^c $.

Note that the top cohomology group $H^0(\sG)$ 
is finitely generated as a module
over $U(\widehat{\fg}_{crit})$, say by $v_1,\ldots,v_N \in H^r(\sG)$.
By Lemma \ref{l:pro-reg}, for each $i = 1, \ldots,N$ 
we can find $m_i \gg 0$ and a map
$\alpha_i:\bV_{\ord_n,m_i} \to \sG$ such that $H^0(\alpha_i)$
maps the vacuum vector in $\bV_{\ord_n,m_i}$ to $v_i$. 

Let $\alpha:\oplus_{i=1}^N \bV_{\ord_n,m_i} \to \sG$
be the induced map;
$\alpha$ is surjective on $H^0$ by design,
so $\Coker(\alpha)$ is in cohomological degrees 
$\leq -1$. It follows that $\sG \to \sF$
is nullhomotopic if and only if its composition
with $\alpha$ is. Therefore, it suffices to show
that any map $\bV_{\ord_n,m} \to \sF$ 
is nullhomotopic.

The map:

\[
\begin{gathered}
H^0\big(\ul{\Hom}_{\widehat{\fg}_{crit}\mod_{\ord_n}}
(\bV_{\ord_n,m},\sF)\big) =
\Hom_{\widehat{\fg}_{crit}\mod_{\ord_n}^{\heart}}
(\bV_{\ord_n,m},H^0(\sF)) \to \\
\Hom_{\widehat{\fg}_{crit}\mod_{\ord_n}^{\heart}}
(\bV_{\ord_n,m+1},H^0(\sF)) = 
H^0\big(\ul{\Hom}_{\widehat{\fg}_{crit}\mod_{\ord_n}}
(\bV_{\ord_n,m+1},\sF)\big) \in \Vect^{\heart}
\end{gathered}
\]

\noindent is injective by Step \ref{st:ren-1}.
But we have:

\[
\underset{m}{\colim} \, 
\ul{\Hom}_{\widehat{\fg}_{crit}\mod_{\ord_n}}
(\bV_{\ord_n,m},\sF) = \Oblv(\sF) = 0
\]

\noindent by Lemma \ref{l:pro-reg} (and compactness
of $\bV_{\ord_n,m}$), 
giving the claim.

\step\label{st:ren-3} 

We now show $t$-exactness of $\rho$.
Right $t$-exactness follows immediately from the
construction, so we show left $t$-exactness.

Let $m \geq n$ be fixed. It what follows, we 
regard $\bV_{crit,m}$ as an object of 
$\widehat{\fg}_{crit}\mod_{\ord_m}^c \subset 
\widehat{\fg}_{crit}\mod_{\ord_m,naive}$.

As $r \geq m$ varies, we have natural maps: 

\[
\ldots \to i_{n,r+1}^*i_{m,r+1,*}(\bV_{crit,m}) \to 
i_{n,r}^*i_{m,r,*}\bV_{crit,m} \to \ldots \to 
i_{n,m}^* \bV_{crit,m} \in \widehat{\fg}_{crit}\mod_{\ord_n}^c.
\]

\noindent We claim that for 
$\sF \in \widehat{\fg}_{crit}\mod_{\ord_n}$, the natural map:

\begin{equation}\label{eq:vm-reg-hom}
\underset{r}{\colim} \, 
\ul{\Hom}_{\widehat{\fg}_{crit}\mod_{\ord_n}}
(i_{n,r}^*i_{m,r,*}\bV_{crit,m},\sF) \to
\ul{\Hom}_{\widehat{\fg}_{crit}\mod}(\bV_{crit,m},i_{n,*} \rho(\sF)) 
\end{equation} 

\noindent is an isomorphism. Indeed, both sides commute with
colimits in $\sF$ by compactness, so we are reduced to the
case where $\sF \in \widehat{\fg}_{crit}\mod_{\ord_n}^c$. 
For such $\sF$, the claim follows from \eqref{eq:hom-colim-ord_n}.

Now suppose that $\sF \in \widehat{\fg}_{crit}\mod_{\ord_n}^{\geq 0}$.
As each object $i_{n,r}^*i_{m,r,*}\bV_{crit,m}$ is
connective in $\widehat{\fg}_{crit}\mod_{\ord_n}$,
\eqref{eq:vm-reg-hom} implies that
$\ul{\Hom}_{\widehat{\fg}_{crit}\mod}(\bV_{crit,m},i_{n,*}\rho(\sF)) \in 
\Vect^{\geq 0}$. As the objects $\bV_{crit,m}$ generate
$\widehat{\fg}_{crit}\mod$ under colimits, this implies that
$i_{n,*}\rho(\sF)$ lies in $\widehat{\fg}_{crit}\mod^{\geq 0}$,
i.e., $i_{n,*}\rho$ is left $t$-exact.

Finally, as $i_{n,*}$ is $t$-exact and conservative
by Proposition \ref{p:central-naive}, $\rho$ itself must be left $t$-exact.

\step Finally, we show that $\rho$ induces an equivalence on 
eventually coconnective subcategories.

By $t$-exactness of $\rho$, we have a commutative diagram:

\[
\xymatrix{
\widehat{\fg}_{crit}\mod_{\ord_n}^+ \ar[rr]^{\rho} \ar[dr] & &
\widehat{\fg}_{crit}\mod_{\ord_n,naive}^+ \ar[dl] \\
& \Vect^+
}
\]

\noindent with the diagonal arrows being the forgetful functors.
Each of these functors is conservative. 

Moreover, the forgetful functor 
$\widehat{\fg}_{crit}\mod_{\ord_n} \to \Vect$ is corepresented
by the pro-object:

\[
\underset{m}{\lim} \, \bV_{\ord_n,m} \in 
\Pro(\widehat{\fg}_{crit}\mod_{\ord_n}^c) \subset
\Pro(\widehat{\fg}_{crit}\mod_{\ord_n}).
\]

\noindent Indeed, this follows immediately from
Lemma \ref{l:pro-reg} and compactness of
$\bV_{\ord_n,m} \in \widehat{\fg}_{crit}\mod_{\ord_n}$.

Applying Lemma \ref{l:pro-reg} again, we see that
$\rho:\widehat{\fg}_{crit}\mod_{\ord_n}^+ \to 
\widehat{\fg}_{crit}\mod_{\ord_n,naive}^+$ intertwines
the pro-left adjoints to the forgetful functors in the
above diagram. Therefore, it induces an equivalence on the
corresponding $\overset{\rightarrow}{\otimes}$-algebras,
so we obtain the claim from \cite{methods} Proposition 3.7.1. 

\end{proof}

\begin{rem}

Unlike $\widehat{\fg}_{crit}\mod$, we 
are not aware of an explicit description of 
compact generators of $\widehat{\fg}_{crit}\mod_{\ord_n}$.
For instance, does $\widehat{\fg}_{crit}\mod_{\ord_n}$ admit
compact generators that admit weakly
$G(O)$-equivariant structures? Does it admit compact generators
lying in $\widehat{\fg}_{crit}\mod_{\ord_n}^{\heart}$?
(For $G = PGL_2$ and $n = 0$, the answer to both questions is
yes by Theorem \ref{t:main}.)

This general issue is closely
related to the technical problems highlighted in 
\S \ref{ss:eq-renorm}.

\end{rem}

\subsection{Equivariant renormalization}\label{ss:eq-renorm}

We now highlight a technical problem: there is not an evident
critical level $G(K)$-action on $\widehat{\fg}_{crit}\mod_{\ord_n}$.
(Similarly, we cannot construct a weak $G(K)$-action 
in the sense of \cite{methods}.)

\begin{conjecture}\label{conj:ord-n-conv}

For any $\sF \in D_{crit}^*(G(K))$ compact, define a functor:

\[
\chi_{\sF}:\widehat{\fg}_{crit}\mod_{\ord_n} \to 
\widehat{\fg}_{crit}\mod_{\ord_n}
\]

\noindent whose restriction to 
$\widehat{\fg}_{crit}\mod_{\ord_n}^c$ is calculated as the
composition:

\[
\widehat{\fg}_{crit}\mod_{\ord_n}^c \subset 
\widehat{\fg}_{crit}\mod_{\ord_n,naive}^+ \xar{\sF \star -}
\widehat{\fg}_{crit}\mod_{\ord_n,naive}^+ \simeq 
\widehat{\fg}_{crit}\mod_{\ord_n}^+ \subset 
\widehat{\fg}_{crit}\mod_{\ord_n}.
\]

Then we conjecture that $\chi_{\sF}$ is left $t$-exact up to shift.

\end{conjecture}

\begin{rem}

Assuming Conjecture \ref{conj:ord-n-conv}
if $K$ is prounipotent, say, then
we obtain $\widehat{\fg}_{crit}\mod_{\ord_n}^K \subset 
\widehat{\fg}_{crit}\mod_{\ord_n}$ as the
essential image of $\chi_{\delta_K}$.
Without assuming the conjecture, we are not otherwise aware of a good
definition of $\widehat{\fg}_{crit}\mod_{\ord_n}^K$.

\end{rem}

\begin{rem}

In the language of \cite{methods} \S 4.4, the above conjecture
asserts that 
the functor $\sF \star -:\widehat{\fg}_{crit}\mod_{\ord_n,naive} 
\to \widehat{\fg}_{crit}\mod_{\ord_n,naive}$ \emph{renormalizes}.

\end{rem}

\begin{rem}

Suppose Conjecture \ref{conj:ord-n-conv} holds for a reductive group $G$
and an integer $n \geq 0$. Then there exists a unique
critical level $G(K)$-action on $\widehat{\fg}_{crit}\mod_{\ord_n}$
such that:

\begin{itemize}

\item The functor $\rho$ upgrades to a morphism of categories
with critical level $G(K)$-actions.

\item The (critical level) $G(K)$-action on 
$\widehat{\fg}_{crit}\mod_{\ord_n}$ is strongly compatible
with the $t$-structure in the sense of \cite{methods} \S 10.12.

\end{itemize}

Indeed, this is essentially immediate from \cite{methods} Lemma 
8.16.4. 

\end{rem}

\begin{rem}

The technical issue associated with the above conjecture 
appears implicitly in \cite{dmod-aff-flag}.

In \S 4.1.4 of \emph{loc. cit}., Frenkel and Gaitsgory suggest
a definition of $\widehat{\fg}_{crit}\mod_{\ord_n}^K$ (adapted
to their particular setting). 
But their definition is not clearly
a good one: for example, it is not clear that 
their category carries the expected Hecke symmetries.
This issue is discussed somewhat in the remark in that same section.
Related to that discussion, 
Main Theorem 2 from \emph{loc. cit}. in effect verifies the
above conjecture in a special case. 

Combined with our proof of Theorem \ref{t:pgl2-action}, 
it may be fair to expect verifying Conjecture \ref{conj:ord-n-conv}
in a given instance requires substantial 
input from local geometric Langlands.

\end{rem}

As an immediate consequence of our main theorem, 
Theorem \ref{t:main}, we may deduce:

\begin{thm}\label{t:pgl2-action}

Suppose $G = PGL_2$ and $n = 0$. 
Then Conjecture \ref{conj:ord-n-conv} holds.

\end{thm}

Conversely, if we a priori knew Theorem \ref{t:pgl2-action}, then
the proof that the functor in Theorem \ref{t:main} is an equivalence
could be substantially simplified: the proof of 
Lemma \ref{l:gamma-gens}
would be applicable and would directly give the essential
surjectivity of $\Gamma^{\Hecke}$ (c.f. the
outline from \S \ref{ss:sketch}).

\section{The localization theorem}\label{s:localization}

\subsection{} 

This section begins our study of the Frenkel-Gaitsgory
conjecture.

First, we recall the constructions
underlying the Frenkel-Gaitsgory localization conjecture,
following \cite{fg2}.
We include more attention to derived issues than \emph{loc. cit}.,
so our discussion distinguishes between naive and renormalized
categories of regular Kac-Moody modules.

We then formulate our main result, Theorem \ref{t:main}.

Next, we recall the main results of Frenkel-Gaitsgory.
We include some details on how to deduce the corresponding
results in the DG framework from the exact results that they
showed.

Finally, in \S \ref{ss:intermediate-results}, we formulate
three lemmas from which we deduce Theorem \ref{t:main}. 
The proofs of these lemmas occupy the remainder of the
section.

\subsection{Regular Kac-Moody representations}

In the setting of \S \ref{s:central}, for $n = 0$, we prefer
the notation $reg$ to $\ord_0$. So we let:

\[
\begin{gathered}
\widehat{\fg}_{crit}\mod_{reg,naive} \coloneqq 
\widehat{\fg}_{crit}\mod_{\ord_0,naive} \\
\widehat{\fg}_{crit}\mod_{reg} \coloneqq 
\widehat{\fg}_{crit}\mod_{\ord_0} 
\end{gathered}.
\]

We highlight that the subscript \emph{reg} is being used
in a completely different way than it was in \S \ref{s:heisenberg}.
In the Kac-Moody context, this terminology rather 
follows \cite{fg1}. (We believe that this point should not
cause confusion in navigating the paper.)

Finally, we let $\bV_{crit} \coloneqq \bV_{0,crit}$ denote
the critical level vacuum representation.

\subsection{Notation regarding geometric Satake}

Let $\sH_{sph} = D_{crit}(\Gr_G)^{G(O)}$, considered as a monoidal
category via convolution. Recall that for any $\sC \in G(K)\mod_{crit}$,
there is a canonical action of $\sH_{sph}$ on $\sC^{G(O)}$ coming
from the identifications $\sH_{sph} = \TwoEnd_{G(K)\mod_{crit}}(D_{crit}(\Gr_G))$
and $\sC^{G(O)} = \TwoHom_{G(K)\mod_{crit}}(D_{crit}(\Gr_G),\sC)$.

In particular, $\sH_{sph}$ acts canonically on $D_{crit}(\Gr_G) = D_{crit}(G(K))^{G(O)}$.

Next, recall that there is a canonical monoidal functor
$\Rep(\ld{G}) \to \sH_{sph}$. This functor is characterized
by the fact that it is $t$-exact and the monoidal equivalence
on abelian categories defined by \cite{mirkovic-vilonen}. 
As in \cite{twisted-satake}, this functor is actually \emph{more} naturally
defined when the critical twisting is included, unlike in \cite{mirkovic-vilonen}.

We refer to the above functor as the geometric Satake functor and
denote it by $V \mapsto \sS_V$.

In what follows, whenever we consider $D_{crit}(\Gr_G)$ as a 
$\Rep(\ld{G})$-module category, it is via this construction.

\subsection{The canonical torsor}\label{ss:torsor}

Let $\sP_{\Op_{\ld{G}}^{reg}}$ denote the canonical
$\ld{G}$-bundle on $\Op_{\ld{G}}^{reg}$; by definition,
it corresponds to the forgetful map 
$\Op_{\ld{G}}^{reg}  \to \LocSys_{\ld{G}}(\sD) = \bB \ld{G}$. 

We obtain a symmetric monoidal
functor $\Rep(\ld{G}) \to \QCoh(\Op_{\ld{G}}^{reg})$.
We denote this functor 
$V \mapsto V_{\sP_{\Op_{\ld{G}}^{reg}}}$. 
Note that for $V \in \Rep(\ld{G})^{\heart}$ finite-dimensional,
$V_{\sP_{\Op_{\ld{G}}^{reg}}}$
is a vector bundle on $\Op_{\ld{G}}^{reg}$.

Throughout this section, 
whenever we consider $\QCoh(\Op_{\ld{G}}^{reg})$ as 
a $\Rep(\ld{G})$-module category, it is via this construction.

\subsection{Hecke $D$-modules}

Define $D_{crit}^{\Heckez}(\Gr_G)$ as:

\[
D_{crit}^{\Heckez}(\Gr_G) \coloneqq 
D_{crit}(\Gr_G) \underset{\Rep(\ld{G})}{\otimes} 
\QCoh(\Op_{\ld{G}}^{reg}).
\]

By construction, $D_{crit}^{\Heckez}(\Gr_G)$
is canonically a $D_{crit}(G(K)) \otimes \QCoh(\Op_{\ld{G}}^{reg})$-module category.

\begin{rem}

The above may be considered as a variant of
the category:

\[
D_{crit}^{\Hecke}(\Gr_G) 
\coloneqq D_{crit}(\Gr_G) \underset{\Rep(\ld{G})}{\otimes} \Vect
\]
 
\noindent that is suitably parametrized by regular opers. 
The category $D_{crit}^{\Hecke}(\Gr_G)$ is the category of Hecke eigenobjects
in $D_{crit}(\Gr_G)$; its Iwahori equivariant subcategory was studied
in detail in \cite{abbgm}. 

\end{rem}

\subsection{}\label{ss:ind-hecke}

There is a natural functor:

\[
\ind^{\Heckez}:
D_{crit}(\Gr_G) \to D_{crit}^{\Heckez}(\Gr_G)
\]

\noindent defined as the composition:

\[
D_{crit}(\Gr_G) = 
D_{crit}(\Gr_G) \underset{\Rep(\ld{G})}{\otimes} \Rep(\ld{G})
\to D_{crit}(\Gr_G) \underset{\Rep(\ld{G})}{\otimes} 
\QCoh(\Op_{\ld{G}}^{reg}) = D_{crit}^{\Heckez}(\Gr_G).
\]

Because because $\Op_{\ld{G}}^{reg} \to \bB \ld{G}$
is affine, $\Rep(\ld{G}) \to \QCoh(\Op_{\ld{G}}^{reg})$ admits
a continuous, conservative, right adjoint that is a morphism
of $\Rep(\ld{G})$-module categories. 
By functoriality, the same is true of $\ind^{\Heckez}$;
we denote this right adjoint by $\Oblv^{\Heckez}$.

In particular, we deduce that $D_{crit}^{\Heckez}(\Gr_G)$ is compactly generated
with compact generators of the form $\ind^{\Heckez}(\sF)$ for
$\sF \in D_{crit}(\Gr_G)$ compact.

\subsection{}\label{ss:hecke-t-str}

The DG category $D_{crit}^{\Heckez}(\Gr_G)$ carries a canonical
$t$-structure that plays an important role.

We construct the $t$-structure by setting connective objects
to be generated under colimits by objects of the
form $\ind^{\Heckez}(\sF)$ for
$\sF \in D_{crit}(\Gr_G)^{\leq 0}$. 

By construction, the 
composition $\Oblv^{\Heckez}\ind^{\Heckez}:D_{crit}(\Gr_G) \to 
D_{crit}(\Gr_G)$ is given by convolution with a spherical
$D$-module in the heart of the $t$-structure, namely,
the object corresponding under Satake to
functions on $\sP_{\Op_{\ld{G}}^{reg}}$ (considered as an
object of $\Rep(\ld{G})$ in the obvious way). Therefore, 
by \cite{central} (or \cite{fg2} \S 8.4), this
monad is $t$-exact on $D_{crit}(\Gr_G)$.

One deduces by standard methods that
$\Oblv^{\Heckez}$ and $\ind^{\Heckez}$ are $t$-exact.
In particular, because $\Oblv^{\Heckez}$ is $t$-exact, 
conservative, and $G(K)$-equivariant, we find that the
$t$-structure on $D_{crit}^{\Heckez}(\Gr_G)$ is strongly compatible
with the (critical level) $G(K)$-action in the sense of
\cite{methods} \S 10.12.

\subsection{}

In \S \ref{ss:gamma-hecke-start}-\ref{ss:gamma-hecke-finish}, 
following Frenkel-Gaitsgory, 
we will construct canonical global sections functors:

\[
\xymatrix{
D_{crit}^{\Heckez}(\Gr_G) \ar[d]_{\Gamma^{\Hecke}} \ar[dr]^{\Gamma^{\Hecke,naive}} & \\
\widehat{\fg}_{crit}\mod_{reg} \ar[r]^{\rho} & \widehat{\fg}_{crit}\mod_{reg,naive}
}
\]

\noindent that are our central objects of study.

\subsection{The Hecke property of the vacuum representation}\label{ss:gamma-hecke-start}

The construction of global sections functors as above is based on the
following crucial construction of Beilinson-Drinfeld.

\begin{thm}[Beilinson-Drinfeld]\label{t:birth}

For $\bV_{crit} \in \widehat{\fg}_{crit}\mod_{reg}^{\heart} \subset 
\widehat{\fg}_{crit}\mod_{reg,naive} \in G(K)\mod_{crit}$ the
vacuum representation
and $V \in \Rep(\ld{G})^{\heart}$ finite-dimensional,
the convolution $\sS_V \star \bV_{crit} \in 
\widehat{\fg}_{crit}\mod_{reg,naive}^{G(O)}$ lies in the heart
of the $t$-structure.

Moreover, there is a canonical isomorphism:

\[
\beta_V:
\sS_V \star \bV_{crit} \isom 
\bV_{crit} \underset{\Op_{\ld{G}}^{reg}}{\otimes} V_{\sP_{\Op_{\ld{G}}^{reg}}}
\in \widehat{\fg}_{crit}\mod_{reg}^{G(O),\heart}.
\]

For $V,W \in \Rep(\ld{G})^{\heart}$ finite-dimensional,
the following diagram in 
$\widehat{\fg}_{crit}\mod_{reg}^{G(O),\heart} \subset
\widehat{\fg}_{crit}\mod_{reg,naive}^{G(O)}$
commutes:

\[
\xymatrix{
\sS_W \star \sS_V \star \bV_{crit} 
\ar[rr]^{\sS_W \star \beta_V} \ar[d]^{\simeq} & &
\sS_W \star \bV_{crit} \underset{\Op_{\ld{G}}^{reg}}{\otimes} V_{\sP_{\Op_{\ld{G}}^{reg}}} \ar[rr]^{\beta_W} & &
\bV_{crit} \underset{\Op_{\ld{G}}^{reg}}{\otimes} W_{\sP_{\Op_{\ld{G}}^{reg}}}
\underset{\Op_{\ld{G}}^{reg}}{\otimes} V_{\sP_{\Op_{\ld{G}}^{reg}}}\ar@{=}[d]
\\
\sS_{W \otimes V} \star \bV_{crit} \ar[rrrr]^{\beta_{W \otimes V}}
& & & & 
\bV_{crit}
\underset{\Op_{\ld{G}}^{reg}}{\otimes} (W \otimes V)_{\sP_{\Op_{\ld{G}}^{reg}}}
}
\]

\noindent Here the left isomorphism comes from 
geometric Satake.

\end{thm}

We refer to \cite{hitchin} \S 5.5-6 and \cite{opers}
for proofs and further discussion.

\subsection{}

Let us reformulate the Hecke property more categorically.

For any $\sC \in G(K)\mod_{crit}$,
$\Rep(\ld{G})$ acts on $\sC^{G(O)}$ via the 
monoidal functor 
$\Rep(\ld{G}) \to \sH_{sph} \actson
\sC^{G(O)}$, where the first functor is the
geometric Satake functor.

For $\sC = \widehat{\fg}_{crit}\mod_{reg,naive}$,
we also have an action of $\Rep(\ld{G})$ on 
$\widehat{\fg}_{crit}\mod_{reg,naive}$ via
the (symmetric) monoidal functor
$\Rep(\ld{G}) \to \QCoh(\Op_{\ld{G}}^{reg})$ defined
by $\sP_{\Op_{\ld{G}}^{reg}}$. By construction,
this action commutes with the $G(K)$-action.

Therefore, $\widehat{\fg}_{crit}\mod_{reg,naive}^{G(O)}$
is canonically a $\Rep(\ld{G})$-bimodule
category. 

\begin{cor}\label{c:hecke-bimod}

There is a
unique morphism:

\[
\lambda:\Rep(\ld{G}) \to 
\widehat{\fg}_{crit}\mod_{reg,naive}^{G(O)}
\in \Rep(\ld{G})\bimod 
\]

\noindent of $\Rep(\ld{G})$-bimodule categories 
sending the trivial representation
$k \in \Rep(\ld{G})$ to $\bV_{crit}$ and
such that for any finite-dimensional representation
$V \in \Rep(\ld{G})^{\heart}$, the isomorphism:

\[
\sS_V \star \bV_{crit} = 
\lambda(V \otimes k) = \lambda(k \otimes V) = 
\bV_{crit} 
\underset{\Op_{\ld{G}}^{reg}}{\otimes} 
V_{\sP_{\Op_{\ld{G}}^{reg}}}
\]

\noindent is the isomorphism $\beta_V$ of 
Theorem \ref{t:birth}.

\end{cor}

\begin{proof}

Suppose $H_1 \subset H_2$ are affine algebraic
groups with $H_2/H_1$ affine, and let $\sC \in \Rep(H_2)\mod$.
Then the functor:

\[
\TwoHom_{\Rep(H_2)\mod}(\Rep(H_1),\sC) \to \sC
\]

\noindent of evaluation on the trivial representation
is monadic, with the corresponding monad on $\sC$
being given by $\Fun(H_2/H_1) \in 
\ComAlg(\Rep(H_2))$.\footnote{This construction
extends for $H_2/H_1$ quasi-affine as well as long as
$\Fun(H_2/H_1)$ is replaced by the (derived)
global sections $\Gamma(H_2/H_1,\sO_{H_2/H_1})$.}

We apply the above to $H_1 = \ld{G}$ diagonally embedded
into $H_2 = \ld{G} \times \ld{G}$. We then have 
$\Fun((\ld{G} \times \ld{G})/\ld{G}) = \Fun(\ld{G}) \in \Rep(\ld{G} \times 
\ld{G})^{\heart}$, where we consider $\ld{G}$ as equipped with its left and
right $\ld{G}$-actions. We are trying to show that
$\bV_{crit} \in \widehat{\fg}_{crit}\mod_{reg,naive}^{G(O)}$
admits a unique $\Fun(\ld{G})$-module structure satisfying
the stated compatibility. In particular, this structure
corresponds to certain maps in the \emph{abelian} category 
$\widehat{\fg}_{crit}\mod_{reg,naive}^{G(O),\heart}$,
so there are no homotopical issues.

From here, the claim is standard.
For example, for $V$ a finite-dimensional representation of $\ld{G}$,
we have a map $\mu_V:V \otimes V^{\vee} \to \Fun(\ld{G})$
of $\ld{G}$-bimodules. The composition 
of $\mu_V$ with the action map for the $\Fun(\ld{G})$-module structure
on $\bV_{crit}$ is given by the map:

\[
\sS_V \star \bV_{crit} 
\underset{\Op_{\ld{G}}^{reg}}{\otimes} V_{\sP_{\Op_{\ld{G}}^{reg}}}^{\vee}
\simeq \sS_V \star \sS_{V^{\vee}} \star \bV_{crit} = 
\sS_{V \otimes V^{\vee}} \star \bV_{crit} \to 
\bV_{crit}
\]

\noindent where the first isomorphism is induced
by $\beta_{V^{\vee}}$ and the second isomorphism 
and the last map is induced by the pairing 
$V \otimes V^{\vee} \to k \in \Rep(\ld{G})$ (for $k$ the 
trivial representation). It is straightforward from 
Theorem \ref{t:birth} that this defines an action of $\Fun(\ld{G})$ 
as desired.

\end{proof}

\subsection{Construction of the naive functor}

For any $\sC \in G(K)\mod_{crit}$, we have
a canonical identification:

\[
\TwoHom_{G(K)\mod_{crit}}(D_{crit}(\Gr_G),\sC) \isom 
\sC^{G(O)}
\]

\noindent given by evaluation on $\delta_1 \in 
D_{crit}(\Gr_G)^{G(O)}$. (Explicitly, the functor
$D_{crit}(\Gr_G) \to \sC$ corresponding to an object
$\sF \in \sC^{G(O)}$ is given by convolution with $\sF$.)

For $\sC = \widehat{\fg}_{crit}\mod_{reg,naive}$ and
$\bV_{crit} \in \widehat{\fg}_{crit}\mod_{reg,naive}$,
we denote the corresponding functor by 
$\Gamma^{\IndCoh}(\Gr_G,-):D_{crit}(\Gr_G) \to \widehat{\fg}_{crit}\mod_{reg,naive}$.
Note that the composition with the forgetful functor 
$\widehat{\fg}_{crit}\mod_{reg,naive} \to \widehat{\fg}_{crit}\mod$ 
is the usual ($\IndCoh$-)global sections functor by Appendix \ref{a:gamma}.

Now observe that $D_{crit}(\Gr_G)$ and $\widehat{\fg}_{crit}\mod_{reg,naive}$
are each $D_{crit}^*(G(K))\otimes \Rep(\ld{G})$-module categories.
We claim that Corollary \ref{c:hecke-bimod} naturally upgrades
$\Gamma^{\IndCoh}(\Gr_G,-)$ to a morphism of $D_{crit}^*(G(K))\otimes \Rep(\ld{G})$-module
categories. 

Indeed, suppose more generally that $\sC$ is a 
$D_{crit}^*(G(K))\otimes \Rep(\ld{G})$-module category. We then have:

\[
\begin{gathered} 
\TwoHom_{D_{crit}^*(G(K)) \otimes \Rep(\ld{G})\mod}(D_{crit}(\Gr_G),\sC) = 
\TwoHom_{\Rep(\ld{G})\bimod}(\Rep(\ld{G}),
\TwoHom_{G(K)\mod_{crit}}(D_{crit}(\Gr_G),\sC)) = \\
\TwoHom_{\Rep(\ld{G})\bimod}(\Rep(\ld{G}),
\sC^{G(O)}). 
\end{gathered}
\]

\noindent Therefore, Corollary \ref{c:hecke-bimod} has the claimed
effect.

Because the action of $\Rep(\ld{G})$ on $\widehat{\fg}_{crit}\mod_{reg,naive}$
comes from an action of $\QCoh(\Op_{\ld{G}}^{reg})$, we obtain 
an induced functor:

\[
D_{crit}^{\Heckez}(\Gr_G) = 
D_{crit}(\Gr_G) \underset{\Rep(\ld{G})}{\otimes} 
\QCoh(\Op_{\ld{G}}^{reg}) \to \widehat{\fg}_{crit}\mod_{reg,naive} 
\in D_{crit}^*(G(K)) \otimes \QCoh(\Op_{\ld{G}}^{reg})\mod.
\]

\noindent In what follows, we denote\footnote{A comment on the
notation: 

We use $\Heckez$ rather than $\Hecke$ 
in $D_{crit}^{\Heckez}(\Gr_G)$ to distinguish
this category from $D_{crit}^{\Hecke}(\Gr_G)$. But the global sections functor
is defined only on $D_{crit}^{\Heckez}(\Gr_G)$, not on 
$D_{crit}^{\Hecke}(\Gr_G)$, so we simplify the
notation here by omitting the subscript $\fz$.}
this functor by:

\[
\Gamma^{\Hecke,naive} = \Gamma^{\Hecke,naive}(\Gr_G,-).
\]

\subsection{Construction of the renormalized functor}\label{ss:gamma-hecke-finish}

Next, we construct a functor $\Gamma^{\Hecke}$ valued
in $\widehat{\fg}_{crit}\mod_{reg}$.

First, we need the following observation.

\begin{lem}\label{l:conv-adj}

Suppose $H$ is a Tate group indschenme,
$K \subset H$ is a polarization (i.e., a compact open
subgroup with $H/K$ ind-proper). 
Let $\sF \in D(H/K)$ be compact. Then 
for any $\sC \in H\mod$, the functor:

\[
\sF \star - :\sC^K \to \sC  
\]

\noindent admits a continuous right adjoint.

\end{lem}

\begin{proof}

Let $\bD \sF \in D(H/K)$ denote the Verdier dual
to $\sF$, and let $\on{inv} \bD \sF \in D(K\backslash H)$
denote the pullback along the inversion map.
As in \cite{fg2} Proposition 22.10.1,
the functor:

\[
\on{inv} \bD \sF \star - :\sC \to \sC^K
\]

\noindent canonically identifies with the desired right adjoint.

Alternatively, we may write convolution as
a composition:

\[
D(H)^K \otimes \sC^K \to D(H) \overset{K}{\otimes} \sC \to 
\sC
\]

\noindent and each of these functors admits a continuous
right adjoint (the former because $K$ is a group scheme,
and the latter because $H/K$ is ind-proper).
This formally implies the claim.

\end{proof} 

By Lemma \ref{l:conv-adj}, the global sections
functor $\rho \circ \Gamma^{\IndCoh}(\Gr_G,-):D_{crit}(\Gr_G) \to 
\widehat{\fg}_{crit}\mod_{reg,naive}$ preserves compact objects;
indeed, it is given as convolution with 
$\bV_{crit} \in \widehat{\fg}_{crit}\mod^{G(O)}$, which 
is compact.

Therefore, the functor $\Gamma^{\IndCoh}(\Gr_G,-)$
maps $D_{crit}(\Gr_G)^c$ to 
$\widehat{\fg}_{crit}\mod_{reg}^c$. 

From \S \ref{ss:ind-hecke},
we deduce that $\Gamma^{\Hecke,naive}$
maps compact objects in 
$D_{crit}^{\Heckez}(\Gr_G)$ into $\widehat{\fg}_{crit}\mod_{reg}^c$.

We now define:

\[
\Gamma^{\Hecke} = \Gamma^{\Hecke}(\Gr_G,-):
D_{crit}^{\Heckez} \to 
\widehat{\fg}_{crit}\mod_{reg}
\]

\noindent as the ind-extension of:

\[
\Gamma^{\Hecke,naive}|_{D_{crit}^{\Heckez}(\Gr_G)^c}:
D_{crit}^{\Heckez}(\Gr_G))^c
\to \widehat{\fg}_{crit}\mod_{reg}^c.
\]

\subsection{}

By abuse of notation, we let $\Gamma^{\IndCoh}(\Gr_G,-)$
denote the induced functor $\Gamma^{\Hecke} \circ \ind^{\Heckez}$,
so we have a commutative diagram:

\[
\xymatrix{
D_{crit}(\Gr_G) \ar[rr]^{\Gamma^{\IndCoh}(\Gr_G,-)} 
\ar[rrd]_{\Gamma^{\IndCoh}(\Gr_G,-)} & &
\widehat{\fg}_{crit}\mod_{reg} \ar[d]^{\rho} \\
&& \widehat{\fg}_{crit}\mod_{reg,naive}.
}
\]

\noindent (This abuse is mild because of
Corollary \ref{c:gamma-ren} below.) 

\subsection{Main result}

We can now state the main theorem of this paper in its precise form.

\begin{thm}\label{t:main}

For $G$ of semisimple rank $1$, 
the functor $\Gamma^{\Hecke}$ is a
$t$-exact equivalence.

\end{thm}

In the remainder of this section, we review some general results of 
Frenkel-Gaitsgory on $\Gamma^{\Hecke}$
and then formulate some
intermediate results in this case from which we will deduce
Theorem \ref{t:main}. The proofs of those intermediate results
occupy the remainder of the paper. 

\subsection{Review of some results of Frenkel-Gaitsgory}

The following exactness result was essentially shown 
in \cite{fg1}. 

\begin{thm}[\cite{fg1}, Theorem 1.2]\label{t:gamma-exact}

The functor:

\[
\Gamma^{\IndCoh}(\Gr_G,-) = \Gamma^{\Hecke,naive} \circ \ind^{\Heckez}: 
D_{crit}(\Gr_G) \to \widehat{\fg}_{crit}\mod_{reg,naive}
\]

\noindent is $t$-exact.

\end{thm}

There is something to do to properly deduce this from the Frenkel-Gaitsgory 
result, so we include a few comments.

Because $D_{crit}(\Gr_G)$ is compactly generated
and compact objects are closed under truncations,
it suffices to show that compact objects in
$D_{crit}(\Gr_G)$ lying in the heart of the
$t$-structure map into 
$\widehat{\fg}_{crit}\mod_{reg,naive}^{\heart}$.

By Proposition \ref{p:central-naive}, we are reduced
to verifying this result after composing with the functor
$\widehat{\fg}_{crit}\mod_{reg,naive} \to \widehat{\fg}_{crit}\mod$.

By Lemma \ref{l:action-tstr-bdd}, for $\sF \in D_{crit}(\Gr_G)$
compact, $\Gamma^{\IndCoh}(\Gr_G,\sF) = \sF \star \bV_{crit} \in \widehat{\fg}_{crit}\mod$
is eventually coconnective. Therefore, it suffices to show that 
when considered as an object of $\Vect$,
$\Gamma^{\IndCoh}(\Gr_G,\sF)$ lies in $\Vect^{\heart}$.

Now the result follows from \cite{fg1} Theorem 1.2 and 
the comparison results of Appendix \ref{a:gamma}.\footnote{In fact, 
that $\Gamma^{\IndCoh}(\Gr_G,-)$ as a functor
$D_{crit}(\Gr_G)^{\heart} \to \Vect$ coincides with the standard
global sections functor is one of the easier results in 
Appendix \ref{a:gamma}; it is shown directly in 
\S \ref{ss:conv-vac-gamma}.}\footnote{Formally, 
\cite{fg1} Theorem 1.2 only asserts that
the non-derived global sections functor is exact on 
$D_{crit}(\Gr_G)^{\heart}$, not exactly that
higher cohomology groups vanish. As the argument is
missing in the literature, we indicate the details here.

For any 
formally smooth $\aleph_0$-indscheme $S$ of ind-finite
type, we claim that if 
$H^0\Gamma^{\IndCoh}(S,-):
D(S)^{\heart} \to \Vect^{\heart}$
is exact, then $\Gamma^{\IndCoh}(S,-):D(S) \to \Vect$
is $t$-exact, and similarly for twisted $D$-modules.
 
Indeed, we are reduced to
showing that the restriction to $D(S)^+$ is $t$-exact.
This category is the bounded below derived category of its heart
by \cite{whit} Lemma 5.4.3 and the corresponding
assertion for finite type schemes.
Therefore, it suffices to show that 
$\Gamma^{\IndCoh}(S,-)$ is the derived functor
of $H^0\Gamma^{\IndCoh}(S,-)$, or equivalently, 
that $\Gamma^{\IndCoh}(S,-)$ a priori maps injective
objects in $D(S)^{\heart}$ into $\Vect^{\heart}$.

Formal smoothness of $S$ implies that
$\ind:\IndCoh(S) \to D(S)$ is $t$-exact, 
so its $t$-exact right adjoint
$\Oblv:D(S) \to \IndCoh(S)$ preserves injective objects.
Therefore, we are reduced to showing that 
$\Gamma^{\IndCoh}(S,-)$ maps injective objects
in $\IndCoh(S)^{\heart}$ into $\Vect^{\heart}$.

As $S$ is a classical indscheme by \cite{indschemes}, 
an argument along the lines of the proof of \cite{whit} Lemma 5.4.3
reduces us to the corresponding assertion for
finite type classical schemes. As $\IndCoh(S)^{\heart} = 
\QCoh(S)^{\heart}$ with $\Gamma^{\IndCoh}$ corresponding
to $\Gamma$, the assertion here is standard.}

\begin{cor}\label{c:gamma-ren}

The functor $\Gamma^{\IndCoh}(\Gr_G,-):D_{crit}(\Gr_G) \to 
\widehat{\fg}_{crit}\mod_{reg}$ is $t$-exact.

\end{cor}

\begin{proof}

For $\sF \in D_{crit}(\Gr_G)^{\heart}$ compact and
hence, compact in $D_{crit}(\Gr_G)$,
$\Gamma^{\IndCoh}(\Gr_G,\sF)$ is compact in
$\widehat{\fg}_{crit}\mod_{reg}$ by construction,
so lies in 
$\widehat{\fg}_{crit}\mod_{reg}^+$.
Therefore, by Theorem \ref{t:gamma-exact}, we deduce
that $\Gamma^{\IndCoh}(\Gr_G,\sF) \in 
\widehat{\fg}_{crit}\mod_{reg}^{\heart}$.

Because $D_{crit}(\Gr_G)^{\heart}$
is compactly generated and our $t$-structures are
compatible with filtered colimits, we obtain the claim.

\end{proof}

\begin{cor}\label{c:gamma-hecke-right-exact}

The functor $\Gamma^{\Hecke}:D_{crit}^{\Heckez}(\Gr_G) \to 
\widehat{\fg}_{crit}\mod_{reg}$ is right $t$-exact.

\end{cor}

\begin{proof}

By construction,
$D_{crit}^{\Heckez}(\Gr_G)^{\leq 0}$ is 
generated under colimits by objects of the form
$\ind^{\Heckez}(\sF)$ for $\sF \in D_{crit}(\Gr_G)^{\leq 0}$.
Then $\rho\Gamma^{\Hecke}(\ind^{\Heckez}(\sF)) = 
\Gamma^{\IndCoh}(\Gr_G,\sF)$ lies in degrees $\leq 0$
by Theorem \ref{t:gamma-exact},
so $\Gamma^{\Hecke}(\ind^{\Heckez}(\sF))$
lies in degrees $\leq 0$ and we obtain the claim.

\end{proof}

\subsection{Fully faithfulness}

Next, we review the fully faithfulness of $\Gamma^{\Hecke}$,
which was essentially shown in \cite{fg2} Theorem 8.7.1.
For the sake of completeness, we include the reduction to a calculation
performed in \cite{fg2}.

\begin{thm}[Modified Frenkel-Gaitsgory]\label{t:ff}

For any reductive $G$, the functor $\Gamma^{\Hecke}$ is fully faithful.

\end{thm} 

This result can be deduced from \cite{fg2} Theorem 8.7.1.
As the argument in \emph{loc. cit}. is quite involved,
we present a simpler one in Appendix \ref{a:ff} based on the
ideas of the current paper (especially the use of Whittaker
categories).

\subsection{Intermediate results}\label{ss:intermediate-results}

We now formulate three results whose proofs we defer to subsequent sections.

For each of the following results, we assume $G$ has 
semisimple rank $1$; we do not
do not know how to prove any of these lemmas for $GL_3$.

\begin{lem}\label{l:gamma-gens}

Let $\widehat{\fg}_{crit}\widetilde{\mod}_{reg,naive} \subset
\widehat{\fg}_{crit}\mod_{reg,naive}$ be the full subcategory
generated by $\widehat{\fg}_{crit}\mod_{reg,naive}^+$ under
colimits.\footnote{This is a technical distinction. 
It may perfectly well be the case that
$\widehat{\fg}_{crit}\widetilde{\mod}_{reg,naive}$ coincides
with $\widehat{\fg}_{crit}\mod_{reg,naive}$. But we do not see an 
argument and do not need to consider this question for the application
to Theorem \ref{t:main}.}

Then the essential image of $\Gamma^{\Hecke,naive}$ lies in 
$\widehat{\fg}_{crit}\widetilde{\mod}_{reg,naive}$ and generates
it under colimits.

\end{lem}

\begin{lem}\label{l:naive-t-exact}

The functor $\Gamma^{\Hecke,naive}$ is $t$-exact.

\end{lem}

\begin{lem}\label{l:gamma-hecke-bdded}

For every $K \subset G(O)$ a compact open subgroup, the composition:

\[
D_{crit}^{\Heckez}(\Gr_G)^K \to 
D_{crit}^{\Heckez}(\Gr_G) \xar{\Gamma^{\Hecke}}
\widehat{\fg}_{crit}\mod_{reg}
\]

\noindent is left $t$-exact up to shift.

\end{lem}

Assuming these results, let us show Theorem \ref{t:main}.

\begin{proof}[Proof of Theorem \ref{t:main}]

\step 

First, we show that $\Gamma^{\Hecke}$ is $t$-exact.

By Theorem \ref{t:gamma-exact} and the definition of the
$t$-structure on $D_{crit}^{\Heckez}(\Gr_G)$, 
$\Gamma^{\Hecke}$ is right $t$-exact.

To see left $t$-exactness, it suffices to see that for any
compact open subgroup $K \subset G(O)$, 
$\Gamma^{\Hecke}|_{D_{crit}^{\Heckez}(\Gr_G)^K}$ is left $t$-exact.
Indeed, for any $\sF \in D_{crit}^{\Heckez}(\Gr_G)$, 
$\sF = \colim_K \Oblv\Av_*^K(\sF)$, and $\Oblv\Av_*^K$
is left $t$-exact by the discussion of \S \ref{ss:hecke-t-str}.

By Lemma \ref{l:gamma-hecke-bdded}, 
$\Gamma^{\Hecke}|_{D_{crit}^{\Heckez}(\Gr_G)^K}$ is left $t$-exact up
to shift. Because $\rho:\widehat{\fg}_{crit}\mod_{reg}^+ \to
\widehat{\fg}_{crit}\mod_{reg,naive}^+$ is a $t$-exact equivalence, it
suffices to see that $\rho \circ \Gamma^{\Hecke}|_{D_{crit}^{\Heckez}(\Gr_G)^K}$
is left $t$-exact. But this is immediate from Lemma \ref{l:naive-t-exact}.

\step 

By Theorem \ref{t:ff}, it suffices to show that $\Gamma^{\Hecke}$ is essentially
surjective.

First, the composition:

\begin{equation}\label{eq:trun-gamma}
D_{crit}^{\Heckez}(\Gr_G) \xar{\Gamma^{\Hecke,naive}}
\widehat{\fg}_{crit}\widetilde{\mod}_{reg,naive} \xar{\tau^{\geq 0}}
\widehat{\fg}_{crit}\mod_{reg,naive}^{\geq 0}
\end{equation}

\noindent generates the target under colimits.
Indeed, the first functor generates under colimits by Lemma \ref{l:gamma-gens},
and the second functor is essentially surjective because  
$\widehat{\fg}_{crit}\widetilde{\mod}_{reg,naive}$ contains
$\widehat{\fg}_{crit}\mod_{reg,naive}^{\geq 0}$ by definition.

By the previous step, if we identify 
$\widehat{\fg}_{crit}\mod_{reg,naive}^{\geq 0}$ with
$\widehat{\fg}_{crit}\mod_{reg}^{\geq 0}$ via $\rho$, then 
\eqref{eq:trun-gamma} factors through 
$D_{crit}^{\Heckez}(\Gr_G)^{\geq 0}$, where it coincides with
$\Gamma^{\Hecke}|_{D_{crit}^{\Heckez}(\Gr_G)^{\geq 0}}$.

It therefore follows that the essential image of $\Gamma^{\Hecke}$
contains $\widehat{\fg}_{crit}\mod_{reg}^{\geq 0}$.
Because $\widehat{\fg}_{crit}\mod_{reg}$ is compactly generated
with compact objects lying in $\widehat{\fg}_{crit}\mod_{reg}^+$,
we deduce that the essential image of
$\Gamma^{\Hecke}$ is all of $\widehat{\fg}_{crit}\mod_{reg}$.

\end{proof}

\section{Equivariant categories}\label{s:iwahori-and-whittaker}

\subsection{} In this section, we collect some
results about $\Gamma^{\Hecke,naive}$ and
$\Gamma^{\Hecke}$ in the presence of
$\o{I}$ and Whittaker invariants. These
results will be used to establish the results formulated
in \S \ref{ss:intermediate-results}. 

We emphasize that we have nothing new to say about 
$\o{I}$-invariants; our proofs here 
consist only of references to \cite{fg-loc}.

\begin{rem}

All of the results of this section are valid for a general
reductive group $G$.

\end{rem}

\subsection{Iwahori equivariance}

The main result in this setting is the following.

\begin{thm}[Frenkel-Gaitsgory, \cite{fg-loc} Theorem 1.7]\label{t:iwahori}

The functor $\Gamma^{\Hecke,naive}$ induces a $t$-exact 
equivalence:

\[
D_{crit}^{\Heckez}(\Gr_G)^{\o{I},+} \isom 
\widehat{\fg}_{crit}\mod_{reg,naive}^{\o{I},+}
\]

\noindent on eventually coconnective 
$\o{I}$-equivariant categories.

\end{thm}

\begin{proof}

Because our setting
is slightly different from that
of \cite{fg-loc}, especially as regards
derived categories and derived functors,
we indicate the deduction from the results of \emph{loc. cit}.

First, we show $t$-exactness.
By \cite{fg-loc} Lemma 3.6 and Proposition 3.18, 
every object $\sF \in D_{crit}^{\Heckez}(\Gr_G)^{\o{I},\heart}$
can be written as a filtered colimit $\sF = \colim_i \sF_i$ for
$\sF_i \in D_{crit}^{\Heckez}(\Gr_G)^{\o{I},\heart}$
admitting a finite filtration with subquotients of the
form 
$\ind^{\Heckez}(\sF_{i,j}) 
{\otimes}_{\Op_{\ld{G}}^{reg}} \sH_{i,j}$ for 
$\sF_{i,j} \in D_{crit}^{\Heckez}(\Gr_G)^{\o{I},\heart}$ and
$\sH_{i,j} \in \QCoh(\Op_{\ld{G}}^{\heart})$.

We then have:

\[
\Gamma^{\Hecke,naive}(\ind^{\Heckez}(\sF_{i,j}) 
\underset{\Op_{\ld{G}}^{reg}}{\otimes} \sH_{i,j}) = 
\Gamma^{\IndCoh}(\Gr_G,\sF_{i,j}) 
\underset{\Op_{\ld{G}}^{reg}}{\otimes} \sH_{i,j}.
\]

\noindent By \emph{loc. cit}. Proposition 3.17, 
$\Gamma^{\IndCoh}(\Gr_G,\sF_{i,j})$ is flat
as an $\fz$-module, so the displayed tensor product
is concentrated in cohomological degree $0$.
This shows that $\Gamma^{\Hecke,naive}(\sF)$ is in 
degree $0$ as well, providing the $t$-exactness.

Next, observe that fully faithfulness follows from
Theorem \ref{t:ff}.

Finally, we show essential surjectivity. 
By \cite{fg-loc} 
Theorem 1.7, Lemma 3.6, Proposition 3.17,
and Proposition 3.18,
any $\sG \in \widehat{\fg}_{crit}\mod_{reg,naive}^{\o{I},\heart}$ 
can be written as a filtered colimit
$\sG = \colim_i \sG_i$ with 
$\sG_i \in \widehat{\fg}_{crit}\mod_{reg,naive}^{\o{I},\heart}$
and such that $\sG_i$ admits a finite filtration
with associated graded terms of the form:

\[
\Gamma^{\IndCoh}(\Gr_G,\widetilde{\sG}) 
\underset{\Op_{\ld{G}}^{reg}}{\otimes} \sH
\]

\noindent for $\widetilde{\sG} \in D_{crit}(\Gr_G)^{\o{I},\heart}$ 
and $\sH \in \QCoh(\Op_{\ld{G}}^{reg})^{\heart}$
(and where we are using the notation of \S \ref{ss:op-action}),
and where the displayed (derived) tensor product is concentrated
in cohomological degree $0$.
Clearly each associated graded term lies in the essential image
of $\Gamma^{\Hecke,naive}$, so $\sG$ does as 
well. This implies the essential image of $\Gamma^{\Hecke,naive}$
contains 
$\widehat{\fg}_{crit}\mod_{reg,naive}^{\o{I},\heart}$,
so all of $\widehat{\fg}_{crit}\mod_{reg,naive}^{\o{I},+}$.

\end{proof}

We include one other result in a similar spirit.

\begin{prop}\label{p:hecke-true}

The functor: 

\[
\Gamma^{\Hecke}|_{D_{crit}^{\Heckez}(\Gr_G)^{\o{I}}}:
D_{crit}^{\Heckez}(\Gr_G)^{\o{I}} \to  
\widehat{\fg}_{crit}\mod_{reg}
\]

\noindent is $t$-exact. 

\end{prop}

\begin{proof}

By Corollary \ref{c:gamma-hecke-right-exact}, the functor
is right $t$-exact. Therefore, we need to show
left $t$-exactness. 
By Theorem \ref{t:iwahori}, it suffices to show 
that objects in $D_{crit}^{\Heckez}(\Gr_G)^{\o{I},\heart}$ 
map to eventually coconnective objects.

Suppose $\sF \in D_{crit}^{\Heckez}(\Gr_G)^{\o{I},\heart}$.
As in the proof of Theorem \ref{t:iwahori}, 
the results of \cite{fg-loc} imply that 
$\sF$ can be written as a filtered colimit $\sF = \colim_i \sF_i$
for $\sF_i \in D_{crit}^{\Heckez}(\Gr_G)^{\o{I},\heart}$
admitting a finite filtration with subquotients of the
form 
$\ind^{\Heckez}(\sF_{i,j}) 
{\otimes}_{\Op_{\ld{G}}^{reg}} \sH_{i,j}$ for 
$\sF_{i,j}$ and $\sH_{i,j}$ as in the proof of 
Theorem \ref{t:iwahori}. 

Therefore, we are reduced to showing that:

\begin{equation}\label{eq:f-h}
\Gamma^{\IndCoh}(\Gr_G,\sF) 
\underset{\Op_{\ld{G}}^{reg}}{\otimes} \sH \in 
\widehat{\fg}_{crit}\mod_{reg}
\end{equation}

\noindent is eventually coconnective for any 
$\sF \in D_{crit}(\Gr_G)^{\o{I},\heart}$ and 
$\sH \in \QCoh(\Op_{\ld{G}}^{reg})^{\heart}$.

If $\sF$ is compact, then $\Gamma^{\IndCoh}(\Gr_G,\sF) \in 
\widehat{\fg}_{crit}\mod_{reg}$ is compact by construction
of the functor. In particular, this object is
eventually coconnective. By Theorem \ref{t:iwahori},
we deduce $\Gamma^{\IndCoh}(\Gr_G,\sF) \in 
\widehat{\fg}_{crit}\mod_{reg}^{\heart}$ in this case.
As the $t$-structures are compatible with filtered colimits,
and every object of $D_{crit}(\Gr_G)^{\o{I},\heart}$
is a filtered colimit of objects in 
$D_{crit}(\Gr_G)^{\o{I},\heart}$ that are compact
in $D_{crit}(\Gr_G)$, we obtain the claim for
general $\sF$ and $\sH$ being the structure sheaf. 

Now if $\sH$ is coherent,\footnote{I.e., $\sH$ corresponds
to a finitely presented $\fz$-module.} then because
$\Op_{\ld{G}}^{reg} = \Spec(\fz)$ with $\fz$ an (infinite) 
polynomial algebra, $\sH$ is perfect. 
Therefore, the object \eqref{eq:f-h} is eventually
coconnective for general $\sF$ and coherent $\sH$. 
Applying Theorem \ref{t:iwahori} again, we
deduce that \eqref{eq:f-h} lies in the heart
of the $t$-structure under these same assumptions.
Finally, the general case follows as any $\sH \in 
\QCoh(\Op_{\ld{G}}^{reg})^{\heart}$ is a filtered
colimit of coherent objects.

\end{proof}

\subsection{Whittaker equivariance}

We now study the behavior of $\Gamma^{\Hecke,naive}$ under 
the Whittaker functor, following \cite{whit} and \cite{methods}.

Our main result is the following.

\begin{thm}\label{t:whit-equiv}

\begin{enumerate}

\item\label{i:whit-equiv-1} 

The functor $\Gamma^{\Hecke,naive}$
induces an equivalence:

\[
\Whit(D_{crit}^{\Heckez}(\Gr_G)) \isom 
\Whit(\widehat{\fg}_{crit}\mod_{reg,naive}).
\]

\item\label{i:whit-equiv-2}

For $n>0$, the functor:

\[
\Gamma^{\Hecke,naive}:
\Whit^{\leq n}(D_{crit}^{\Heckez}(\Gr_G)) \to 
\Whit^{\leq n}(\widehat{\fg}_{crit}\mod_{reg,naive})
\]

\noindent is a $t$-exact equivalence for the natural\footnote{We
remind that $\Whit^{\leq n}$ is defined as equivariance
against a character for a compact open subgroup. For our
two categories, the $t$-structures are compatible with the
$G(K)$-action, so there
are natural $t$-structures on such equivariant categories.}
$t$-structures on both sides.

\end{enumerate}

\end{thm}

We will verify the above result in what follows
after recalling some results on Whittaker categories 
in this setting.

\subsection{}

We recall the following result, which appears as 
\cite{methods} Theorem 11.19.1 
and is an enhancement of the affine Skryabin theorem
\cite{whit} Theorem 5.1.1.

\begin{thm}\label{t:whit-crit}

There is a canonical equivalence of $\IndCoh^*(\Op_{\ld{G}})$-comodule categories:

\[
\Whit(\widehat{\fg}_{crit}\mod) \simeq \IndCoh^*(\Op_{\ld{G}}).
\]

\noindent Under this equivalence, the full subcategory (c.f. \S \ref{ss:ad-whit-review})
$\Whit^{\leq m}(\widehat{\fg}_{crit}\mod) \subset 
\Whit(\widehat{\fg}_{crit}\mod)$ identifies with the full subcategory
$\IndCoh_{\Op_{\ld{G}}^{\leq m}}^*(\Op_{\ld{G}}) \subset 
\IndCoh^*(\Op_{\ld{G}})$ generated under colimits by pushforwards
from $\QCoh(\Op_{\ld{G}}) \simeq \IndCoh^*(\Op_{\ld{G}}^{\leq m}) 
\to \IndCoh^*(\Op_{\ld{G}})$.

\end{thm}

\begin{cor}\label{c:crit-whit-naive}

For any $n$, there is a canonical equivalence of 
$\QCoh(\Op_{\ld{G}}^{\leq n})$-module categories:

\[
\Whit(\widehat{\fg}_{crit}\mod_{\ord_n,naive}) \simeq \QCoh(\Op_{\ld{G}}^{\leq m}).
\]

\noindent Moreover, for any positive $m$ with $m \geq n$, the embedding:

\[
\Whit^{\leq m}(\widehat{\fg}_{crit}\mod_{\ord_n,naive}) \to 
\Whit(\widehat{\fg}_{crit}\mod_{\ord_n,naive})
\] 

\noindent is an equivalence.

\end{cor}

\begin{proof}

By \cite{whit} Theorem 2.1.1 (or its refinement Theorem 2.7.1,
which we recalled above as Theorem \ref{t:whit}),
the functor $G(K)\mod_{crit} \xar{\sC \mapsto \Whit(\sC)} \DGCat_{cont}$ 
is a morphism of $\DGCat_{cont}$-module categories that commutes
with limits and colimits.

Therefore, from the definitions, we have:

\[
\begin{gathered}
\Whit(\widehat{\fg}_{crit}\mod_{\ord_n,naive}) =
\IndCoh^*(\Op_{\ld{G}}^{\leq n}) 
\overset{\IndCoh^*(\Op_{\ld{G}})}{\otimes} 
\Whit(\widehat{\fg}_{crit}\mod) \overset{Thm. \ref{t:whit-crit}}
{\simeq} \\
\IndCoh^*(\Op_{\ld{G}}^{\leq n}) 
\overset{\IndCoh^*(\Op_{\ld{G}})}{\otimes} 
\IndCoh^*(\Op_{\ld{G}}) \simeq \IndCoh^*(\Op_{\ld{G}}^{\leq n}) \simeq
\QCoh(\Op_{\ld{G}}^{\leq n}).
\end{gathered}
\]

The stabilization of adolescent Whittaker
models is proved similarly.
For $m$ positive, we have:

\[
\begin{gathered}
\Whit^{\leq m}(\widehat{\fg}_{crit}\mod_{\ord_n,naive}) =
\IndCoh^*(\Op_{\ld{G}}^{\leq n}) 
\overset{\IndCoh^*(\Op_{\ld{G}})}{\otimes} 
\Whit^{\leq m}(\widehat{\fg}_{crit}\mod) \overset{Thm. \ref{t:whit-crit}}
{\simeq} \\
\IndCoh^*(\Op_{\ld{G}}^{\leq n}) 
\overset{\IndCoh^*(\Op_{\ld{G}})}{\otimes} 
\IndCoh_{\Op_{\ld{G}}^{\leq m}}^*(\Op_{\ld{G}}) \subset \\
\IndCoh^*(\Op_{\ld{G}}^{\leq n}) 
\overset{\IndCoh^*(\Op_{\ld{G}})}{\otimes} 
\IndCoh^*(\Op_{\ld{G}}) \simeq \IndCoh^*(\Op_{\ld{G}}^{\leq n})
\simeq \QCoh(\Op_{\ld{G}}^{\leq n}).
\end{gathered}
\]

\noindent The functor at the end of the second line is indeed
fully faithful because 
$\IndCoh_{\Op_{\ld{G}}^{\leq m}}^*(\Op_{\ld{G}}) \to 
\IndCoh^*(\Op_{\ld{G}})$ is fully faithful (by definition) and
admits a right adjoint that is a morphism
of $\IndCoh^*(\Op_{\ld{G}})$-module categories. Clearly this
functor is essentially surjective for $m \geq n$.

\end{proof}

\subsection{}\label{ss:av-!-psi-sph}

Before proceeding,
we recall that for 
$\sC \in G(K)\mod_{crit}$, the functor:

\[
\Whit(\sC) \xar{\Oblv} \sC \xar{\Av_*} \sC^{G(O)}
\]

\noindent admits a left adjoint, which we denote
$\Av_!^{\psi}$ in what follows.
That this left adjoint is 
defined is the special case 
$n = 0$, $m = \infty$ of \cite{whit} Theorem 2.7.1. 

\subsection{}

We now recall the following result.

\begin{thm}[Frenkel-Gaitsgory-Vilonen, \cite{fgv}]\label{t:whit-sph}

The composition:

\[
\Rep(\ld{G}) \xar{V \mapsto \sS_V} \sH_{sph} =
D_{crit}(\Gr_G)^{G(O)} \xar{\Av_!^{\psi}} 
\Whit(D_{crit}(\Gr_G))
\]

\noindent is an equivalence. 

\end{thm}

\begin{rem}

Formally, the setting of \cite{fgv} is somewhat different.
We refer to \cite{whit-comp} for the necessary comparison results.

\end{rem}

\subsection{}

We now can prove the main result on Whittaker categories.

\begin{proof}[Proof of Theorem \ref{t:whit-equiv}]

We begin with \eqref{i:whit-equiv-1}. We first
construct some equivalence, and then
we show that $\Gamma^{\Hecke,naive}$ induces the corresponding
functor.

By Corollary \ref{c:crit-whit-naive} (for $n = 0$),
we have:

\[
\Whit(\widehat{\fg}_{crit}\mod_{reg,naive}) \simeq \QCoh(\Op_{\ld{G}}^{reg}).
\]

Moreover, as Whittaker invariants coincide with coinvariants
by \cite{whit} Theorem 2.1.1, we can calculate:

\[
\Whit(D_{crit}^{\Heckez}(\Gr_G)) = 
\Whit(D_{crit}(\Gr_G)) \underset{\Rep(\ld{G})}{\otimes} 
\QCoh(\Op_{\ld{G}}^{reg}).
\]

\noindent By Theorem \ref{t:whit-sph}, 
$\Whit(D_{crit}(\Gr_G))$ identifies canonically with 
$\Rep(\ld{G})$ as a $\Rep(\ld{G})$-module category.
Therefore, we obtain:

\[
\Whit(D_{crit}^{\Heckez}(\Gr_G)) = 
\Rep(\ld{G}) \underset{\Rep(\ld{G})}{\otimes} 
\QCoh(\Op_{\ld{G}}^{reg}) = \QCoh(\Op_{\ld{G}}^{reg}).
\]

We now show that $\Gamma^{\Hecke,naive}$ induces the
evident equivalence on Whittaker categories.
By construction, $\Gamma^{\Hecke,naive}$ is a morphism
of $\QCoh(\Op_{\ld{G}}^{reg})$-module categories.
Therefore, it suffices to show that it sends the
structure sheaf $\sO_{\Op_{\ld{G}}^{reg}}$ to
itself.

This follows from the following diagram, 
which is commutative by functoriality:

\[
\xymatrix{
D_{crit}(\Gr_G)^{G(O)} \ar[d]^{\Av_!^{\psi}} \ar[rr]^{\Gamma^{\Hecke,naive}} &&
\widehat{\fg}_{crit}\mod_{reg,naive}^{G(O)} \ar[d]^{\Av_!^{\psi}} \ar[dr] \\
\Whit(D_{crit}(\Gr_G)) \ar[rr]^{\Gamma^{\Hecke,naive}} &&
\Whit(\widehat{\fg}_{crit}\mod_{reg,naive}) \ar@{=}[r] & \QCoh(\Op_{\ld{G}}^{reg}).
}
\]

\noindent By construction of the equivalence
of Theorem \ref{t:whit-crit}, 
the diagonal arrow in the diagram above
is the Drinfeld-Sokolov functor $\Psi$.
Therefore, if 
we consider the $\delta$ $D$-module 
$\delta_1 \in D_{crit}(\Gr_G)^{G(O)}$
supported at the origin $1 \in \Gr_G$, apply Hecke induction 
and the above diagram,
we find:

\[
\begin{gathered}
\Gamma^{\Hecke,naive}(\Av_!^{\psi}\ind^{\Heckez}\delta_1) = 
\Psi(\Gamma^{\IndCoh}(\Gr_G,\delta_1)) = \Psi(\bV_{crit}).
\end{gathered}
\]

\noindent Clearly $\Av_!^{\psi}\ind^{\Heckez}\delta_1 \in \Whit(D_{crit}^{\Heckez}(\Gr_G))$ corresponds
to $\sO_{\Op_{\ld{G}}^{reg}} \in \QCoh(\Op_{\ld{G}}^{reg})$.
Moreover, $\Psi(\bV_{crit})$ corresponds to the structure sheaf
$\sO_{\Op_{\ld{G}}^{reg}}$
by design. 

We now verify \eqref{i:whit-equiv-2}. 
For $n > 0$, we have natural functors:

\begin{equation}\label{eq:sph-whit-n}
\begin{gathered}
\Whit^{\leq n}(D_{crit}^{\Heckez}(\Gr_G)) \to 
\Whit(D_{crit}^{\Heckez}(\Gr_G)) \\
\Whit^{\leq n}(\widehat{\fg}_{crit}\mod_{reg,naive}) \to 
\Whit(\widehat{\fg}_{crit}\mod_{reg,naive})
\end{gathered} 
\end{equation}

\noindent as in Theorem \ref{t:whit},
and that we claim are equivalences. In the second case,
this assertion is part of Corollary \ref{c:crit-whit-naive}.
In the first case, this follows from the fact that:

\[
\Whit^{\leq 1}(D_{crit}(\Gr_G)) \to \Whit(D_{crit}(\Gr_G))
\]

\noindent is an equivalence; see \cite{cpsii} Theorem 7.3.1
for a stronger assertion.

It now follows by functoriality and \eqref{i:whit-equiv-1}
that $\Gamma^{\Hecke,naive}$ is an equivalence on
$\Whit^{\leq n}$ for all $n>0$.

Finally, we need to show that $\Gamma^{\Hecke,naive}$
is $t$-exact on $\Whit^{\leq n}$ for all $n$. 

In \cite{whit}, the functors:

\[
\iota_{n,n+1,!}[-2(\check{\rho},\rho)]:
\Whit^{\leq n}(\widehat{\fg}_{crit}\mod) \to 
\Whit^{\leq n+1}(\widehat{\fg}_{crit}\mod)
\]

\noindent were shown to be $t$-exact.
Moreover, by the proof of the affine Skryabin theorem 
Theorem \ref{t:whit-crit}, the resulting $t$-structure on 
$\Whit(\widehat{\fg}_{crit}\mod)$ identifies with
the canonical one on $\IndCoh^*(\Op_{\ld{G}})$.
We deduce parallel results for 
$\widehat{\fg}_{crit}\mod_{reg,naive}$
in place of $\widehat{\fg}_{crit}\mod$ in the setting of 
Corollary \ref{c:crit-whit-naive}.

Similarly, the functors:

\[
\iota_{n,n+1,!}[-2(\check{\rho},\rho)]:
\Whit^{\leq n}(D_{crit}(\Gr_G)) \to 
\Whit^{\leq n+1}(D_{crit}(\Gr_G))
\]

\noindent are $t$-exact. The resulting $t$-structure
on $\Whit(D_{crit}(\Gr_G))$ identifies with the canonical
one on $\Rep(\ld{G})$ under Theorem \ref{t:whit-sph}; 
indeed, the geometric Satake functor $\Rep(\ld{G}) \to \sH_{sph}$ is
$t$-exact by construction, and $\Av_!^{\psi}$ is $t$-exact
by \cite{whit} Remark B.7.1. As $\Oblv^{\Heckez}$ is $t$-exact,
we obtain similar results for 
$D_{crit}^{\Heckez}(\Gr_G)$.

Finally, we deduce $t$-exactness. Indeed, we have equivalences:

\[
\Whit^{\leq n}(D_{crit}^{\Heckez}(\Gr_G)) \simeq 
\QCoh(\Op_{\ld{G}}^{reg}) \simeq  
\Whit^{\leq n}(\widehat{\fg}_{crit}\mod_{reg,naive})
\]

\noindent with the $t$-structures on the left and
right hand sides corresponding to the canonical
$t$-structure on $\QCoh(\Op_{\ld{G}}^{reg})$, and the
composition being given by 
$\Gamma^{\Hecke,naive}$.

\end{proof}

\subsection{Exactness of renormalized global sections}

We will also need the following parallel
to Proposition \ref{p:hecke-true}.

\begin{prop}\label{p:hecke-true-whit}

\begin{enumerate}

\item\label{i:true-whit-1}

For any $n \geq 1$, the functor: 

\[
\Gamma^{\Hecke}|_{\Whit^{\leq n}(D_{crit}^{\Heckez}(\Gr_G))}:
\Whit^{\leq n}(D_{crit}^{\Heckez}(\Gr_G)) \to  
\widehat{\fg}_{crit}\mod_{reg}
\]

\noindent is $t$-exact. 

\item\label{i:true-whit-2}

More generally, suppose
$\sG \in D_{crit}(G(K))$ has the following properties: 

\begin{itemize}

\item For $m \gg 0$, $\sG$ is right $K_m$-equivariant
(where $K_m \subset G(O)$ is the $m$th
congruence subgroup). 

\item There exists a $K_m$-stable closed subscheme 
$S \subset G(K)$ such that $\sG$ is supported on $S$.

\item As an object of $D(S/K_m)$, $\sG$ is eventually coconnective.

\end{itemize}

Then for every $\sF \in 
\Whit^{\leq n}(D_{crit}^{\Heckez}(\Gr_G))^+$, 
$\Gamma^{\Hecke}(\Gr_G,\sG \star \sF) \in 
\widehat{\fg}_{crit}\mod_{reg}^+$.

\end{enumerate}

\end{prop}

\begin{proof}

We begin with \eqref{i:whit-equiv-1}.

As above, we have a $t$-exact equivalence:

\begin{equation}\label{eq:whit-n-hecke}
\Whit^{\leq n}(D_{crit}^{\Heckez}(\Gr_G)) \simeq 
\QCoh(\Op_{\ld{G}}^{reg}).
\end{equation}

\noindent As $\Op_{\ld{G}}^{reg}$ is the
spectrum of a polynomial algebra (however infinitely
generated), we deduce that every
object of 
$\Whit^{\leq n}(D_{crit}^{\Heckez}(\Gr_G))^{\heart}$
is a filtered colimit of objects that
are compact in 
$\Whit^{\leq n}(D_{crit}^{\Heckez}(\Gr_G))$,
hence in $D_{crit}^{\Heckez}(\Gr_G)$.

By construction, $\Gamma^{\Hecke}$ maps compact
objects to compact objects, and in particular
maps compact objects to 
$\widehat{\fg}_{crit}\mod_{reg}^+$. 
By Theorem \ref{t:whit-equiv}, we deduce
that it maps compact objects 
of $\Whit^{\leq n}(D_{crit}^{\Heckez}(\Gr_G))$
that lie in the heart of the $t$-structure
into $\widehat{\fg}_{crit}\mod_{reg}^{\heart}$.
As every object of
$\Whit^{\leq n}(D_{crit}^{\Heckez}(\Gr_G))^{\heart}$
is a filtered colimit of such (by the above), we
obtain the result.

We now proceed to \eqref{i:whit-equiv-2}.
We begin by noting that our assumptions imply that
for any $\sC \in G(K)\mod_{crit}$ equipped with a $t$-structure
that is strongly compatible with the $G(K)$-action (in the
sense of \cite{methods} \S 10.12), the
functor $\sG \star - : \sC \to \sC$ is left $t$-exact up to shift
(see the proof of Lemma \ref{l:action-tstr-bdd} below).
This is the key property we will use about $\sG$.
By \cite{methods} Lemma 10.14.1, this property is true for
$\sC = \widehat{\fg}_{crit}\mod$.

Next, if $\sF$ is the object corresponding
under \eqref{eq:whit-n-hecke}
to the structure sheaf on
$\Op_{\ld{G}}^{reg}$, then 
$\sF = \ind^{\Heckez}(\delta_n)$ 
for $\delta_n \in \Whit^{\leq n}(D_{crit}(\Gr_G))^{\heart}
\simeq \Rep(\ld{G})^{\heart}$
corresponding to the trivial representation 
(by construction of \eqref{eq:whit-n-hecke}).
Therefore, $\Gamma^{\Hecke}(\Gr_G,\sG \star \sF)
= \Gamma^{\IndCoh}(\Gr_G,\sG \star \delta_n)$.
As $\sG \star \delta_n \in D_{crit}(\Gr_G)$ is eventually
coconnective by the above, the resulting object of
$\widehat{\fg}_{crit}\mod_{reg}$ is eventually coconnective as well by
Corollary \ref{c:gamma-ren}.

We deduce from \eqref{eq:whit-n-hecke}
that for $\sF \in \Whit^{\leq n}(D_{crit}(\Gr_G))$ compact,
$\Gamma^{\Hecke}(\Gr_G,\sG \star \sF)$ is
eventually coconnective. We claim that in fact there
is a universal integer $r$ such that
for compact $\sF$ lying in $\Whit^{\leq n}(D_{crit}(\Gr_G))^{\geq 0}$, 
we have:

\[
\Gamma^{\Hecke}(\Gr_G,\sG \star \sF) \in 
\widehat{\fg}_{crit}\mod_{reg}^{\geq -r}.
\]

Indeed, choose $r$ such that $\sG \star -$
maps $\widehat{\fg}_{crit}\mod^{\geq 0}$ into 
$\widehat{\fg}_{crit}\mod^{\geq -r}$. 
As we know the above object is eventually coconnective,
it suffices to verify the boundedness after applying $\rho$.
Then the resulting object is 
$\sG \star \Gamma^{\Hecke,naive}(\Gr_G,\sF)$,
which lies in degrees $\geq -r$ by construction of $r$
and Theorem \ref{t:whit-equiv}.

Finally, the same claim for general (possibly non-compact) 
$\sF \in \Whit^{\leq n}(D_{crit}(\Gr_G))^{\geq 0}$ 
follows by the same argument as in \eqref{i:true-whit-1}:
such $\sF$ is a filtered colimit of objects
of $\Whit^{\leq n}(D_{crit}(\Gr_G))^{\geq 0}$ that
are compact in $\Whit^{\leq n}(D_{crit}(\Gr_G))$.

\end{proof}

\section{Generation under colimits}\label{s:essential-surjectivity}

\subsection{} In this section, we prove Lemma \ref{l:gamma-gens}.

\subsection{Preliminary observations}

We begin with the following basic result. 

\begin{lem}\label{l:tilde}

The subcategory $\widehat{\fg}_{crit}\widetilde{\mod}_{reg,naive} \subset
\widehat{\fg}_{crit}\mod_{reg,naive}$ is a $D_{crit}^*(G(K))$-submodule
category.

\end{lem}

\begin{proof}

By definition of $\widehat{\fg}_{crit}\widetilde{\mod}_{reg,naive}$, 
we need to show that for $\sF \in D_{crit}^*(G(K))$, the functor $\sF \star -$
maps $\widehat{\fg}_{crit}\mod_{reg,naive}^+$ into 
$\widehat{\fg}_{crit}\widetilde{\mod}_{reg,naive}$.
As $D_{crit}^*(G(K))$ is compactly generated, we are reduced to the
case where $\sF$ is compact. In that case, we claim that 
$\sF \star -$ maps $\widehat{\fg}_{crit}\mod_{reg,naive}^+$ into
itself. 

Indeed, this follows from Lemma \ref{l:action-tstr-bdd} and the
observation that the action of $G(K)$ on $\widehat{\fg}_{crit}\mod_{reg,naive}$
is strongly compatible with the $t$-structure; the latter claim
reduces via Lemma \ref{l:ren} 
to the same claim for $\widehat{\fg}_{crit}\mod$,
which is shown as \cite{methods} Lemma 10.14.1 (3).

\end{proof}

Above, we used the following result.

\begin{lem}\label{l:action-tstr-bdd}

Let $H$ be a
Tate group indscheme with prounipotent tail 
acting strongly on $\sC \in \DGCat_{cont}$.
Suppose $\sC$ is equipped with a $t$-structure strongly
compatible with the $H$-action in the sense of 
\cite{methods} \S 10.12.
Then for any $\sF \in D^*(H)$ compact, 
the functor $\sF \star -:\sC \to \sC$ is left $t$-exact up to shift.

\end{lem}

\begin{proof}

Because $\sF$ is compact and $H$ has a prounipotent tail, 
$\sF \in D^*(H)^K \simeq D(H/K)$ for some prounipotent compact open subgroup
$K \subset H$.
Again because $\sF$ is compact, as an object of $D(H/K)$, it 
is supported on a closed subscheme $S \subset H/K$. 
By \cite{finiteness}, 
$\sF$ has a bounded resolution by compact objects
of the form $\ind(i_*^{\IndCoh}(\sG))$ for $i:S \to H/K$
the embedding, $\sG \in \IndCoh(S)$ compact,
and $\ind$ the functor of (right) $D$-module induction.
Therefore, we may consider $\sF$ of this form.

The functor $\sF \star -$ then factors as:

\[
\sC \xar{\Av_*^K} \sC^K \xar{\Oblv} \sC^{K,w} 
\xar{i_*^{\IndCoh}(\sG) \overset{K_0,w}{\star} -}
\sC
\]

\noindent where $-\overset{K_0,w}{\star} -$
indicates the appropriate relative convolution
functor $\IndCoh(H/K)^{K_0/K,w} \otimes \sC^{K_0,w}
= \IndCoh(H/K_0) \otimes \sC^{K_0,w} \to \sC$.

As the $H$-action on $\sC$ is compatible with the $t$-structure,
$\sC^K \subset \sC$ is closed under truncations; it follows
that $\Av_*^K$ is left $t$-exact. 
By \cite{methods} \S 10.13,
$\sC^{K,w}$ has a canonical $t$-structure for
which $\Oblv:\sC^K \to \sC^{K,w}$ is $t$-exact. 
Finally, the functor of convolution with $\sG$ is
left $t$-exact by \cite{methods} Proposition 10.16.1.\footnote{There 
is a polarizability assumption at this point in \emph{loc. cit}
that we have omitted here.
This assumption is only needed in \emph{loc. cit}. to 
deduce a stronger result.
The beginning of that argument from \emph{loc. cit}. 
is all that is needed here,
and for that the polarizability is not needed. (Regardless,
we only apply this result to $G(K)$, 
which is polarizable.)}

\end{proof}

\begin{cor}

$\Gamma^{\Hecke,naive}$ factors through 
$\widehat{\fg}_{crit}\widetilde{\mod}_{reg,naive}$.

\end{cor}

\begin{proof}

By \S \ref{ss:ind-hecke}, 
it suffices to show $\Gamma^{\IndCoh} = \Gamma^{\Hecke,naive} \circ \ind^{\Hecke}$
factors through $\widehat{\fg}_{crit}\widetilde{\mod}_{reg,naive}$. 
This functor is given by convolution with 
$\bV_{crit} \in 
\widehat{\fg}_{crit}\mod_{reg,naive}^{\heart} \subset 
\widehat{\fg}_{crit}\widetilde{\mod}_{reg,naive}$, so the claim follows
from Lemma \ref{l:tilde}.

\end{proof}

\begin{cor}\label{c:tilde-gen}

Let $K \subset G(O)$ be a prounipotent\footnote{This assumption
can be omitted, but the argument requires some additional details.}
group subscheme. Then
$\widehat{\fg}_{crit}\widetilde{\mod}_{reg,naive}^K$ is the subcategory
of $\widehat{\fg}_{crit}\mod_{reg,naive}^K$ generated under colimits
by $\widehat{\fg}_{crit}\mod_{reg,naive}^{K,+}$.

\end{cor}

\begin{proof}

We have a commutative diagram: 

\[
\xymatrix{
\widehat{\fg}_{crit}\mod_{reg,naive}^+ \ar@{^(->}[rr] \ar[d]_{\Av_*^K} & &
\widehat{\fg}_{crit}\widetilde{\mod}_{reg,naive} \ar[d]^{\Av_*^K} \\
\widehat{\fg}_{crit}\mod_{reg,naive}^{K,+} \ar@{^(->}[rr] & &
\widehat{\fg}_{crit}\widetilde{\mod}_{reg,naive}^K.
}
\]

\noindent The top and right functors generate under colimits, so the same is true of their
composition. This implies that the bottom arrow generates under colimits, 
as desired. 

\end{proof}

\subsection{Proof for $PGL_2$}

To simplify the discussion, we first assume $G = PGL_2$.
We indicate the necessary modifications for general $G$ of
semisimple rank $1$ in \S \ref{ss:rk1}.  

By construction, $\Gamma^{\Hecke,naive}$ is 
a $G(K)$-equivariant functor (at critical level). 
In particular, the subcategory of
$\widehat{\fg}_{crit}\widetilde{\mod}_{reg,naive}$ 
generated under colimits by its
essential image is closed under the $G(K)$-action.

Therefore, by Theorem \ref{t:generic}, to prove Lemma \ref{l:gamma-gens}
it suffices to show that
the essential image of $\Gamma^{\Hecke,naive}$ contains
$\widehat{\fg}_{crit}\widetilde{\mod}_{reg,naive}^{\o{I}}$ and
$\Whit(\widehat{\fg}_{crit}\widetilde{\mod}_{reg,naive})$.
The former follows from Theorem \ref{t:iwahori}, while
the latter follows from Theorem 
\ref{t:whit-equiv}.\footnote{In the latter
case, it is shown that
$\Gamma^{\Hecke,naive}$ even induces
an equivalence on Whittaker categories
with $\widehat{\fg}_{crit}\mod_{reg,naive}$,
i.e., the distinction with
$\widehat{\fg}_{crit}\widetilde{\mod}_{reg,naive}$
is not necessary for the Whittaker part of the argument.}

\subsection{Generalization to groups of semisimple rank $1$}\label{ss:rk1}

We briefly indicate the argument for general $G$ of
semisimple rank $1$. 

First, for $\vph:G_1 \to G_2$ an isogeny of reductive groups,
the natural functor:

\[
D_{crit}^{\Heckez}(\Gr_{G_1}) \to D_{crit}^{\Heckez}(\Gr_{G_2})
\]

\noindent is an equivalence. Indeed, this
follows as:

\[
D_{crit}(\Gr_{G_1}) 
\underset{\Rep(\ld{G}_1)}{\otimes} \Rep(\ld{G}_2) \to 
D_{crit}(\Gr_{G_2})
\]

\noindent and:

\[
\Op_{\ld{G}_2}^{reg} \to \Op_{\ld{G}_1}^{reg}
\]

\noindent are equivalences (the latter being a consequence
of Remark \ref{r:opers-defin}).

In particular, one deduces that
$G(K)$ acts (with critical level) 
on $D_{crit}^{\Heckez}(\Gr_G)$
through $G^{ad}(K)$ (e.g., it is easy to see directly 
that the action is trivial for $G$ a torus).
The same is evidently true
for the action on 
$\widehat{\fg}_{crit}\mod_{reg,naive}$. Moreover,
$\Gamma^{\Hecke,naive}$ is $G^{ad}(K)$-equivariant.

Next, one observes that the Whittaker category with 
respect to the $G^{ad}(K)$ action coincides with the
Whittaker category for the $G(K)$ action, and similarly
for the radical of Iwahori. For later reference,
we also highlight that for $n>0$, the
invariants for the $n$th subgroup
of $G(K)$ coincide with the similar invariants 
for the $G^{ad}(K)$-action.

Finally, we observe that for $G$ of semisimple rank
$1$, $G^{ad} = PGL_2$, so we can apply the above
observations and Theorem \ref{t:generic}.

\section{Exactness}\label{s:exactness}

\subsection{} In this section, we prove 
Lemma \ref{l:naive-t-exact}.
The main idea is Proposition \ref{p:pgl2-exact}.

\subsection{$t$-structures on quotient categories}\label{ss:quot-t}

We will need the following construction.

Suppose $\sC \in \DGCat_{cont}$ is equipped with a $t$-structure that
is compatible with filtered colimits. Let $i_*:\sC_0 \into \sC$
be a fully faithful functor admitting a continuous right adjoint $i^!$.
We suppose the full subcategory $\sC_0 \subset \sC$ is closed
under truncation functors for the $t$-structure; in particular,
$\sC_0$ admits a unique $t$-structure for which $i_*$ is $t$-exact.

Define $\o{\sC}$ as $\Ker(i^!:\sC \to \sC_0)$. We denote the
embedding of $\o{\sC}$ into $\sC$ by $j_*$. This embedding 
admits a left adjoint $\sF \mapsto \Coker(i_*i^! \sF \to \sF)$,
which we denote by $j^*:\sC \to \o{\sC}$.

\begin{lem}\label{l:quot-t}

Suppose that the functor $j_*j^*:\sC \to \sC$ is left $t$-exact.
Then there is a unique $t$-structure on $\o{\sC}$ such that
$j^*:\sC \to \sC$ is $t$-exact. 

\end{lem}

\begin{rem}\label{r:epaisse}

The hypothesis of the lemma is equivalent to the assertion that
for $\sF \in \sC^{\heart}$, the map $H^0(i_*i^!\sF) \to \sF$ is
a monomorphism in the abelian category $\sC^{\heart}$. In turn,
this assertion is well-known to be equivalent to $\sC_0^{\heart} \subset 
\sC^{\heart}$ being closed under subobjects. 

\end{rem}

\begin{proof}[Proof of Lemma \ref{l:quot-t}]

Define $\o{\sC}^{>0} \subset \o{\sC}$ as the
full subcategory of $\sF \in \o{\sC}$ with $j_*(\sF) \in \sC^{> 0}$.
Define $\o{\sC}^{\leq 0} \subset \o{\sC}$ as the left orthogonal
to $\sC^{>0}$. 

The functor $j^*:\sC \to \o{\sC}$ maps
$\sC^{\leq 0}$ to $\o{\sC}^{\leq 0}$ immediately from the
definition, and maps $\sC^{>0}$ to $\o{\sC}^{>0}$ by our
assumption that $j_*j^*$ is left $t$-exact.

In particular, for $\sF \in \o{\sC}$, 
$j^*\tau^{>0} j_*(\sF) \in \o{\sC}^{>0}$ 
and $j^*\tau^{\leq 0} j_*(\sF) \in \o{\sC}^{\leq 0}$.
As $j^*j_*(\sF) \isom \sF$, we see that we have in fact
defined a $t$-structure on $\o{\sC}$. By the previous paragraph,
the functor $j^*$ is $t$-exact as desired.

\end{proof}

\subsection{Subobjects in equivariant categories}

To apply the previous material, we use the following
result. 

\begin{prop}\label{p:eq-subobj}

Suppose $H$ is a connected, 
affine algebraic group acting strongly on
$\sC \in \DGCat_{cont}$. Suppose that $\sC$ is equipped with 
a $t$-structure compatible with the $H$-action. 

Then the functor $\sC^{H,\heart} \to \sC^{\heart}$ is fully faithful
and the resulting subcategory is closed under subobjects.

\end{prop}

\begin{proof}

In what follows, we let $\iota:\Spec(k) \to H$ denote the unit
for the group structure and we let $\o{H} \coloneqq H \setminus 1$
the complementary open with embedding $\jmath:\o{H} \to H$.

We let $\delta_1 = \iota_{*,dR}(k) \in D(H)$ denote the $\delta$ $D$-module
on $G$ supported at $1 \in H$, and we let $k_H \in D(H)$ 
(resp. $k_{\o{H}} \in D(\o{H})$) denote the constant $D$-module
(i.e., the $*$-dR pullback of $k \in D(\Spec(k)) = \Vect$).

\step\label{st:subobj-1}

We begin with reductions.

First, that $\sC^{H,\heart} \to \sC^{\heart}$ is fully faithful
for $H$ connected is 
well-known.\footnote{We recall the argument
for the reader's convenience. 
For $\sF \in \sC^{H,\heart}$,
we need to show that $\sF \to \Av_*^H\Oblv(\sF)$ 
gives an isomorphism after applying $H^0$.
Moreover, it suffices to do so after applying $\Oblv$.

The resulting map is obtained by 
($H$-equivariant) convolution with the canonical
map $k_H \to k_H \star k_H \in D(H)^{H}$.
Under the identification $D(H)^H = \Vect$ with $k \in \Vect$ corresponding
to $k_H \in D(H)^H$, the resulting map corresponds to 
$k \to \Gamma_{dR}(H,k_H)$.
Because $H$ is connected (hence, geometrically connected), 
this map is an isomorphism in degree $0$, giving the claim.}

By Remark \ref{r:epaisse}, it suffices to show that for
$\sF \in \sC^{H,\heart}$, the map:

\[
\Oblv\Av_*^H(\sF) \to \sF
\]

\noindent induces a monomorphism on $H^0$, or equivalently,
the (homotopy) cokernel of this map is coconnective.
As the above map is obtained by convolution with 
the map $k_H \to \delta_1 \in D(H)$, it suffices
to show that convolution with its cokernel,
which is $\jmath_!(k_{\o{H}})[1]$, is left $t$-exact.

\step\label{st:subobj-2} 

Let $\sF \in D(H)$ be given. Suppose the functor $\sF \star -:D(H) \to D(H)$
is left $t$-exact. We claim that the functor $\sF \star -:\sC \to \sC$
is left $t$-exact.

Indeed, by definition of the $t$-structure on $\sC$ being
compatible with the $H$-action, the functor
$\coact:\sC \to D(H) \otimes \sC$ is $t$-exact up to shift.
The functor $\coact$ is 
$H$-equivariant for the $H$-action on $D(H) \otimes \sC$ on the first
factor alone. Moreover, $\coact$ is conservative: its composition 
with $!$-restriction along the origin $\Spec(k) \into H$ is the
identity functor for $\sC$.

Therefore, the claim follows from \cite{whit} Lemma B.6.2.

\step 

By Step \ref{st:subobj-2}, we are reduced to showing that convolution with 
$\jmath_!(k_{\o{H}})[1]$ defines a left $t$-exact functor
$D(H) \to D(H)$. By the reasoning of Step \ref{st:subobj-1}, it is equivalent
to say that the essential image of the functor
$D(H)^{H,\heart} = \Vect^{\heart} \into D(H)^{\heart}$ is closed
under subobjects, which is evident: a sub $D$-module of a constant
one is itself constant.

\end{proof}

\subsection{An exactness criterion}

We begin with a scheme for checking that 
a functor between categories with (finite jets into)
$PGL_2$-actions is $t$-exact.

\begin{prop}\label{p:pgl2-exact}

Let $G = PGL_2$ and let $G_n$ be as in \S \ref{ss:gn-notation}.

Let $\sC, \sD \in G_n\mod$ be equipped with $t$-structures
compatible with the $G_n$-actions. 

Suppose $F:\sC \to \sD$ is a $G_n$-equivariant functor.

Then $F$ is left $t$-exact if and only if the functors:

\[
\sC^{N_n,\psi} \to \sD^{N_n,\psi}  
\text{ and }\begin{cases}
\sC^{N} \to \sD^{N} & n = 1 \\
\sC^{\fg \otimes \bG_a} \to \sD^{\fg \otimes \bG_a} & n \geq 2
\end{cases}
\]

\noindent are left $t$-exact, where $\fg \otimes \bG_a$
is embedded into $G_n$ via \eqref{eq:gn-lower-central}.

\end{prop}

Below, we give the proofs separately for $n = 1$ and $n \geq 2$.
We remark that in both cases, the ``only if" direction
is obvious.

\begin{proof}[Proof of Proposition \ref{p:pgl2-exact} for $n = 1$]

As we will see, in this case we 
only need the action of the 
Borel $B = T \ltimes N = \bG_m \ltimes \bG_a$ 
of $G = PGL_2$.

Define $\o{\sC}$ as $\Ker(\sC \xar{\Av_*^{\bG_a}} \sC^{\bG_a})$.
The embedding $\o{\sC} \into \sC$ admits a
left adjoint 
calculated as 
${\sF \mapsto \Coker(\Oblv\Av_*^{\bG_a}(\sF) \to \sF)}$.
By Lemma \ref{l:quot-t} and Proposition \ref{p:eq-subobj},
$\o{\sC}$ admits a unique
$t$-structure such that this functor $\sC \to \o{\sC}$
is $t$-exact.

The action functor $\act:D(\bG_m) \otimes \sC \to \sC$
maps $D(\bG_m) \otimes \sC^{\bG_a,\psi}$ into $\o{\sC}$,
and the resulting functor is an equivalence
(by Fourier transforming, c.f. \cite{beraldo-*/!}).
We claim that this equivalence is $t$-exact, where
$D(\bG_m) \otimes \sC^{\bG_a,\psi}$ is given the
tensor product $t$-structure.

To verify this, we will need the following commutative diagram:

\begin{equation}\label{eq:gm-ga}
\vcenter{\xymatrix{
D(\bG_m) \otimes \sC^{\bG_a,\psi} \ar[rr]^{\act} 
\ar[d]^{\id_{D(\bG_m)} \otimes \Oblv} &&
\o{\sC} \ar[d] \\
D(\bG_m) \otimes \o{\sC} \ar[rr] &&
D(\bG_m) \otimes \o{\sC} 
}}
\end{equation}

\noindent with morphisms as follows. 
The top arrow is
induced by the action functor from above.
The left arrow is $\id_{\bG_m}$ tensored
with the embedding $\sC^{\bG_a,\psi} \into \o{\sC}$ ($\subset \sC$).
For the right arrow, note that
$\o{\sC} \subset \sC$ is closed under the $\bG_a$-action,
and the corresponding coaction functor
$\coact:\o{\sC} \to D(\bG_a) \otimes \o{\sC}$ composed
with the Fourier transform $D(\bG_a) \simeq D(\bA^1)$ (tensored
with $\id_{\sC}$) maps into $D(\bA^1\setminus 0) \otimes \o{\sC}$;
we have identified $\bA^1 \setminus 0$ with $\bG_m$ here.
Finally, the bottom arrow is the unique
map of $D(\bG_m)$-comodule categories
whose composition with $\Gamma_{dR}(\bG_m,-) \otimes \id_{\o{\sC}}$
is $\act$ (the action functor for the $\bG_m$-action on $\o{\sC}$);
here $D(\bG_m)$ is a coalgebra in $\DGCat_{cont}$ via
diagonal pushforwards, and both sides are considered
as cofree comodules coinduced from $\o{\sC}$. (That the
diagram commutes is immediate.)

Now in \eqref{eq:gm-ga}, the bottom arrow is $t$-exact
by \cite{whit} Lemma B.6.2. By \cite{whit} Lemma B.6.2, 
the left arrow is $t$-exact
because $\sC^{\bG_a,\psi} \into \o{\sC}$ is
(as this functor coincides with the composition
$\sC^{\bG_a,\psi} \into \sC \to \o{\sC}$ of $t$-exact
functors). The right functor is $t$-exact because
the $t$-structure on $\sC$ is compatible with the $\bG_a$-action.
As the vertical arrows are fully faithful and the
bottom arrow is an equivalence, we obtain that the
top arrow is a $t$-exact equivalence as well.

We can now conclude the argument. By assumption 
and \cite{whit} Lemma B.6.2, the functor:

\[
\o{\sC} \simeq D(\bG_m) \otimes \sC^{\bG_a,\psi} \to 
D(\bG_m) \otimes \sD^{\bG_a,\psi} \simeq \o{\sD}
\]

\noindent is left $t$-exact.

Suppose 
$\sF \in \sC^{\geq 0}$. Then 
$\Oblv\Av_*^{\bG_a}(\sF) \in \sC^{\bG_a,\geq 0}$, so
$F(\Oblv\Av_*^{\bG_a}(\sF)) \in \sD^{\bG_a,\geq 0}$. 
Moreover, defining: 

\[
\o{\sF} \coloneqq \Coker(\Oblv\Av_*^{\bG_a}(\sF) \to \sF) \in 
\o{\sC}
\]

\noindent we have $\o{\sF} \in \o{\sC}^{\geq 0}$ by definition
of the $t$-structure on $\o{\sC}$. Therefore, 
$F(\o{\sF}) \in \o{\sD}^{\geq 0}$. Because the
embedding $\o{\sD} \into \sD$ is left $t$-exact
(being right adjoint to a $t$-exact functor), 
we obtain:

\[
\Oblv\Av_*^{\bG_a}F(\sF), \, \,
\Coker(\Oblv\Av_*^{\bG_a}F(\sF) \to F(\sF)) \in \sD^{\geq 0}
\]

\noindent implying $F(\sF) \in \sD^{\geq 0}$.
 
\end{proof}

\begin{proof}[Proof of Proposition \ref{p:pgl2-exact} for $n \geq 2$]

Let $\sC_{reg} \subset \sC$ be defined as in \S \ref{ss:fourier-decomp}.
The embedding $\sC_{reg} \into \sC$ admits a left adjoint $j^!$
as in \emph{loc. cit}. Moreover, because $G = PGL_2$,
the argument from \S \ref{ss:descent} shows that
$\Ker(j^!) = \sC^{\fg \otimes \fG_a}$. 
Applying Lemma \ref{l:quot-t} and Proposition \ref{p:eq-subobj}, 
we find that $\sC_{reg}$ admits a unique $t$-structure for
which $j^!$ is $t$-exact.

By Theorem \ref{t:whit-conv}, the convolution functor:

\[
D(G_n)^{N_n,\psi} \otimes \sC^{N_n,\psi} \to \sC
\]

\noindent admits a left adjoint 
$\Av_!^{\psi,-\psi} = \Av_*^{\psi,-\psi}[2\dim N_n]$.
By \cite{whit} Lemma B.6.1, $\Av_!^{\psi,-\psi}[-\dim N_n] = 
\Av_*^{\psi,-\psi}[\dim N_n]$ is $t$-exact.

Because the above convolution functor factors through
$\sC_{reg}$, $\Av_!^{\psi,-\psi}:\sC \to 
D(G_n)^{N_n,\psi} \otimes \sC^{N_n,\psi}$ coincides
with $\Av_!^{\psi,-\psi}j_{*,dR}j^!$.
By the above, we find that $\Av_!^{\psi,-\psi} \circ j_{*,dR}$
is $t$-exact. Moreover, by Corollary \ref{c:av-!-cons},
$\Av_!^{\psi,-\psi} \circ j_{*,dR}$ is conservative.

Therefore, as:

\[
D(G_n)^{N_n,\psi} \otimes \sC^{N_n,\psi} \xar{\id \otimes F} 
D(G_n)^{N_n,\psi} \otimes \sD^{N_n,\psi}
\]

\noindent is left $t$-exact by assumption and \cite{whit} Lemma B.6.2, 
the resulting functor $\sC_{reg} \to \sD_{reg}$ is left $t$-exact.

As $\sC^{\fg \otimes \bG_a} \to \sD_{\fg \otimes \bG_a}$ is
left $t$-exact by assumption, the argument concludes as in 
the $n = 1$ case.

\end{proof}

\subsection{Exactness of $\Gamma^{\Hecke,naive}$}

We can now show $t$-exactness.

\begin{proof}[Proof of Lemma \ref{l:naive-t-exact}]

For simplicity, we take $G = PGL_2$; the argument for
general $G$ of semisimple rank $1$ follows by the 
considerations of \S \ref{ss:rk1}.

By Corollary \ref{c:gamma-hecke-right-exact}, it remains to show
left $t$-exactness.
It suffices to show that for every $n \geq 1$, the functor:

\[
\Gamma^{\Hecke,naive}:D_{crit}^{\Heckez}(\Gr_G)^{K_n} \to 
\widehat{\fg}_{crit}\mod_{reg,naive}^{K_n}
\]

\noindent is left $t$-exact; here $K_n \subset G(O)$ is the $n$th congruence
subgroup.
We show this by induction on $n$.

First, we treat the $n = 1$ case. By Proposition \ref{p:pgl2-exact},
it suffices to show (left) $t$-exactness for the corresponding
functors:

\[
\begin{gathered}
D_{crit}^{\Heckez}(\Gr_G)^{\o{I}} \to 
\widehat{\fg}_{crit}\mod_{reg,naive}^{\o{I}} \\
\Whit^{\leq 1}(D_{crit}^{\Heckez}(\Gr_G)) \to 
\Whit^{\leq 1}(\widehat{\fg}_{crit}\mod_{reg,naive})
\end{gathered}
\]

\noindent These results follow from Theorems \ref{t:iwahori} and
\ref{t:whit-equiv}.

We now proceed by induction; we suppose the result is true for $n \geq 1$ 
and deduce it for $n+1$. By Proposition \ref{p:pgl2-exact}, it
suffices to show that the functors:

\[
\begin{gathered}
D_{crit}^{\Heckez}(\Gr_G)^{K_n} \to 
\widehat{\fg}_{crit}\mod_{reg,naive}^{K_n} \\
\Whit^{\leq n+1}(D_{crit}^{\Heckez}(\Gr_G)) \to 
\Whit^{\leq n+1}(\widehat{\fg}_{crit}\mod_{reg,naive})
\end{gathered}
\]

\noindent are (left) $t$-exact. The former is the inductive hypothesis
and the latter is Theorem \ref{t:whit-equiv}.

\end{proof}

\section{The renormalized category}\label{s:bddness}

\subsection{} In this section, we prove
Lemma \ref{l:gamma-hecke-bdded}. The argument is
quite similar to the proof of Lemma \ref{l:naive-t-exact}.

\subsection{A boundedness criterion}

The following result is a cousin of Proposition \ref{p:pgl2-exact}.

\begin{prop}\label{p:pgl2-bdded}

Let $G = PGL_2$ and let $G_n$ be as in \S \ref{ss:gn-notation}.

Let $\sC \in G_n\mod$ be equipped with a $t$-structure
compatible with the $G_n$-action. Suppose that $\sD$ is 
equipped with a $t$-structure compatible with filtered
colimits. Suppose $F:\sC \to \sD \in \DGCat_{cont}$ is given.

Then $F$ is left $t$-exact up to shift 
if and only if:

\[
\begin{cases}
F|_{\sC^N} & n = 1 \\
F|_{\sC^{\fg \otimes \bG_a}} & n \geq 2
\end{cases}
\]

\noindent is left $t$-exact up to shift, and
$F(\sG \star \sF) \in \sD^+$ for
every: 

\[
\sG \in D(G_n)^+, \,
\sF \in \sC^{N_n,\psi,+}.
\]

\end{prop}

\begin{rem}

We emphasize that there is no assumption here that $G_n$ acts
on $\sD$, in contrast to Proposition \ref{p:pgl2-exact}.

\end{rem}

\begin{rem}

The ``only if" direction of Proposition \ref{p:pgl2-bdded} 
is clear, as $\sG \star -: \sC \to \sC$ is left $t$-exact up to 
shift for $\sG \in D(G_n)^+$.

\end{rem}

\begin{proof}[Proof of Proposition \ref{p:pgl2-bdded} for $n = 1$]

As the $t$-structures on $\sC$ and $\sD$ are compatible with
filtered colimits, $F$ is left $t$-exact up to shift if and only
if $F(\sC^+) \subset \sD^+$. We verify the result in this form.

Suppose $\sF \in \sC^+$. Then $\Oblv\Av_*^N \sF \in \sC^{N,+}$, so 
by assumption $F(\Oblv\Av_*^N \sF) \in \sD^+$.
Therefore, setting
$\o{\sF} \coloneqq \Coker(\Oblv\Av_*^N \sF \to \sF)$, it suffices
to show that $F(\o{\sF}) \in \sD^+$.

As in the proof of Proposition \ref{p:pgl2-exact} (for $n = 1$), 
$\o{\sF}$ is
in the essential image of the fully faithful, $t$-exact convolution functor
$D(T) \otimes \sC^{N,\psi} \to \sC$. Therefore, it suffices to 
show that the composition:

\[
D(T) \otimes \sC^{N,\psi} \to \sC \xar{F} \sD
\]

\noindent is left $t$-exact up to shift. For convenience, in what
follows, we identify
$\o{\sF}$ with the correpsonding object of $D(T) \otimes \sC^{N,\psi}$.

For this, we observe that any object $\o{\sF} \in D(T) \otimes \sC^{N,\psi}$
lies in the full subcategory of $D(T) \otimes \sC^{N,\psi}$
generated under finite colimits and direct summands
by objects of the form $D_T \boxtimes (\Gamma(T,-) \otimes \id)(\o{\sF})$,
where $D_T \in D(T)^{\heart}$ is the sheaf of differential operators;
c.f. Lemma \ref{l:boxtimes} below.
Then by \cite{whit} Lemma B.6.2, 
for $\o{\sF} \in D(T) \otimes \sC^{N,\psi}$, we have:

\[
(\Gamma(T,-) \otimes \id)(\o{\sF}) \in \sC^{N,\psi,+}.
\]

\noindent Therefore, by assumption, 
$F(D_T \star (\Gamma(T,-) \otimes \id)(\o{\sF})) \in \sD^+$,
so we find that the same is true of $F(\o{\sF})$.

\end{proof}

Above, we used the following result.

\begin{lem}\label{l:boxtimes}

Let $S$ be a smooth affine scheme (over $\Spec(k)$). 

As is standard, let $\Oblv:D(S) \to \IndCoh(S) \overset{\Psi}{\simeq} \QCoh(S)$ 
denote the ``right" $D$-module 
forgetful functor from \cite{grbook} and let $\ind:\QCoh(S) \to D(S)$ denote
its left adjoint. Let $D_S \coloneqq \ind(\sO_S) \in D(S)^{\heart}$.
Let $\Gamma(S,-):D(S) \to \Vect$ denote the composition of
$\Oblv$ with the usual global sections functor on $\QCoh(S)$.

Then for any $\sC \in \DGCat_{cont}$ and any 
$\sF \in D(S) \otimes \sC$, $\sF$ lies in the full subcategory
of $D(S) \otimes \sC$ generated under finite colimits and
direct summands by objects of the form:

\[
D_S \boxtimes (\Gamma(S,-) \otimes \id_{\sC})(\sF).
\]

\end{lem}

\begin{proof}

As $S$ is affine, $D(S \times S)$ is compactly generated by
objects of the form $D_S \boxtimes D_S$. As $\Delta_{dR,*}(\omega_S)$
is compact and connective, 
it lies in the full subcategory generated under finite colimits and
direct summands by objects of the form $D_S \boxtimes D_S$. 

Identifying $D(S \times S)$ in the usual way with 
$\TwoEnd_{\DGCat_{cont}}(D(S))$ (c.f. \cite{grbook}), the object
$\Delta_{dR,*}(\omega_S)$ corresponds to the identity functor,
while $D_S \boxtimes D_S$ corresponds to $D_S \otimes \Gamma(S,-)$.

Therefore, $\id_{D(S) \otimes \sC}$ lies in the full subcategory
of $\TwoEnd_{\DGCat_{cont}}(D(S) \otimes \sC)$ generated under
finite colimits and direct summands by endofunctors of the
form $(D_S \otimes \Gamma(S,-)) \otimes \id_{\sC}$. Applying such a resolution
to the object $\sF$, we obtain the claim.

\end{proof}

We now turn to the higher $n$ case.

\begin{proof}[Proof of Proposition \ref{p:pgl2-bdded} for $n \geq 2$]

Suppose $\sF \in \sC^+$. Then $\Oblv\Av_*^{\fg \otimes \bG_a} \sF \in 
\sC^{\fg \otimes \bG_a,+}$, so by assumption 
$F(\Oblv\Av_*^{\fg \otimes \bG_a} \sF) \in \sD^+$.
Therefore, setting
$\o{\sF} \coloneqq \Coker(\Oblv\Av_*^{\fg \otimes \bG_a} \sF \to \sF)$, 
it suffices to show that $F(\o{\sF}) \in \sD^+$.

As $G = PGL_2$, $\o{\sF} \in \sC_{reg}$, hence in $\sC_{reg}^+$. 
Now the claim follows as in the $n = 1$ case 
from our assumption and Corollary \ref{c:reg-eff}.

\end{proof}

\subsection{Boundedness of $\Gamma^{\Hecke}$}

We now boundedness of the non-naive version of the Hecke 
global sections functor.

\begin{proof}[Proof of Lemma \ref{l:gamma-hecke-bdded}]

We again assume $G = PGL_2$ for simplicity, referring
to \S \ref{ss:rk1} for indications on general $G$ of
semisimple rank $1$.

It suffices to show the result for $K$ being the $n$th congruence
subgroup of $G(O)$ for some $n \geq 1$. We proceed by induction on 
$n$.

For $\sF \in D_{crit}^{\Heckez}(\Gr_G)^{\o{I},+}$,
$\Gamma^{\Hecke}(\Gr_G,\sF) \in 
\widehat{\fg}_{crit}\mod_{reg}^+$: 
this follows from Proposition \ref{p:hecke-true}.

Next, suppose that $\sG \in D(G)^+$ and
$\sF \in \Whit^{\leq 1}(D_{crit}^{\Heckez}(\Gr_G))^+$.
Then $\Gamma^{\Hecke}(\Gr_G,\sG \star \sF) \in 
\widehat{\fg}_{crit}\mod_{reg}^+$ by 
Proposition \ref{p:hecke-true-whit}. 

Therefore, Proposition \ref{p:pgl2-bdded} implies the
$n = 1$ case of the claim.

We now suppose the result is true for some $n$ and
deduce it for $n+1$. 
The inductive hypothesis
states that $\Gamma^{\Hecke}(\Gr_G,\sG \star \sF)$ is eventually
coconnective
for $\sF \in D_{crit}^{\Heckez}(\Gr_G)^{K_n,+}$,
while Proposition \ref{p:hecke-true-whit} implies the result
if $\sF \in \Whit^{\leq n+1}(D_{crit}^{\Heckez}(\Gr_G))^+$.
Therefore, Proposition \ref{p:pgl2-bdded} gives the
result for general $\sF \in D_{crit}^{\Heckez}(\Gr_G)^{K_{n+1},+}$.

\end{proof}

\appendix

\section{The global sections functor}\label{a:gamma}

\subsection{}

Let $\kappa$ be a level for $\fg$. In this appendix,
we define a global sections functor:

\[
\Gamma(G(K),-):D_{\kappa}^*(G(K)) \to 
\widehat{\fg}_{\kappa}\mod \otimes
\widehat{\fg}_{-\kappa+2\cdot crit}\mod.
\]

Moreover, we show the following basic property:

\begin{prop}\label{p:gamma-exact}

The functor $\Gamma(G(K),-)$ is $t$-exact
for the natural $t$-structure on $D^*(G(K))$.

\end{prop}

To define both $\Gamma$ and the ``natural" $t$-structure
mentioned above, there is an implicit choice
of compact open subgroup of $G(K)$ (or rather,
its Tate extension)
that goes into the definitions.
For definiteness, we choose $G(O)$ in what follows.

Abelian categorically, this construction
is well-known from \cite{cdo}. Our setup is a little
different from \emph{loc. cit}., so we indicate
basic definitions and properties. We compare
our construction to theirs in 
Proposition \ref{p:gamma-comparison}.

\subsection{Definition of the functor}

By \cite{methods} \S 11.9, we have a canonical
isomorphism:

\[
\widehat{\fg}_{\kappa}\mod^{\vee} \simeq
\widehat{\fg}_{-\kappa+2\cdot crit}\mod.
\]

\noindent Here the left hand side is the
dual in $\DGCat_{cont}$. This isomorphism
depends (mildly) on our choice $G(O)$ of
compact open subgroup of $G(K)$.
This isomorphism is a refinement of
the usual semi-infinite
cohomology construction; more precisely, by \emph{loc. cit}., 
the pairing:

\[
\widehat{\fg}_{\kappa}\mod \otimes 
\widehat{\fg}_{-\kappa+2\cdot crit}\mod \to \Vect
\] 

\noindent is calculated by tensoring Kac-Moody 
representations and then taking semi-infinite 
cohomology for the diagonal action.

In addition, by \cite{methods} \S 8,
we have a level $\kappa$ $G(K)$-action on 
$\widehat{\fg}_{\kappa}\mod$. 

Therefore, we obtain a functor:

\[
D_{\kappa}^*(G(K)) \to
\TwoEnd_{\DGCat_{cont}}(\widehat{\fg}_{\kappa}\mod) 
\simeq \widehat{\fg}_{\kappa}\mod \otimes
\widehat{\fg}_{-\kappa+2\cdot crit}\mod.
\]

\noindent By definition, the resulting
functor is $\Gamma(G(K),-)$.

\subsection{Definition of the $t$-structure}

The choice of $G(O)$ also defines a $t$-structure
on $D_{\kappa}^*(G(K))$: we write $D_{\kappa}^*(G(K))$ as 
$\colim_n D_{\kappa}(G(K)/K_n)$ under 
$*$-pullback functors;
the structure functors are $t$-exact
up to shift by smoothness of the structure maps, 
so there is a unique
$t$-structure such that the pullback
functor $\pi_n^{*,dR}[-\dim G(O)/K_n]:
D_{\kappa}(G(K)/K_n) \to 
D_{\kappa}^*(G(K))$ is $t$-exact for all $n$.

\subsection{$t$-exactness}

Below, we prove
Proposition \ref{p:gamma-exact}.

\subsection{}

Because compact objects in $D_{\kappa}^*(G(K))$
are bounded in the $t$-structure and 
closed under truncations, it suffices to show
that for 
$\sF \in D_{\kappa}^*(G(K))^{\heart}$ compact
in $D_{\kappa}^*(G(K))$, 
$\Gamma(G(K),\sF) \in 
(\widehat{\fg}_{\kappa}\mod \otimes
\widehat{\fg}_{-\kappa+2\cdot crit}\mod)^{\heart}$.

We fix such an $\sF$ in what follows.

\subsection{}

Because $\sF$ is compact, there exists a
positive integer $r$ such that $\sF$ is 
$K_r$-equivariant on the right.
Moreover, by compactness again, $\sF$ is
supported on some closed subscheme $S \subset G(K)$,
which we may assume is preserved under the
right $K_r$-action. 

Note that $S$ is necessarily
affine as $G(K)$ is ind-affine. 
We have $S = \lim S/K_{r+r^{\prime}}$, so 
by Noetherian approximation, $S/K_{r+r^{\prime}}$
is affine for some $r^{\prime} \geq 0$.
Up to replacing $r$ by $r+r^{\prime}$, we may
assume $S/K_r$ itself is affine.

\subsection{}

For any two integers 
$m_1,m_2>0$, we have:

\[
\begin{gathered}
\ul{\Hom}_{\widehat{\fg}_{\kappa}\mod
\otimes \widehat{\fg}_{-\kappa+2\cdot crit}\mod}
(\bV_{\kappa,m_1}
\boxtimes \bV_{-\kappa+2\cdot crit,m_2},
\Gamma(G(K),\sF)) = \\
\ul{\Hom}_{\widehat{\fg}_{\kappa}\mod}(\bV_{\kappa,m_1},
\sF \star \bD \bV_{-\kappa+2\cdot crit,m_2}).
\end{gathered}
\]

\noindent by definition of $\Gamma$.
Here $\bD: (\widehat{\fg}_{-\kappa+2\cdot crit}\mod^c)^{op} 
\simeq \widehat{\fg}_{\kappa}\mod^c$ 
is the isomorphism
defined by the (semi-infinite) duality
$\widehat{\fg}_{\kappa}\mod \simeq 
\widehat{\fg}_{\kappa}\mod^{\vee}$ used above.

To see that $\Gamma(G(K),\sF)$ is in degrees
$\geq 0$, it suffices to see that the
above complex is in degrees $\geq 0$ for
all $m_1,m_2$. Moreover, it suffices to check
this for all sufficiently large $m_1,m_2$; 
we will do so for $m_1,m_2\geq r$.

Then to see that $\Gamma(G(K),\sF)$ is in 
degree $0$, it suffices to show that when
we pass to the limits $m_1,m_2 \to \infty$ 
(using the standard structure maps between
our modules as we vary these parameters),
we obtain a complex in degree $0$. 
In fact, we will see that already
$\sF \star \sF \star \bD \bV_{-\kappa+2\cdot crit,m_2}$
is in the heart of the $t$-structure 
(for $m_2 \geq r$), which clearly suffices.

\subsection{}

By \cite{methods} Lemma 9.17.1,
$\bD \bV_{-\kappa+2\cdot crit,m_2} \simeq
\bV_{\kappa,m_2}[\dim G(O)/K_{m_2}] = 
\bV_{\kappa,m_2}[m_2 \cdot \dim G]$.

We then have 
$\sF \star \bV_{\kappa,r} = 
\Gamma^{\IndCoh}(G(K)/K_r,\sF)$; here
we have descended $\sF$ by $K_r$-equivariance
to a $D$-module on $G(K)/K_r$ and then 
we have calculated its $\IndCoh$-global sections.

Putting these together, we find:

\[
\begin{gathered}
\ul{\Hom}_{\widehat{\fg}_{\kappa}\mod}(\bV_{\kappa,m_1},
\sF \star \bD \bV_{\kappa,m_2}) = \\
\ul{\Hom}_{\widehat{\fg}_{\kappa}\mod}(\bV_{\kappa,m_1},
\sF \star 
\bV_{\kappa,r}[m_2 \cdot \dim G-(m_2-r) \dim G]) = \\
\ul{\Hom}_{\widehat{\fg}_{\kappa}\mod}(\bV_{\kappa,m_1},
\sF \star 
\bV_{\kappa,r}[r \cdot \dim G]).
\end{gathered}
\]

By Lemma \ref{l:action-tstr-bdd} (and 
\cite{methods} Proposition 10.16.1), 
$\sF \star \bV_{\kappa,r} \in 
\widehat{\fg}_{\kappa}\mod^+$. Moreover,
by \S \ref{ss:conv-vac-gamma} below, 
$\sF \star \bV_{\kappa,r}$ maps under the
forgetful functor to $\Vect$ to
$\Gamma^{\IndCoh}(G(K)/K_r,\sF) \in \Vect$
(i.e., descend $\sF$ to $G(K)/K_r$ and take 
$\IndCoh$-global sections).

As $\sF \in D_{\kappa}^*(G(K))^{\heart}$, 
when we consider $\sF$ as an object of
$D_{\kappa}(G(K)/K_r)$, it lies in 
cohomological degree $\dim(G(O)/K_r) = r\cdot \dim G$.
Therefore, the same is true when we forget to
$\IndCoh(G(K)/K_r)$, as that forgetful functor
is $t$-exact (c.f. \cite{grbook}). Finally,
as $\sF$ is supported on an affine subscheme
of $G(K)/K_r$ by construction, 
$\Gamma^{\IndCoh}(G(K)/K_r,\sF)$ is in 
cohomological degree $r \cdot \dim G$.

Combining this with the above, we find that
$\sF \star \bV_{\kappa,r}[r \cdot \dim G]) 
\in \widehat{\fg}_{\kappa}\mod^{\heart}$. This
gives the desired claims, proving
Proposition \ref{p:gamma-exact} modulo the
above assertion.

\subsection{}\label{ss:conv-vac-gamma}

Above, we needed the following 
observation.

Suppose $\sF \in D_{\kappa}^*(G(K))^{K_r}$.
We claim that $\Oblv(\sF \star \bV_{\kappa,r}) = 
\Gamma^{\IndCoh}(G(K)/K_r,\sF) \in \Vect$,
where we implicitly descend 
$\sF$ to $G(K)/K_r$ through 
equivariance. 

To simplify the notation, we omit the level
$\kappa$ and work with a general Tate group
indscheme $H$ and a compact open subgroup $K$.
(Then the level may easily be reincorporated in a standard way
by taking $H$ to be the Tate extension of $G(K)$, c.f.
\cite{methods} \S 11.3.)

For any $\sC \in H\mod_{weak}$,
suppose $\sG \in \sC^{K,w}$ and 
$\sF \in D(H/K)$. As in 
\cite{methods} \S 8, 
$D(H/K)$ is canonically isomorphic
to $\IndCoh^*(H/K)_{H_K^{\wedge}}$,
with $H_K^{\wedge}$ the formal completion
of $H$ along $K$. Moreover, the functor
$\Oblv:\sC^{H_K^{\wedge},w} \to 
\sC^{K,w}$ admits a left adjoint, 
which we denote by $\Av_!^w$. 

Then we claim that we have isomorphisms:

\[
\sF \overset{H_K^{\wedge},w}{\star} 
\Av_!^w(\sG) = \Oblv(\sF) \overset{K,w}{\star}\sG 
\in \sC
\]

\noindent functorial in $\sF$ and $\sG$
(i.e., an isomorphism of functors
$D(H/K) \otimes \sC^{K,w} \to \sC$).
Here for the convolution on the
left, we regard $\sF$ as an object of
$\IndCoh^*(H)_{H_K^{\wedge}}$ as above. 
The notation $\overset{H_K^{\wedge},w}{\star} $
means we convolve (in the setting of weak actions)
over $H_K^{\wedge}$, and similarly on the 
right hand side.
Then $\Av_!^w(\sG) = 
\omega_{H_K^{\wedge}/K} \overset{K,w}{\star} \sG$, and 
$\sF \overset{H_K^{\wedge},w}{\star}
\omega_{H_K^{\wedge}/K} = \Oblv(\sF)$, so we obtain
the claim.

Now taking $\sC = \Vect$ and $\sG = k$ the trivial
representation in $\Vect^{K,w} = \Rep(K)$, 
we obtain:

\[
\sF \star \ind_{\fk}^{\fh}(k) = \Oblv(\sF) \star k
\in \Vect.
\]

\noindent The right hand side calculates 
$\Gamma^{\IndCoh}(H/K,\Oblv(\sF))$ as desired.

\subsection{Comparison with Arkhipov-Gaitsgory}

To conclude, we observe that our construction
above recovers the one given by Arkhipov-Gaitsgory.

More precisely, $D^*(G(K))^{\heart}$
manifestly coincides with the abelian category denoted
$D\mod(G((t)))$ in \S 6.10 of \cite{cdo}, and 
similarly with a level $\kappa$ included
(which they discuss only in passing).

Below, we outline the proof of the following
comparison result.

\begin{prop}\label{p:gamma-comparison}

The functor:

\[
\Gamma(G(K),-):D_{\kappa}^*(G(K))^{\heart} \to 
(\widehat{\fg}_{\kappa}\mod \otimes 
\widehat{\fg}_{-\kappa+2\cdot crit}\mod)^{\heart} =
\widehat{\fg \times \fg}_{(\kappa,-\kappa+2\cdot crit)}
\mod^{\heart} 
\]

\noindent constructed above coincides with the
one constructed in \cite{cdo}.

\end{prop}

\begin{proof}

\step 

Define $\on{CDO}_{G,\kappa} \in \Vect$ as 
$\Gamma(G(K),\delta_{G(O)})$, where
$\delta_{G(O)} \in D_{\kappa}^*(G(K))$ is the
$*$-pullback of $\delta_1 \in D_{\kappa}(\Gr_G)$.

As $\delta_{G(O)} \in D^*(G(K))^{\heart}$,
$\on{CDO}_{G,\kappa} \in \Vect^{\heart}$.

The object $\delta_{G(O)}$ manifestly upgrades
to a factorization algebra in the factorization
category with fiber $D_{\kappa}^*(G(K))$
(defined using the standard unital factorization 
structure on $G(K)$, c.f. \cite{cpsii} \S 2).
Therefore, by \cite{chiral},
$\on{CDO}_{G,\kappa}$ has a natural
vertex algebra structure.

Note that $\on{CDO}_{G,\kappa}$ has
commuting $\widehat{\fg}_{\kappa}$ and
$\widehat{\fg}_{-\kappa+2\cdot crit}$-actions.

There is a tautological map $\Fun(G(O)) \to 
\on{CDO}_{G,\kappa} \in \Vect^{\heart}$, 
which is compatible with
factorization and is a morphism of 
$\fg[[t]]$-bimodules. Regarding 
$\on{CDO}_{G,\kappa}$ as a 
$\widehat{\fg}_{\kappa}$-module,
we obtain an induced map:

\[
\ind_{\fg[[t]]}^{\widehat{\fg}_{\kappa}}(\Fun(G(O)))
\to \on{CDO}_{G,\kappa} \in \Vect^{\heart}.
\]

\noindent In \cite{cdo}, a natural vertex
algebra structure is defined on the left hand side.
We claim that this map is an isomorphism of vertex
algebras.

Indeed, the construction of the vertex
algebra structure from \cite{cdo}
exactly uses factorization geometry, showing
that the map above is a map of vertex algebras.

This map is an isomorphism because both sides
have standard filtrations and
the map is an isomorphism at the associated graded level.

\step 

Next, \cite{cdo} constructs a 
$\widehat{\fg}_{-\kappa+2\cdot crit}$-action
on 
$\ind_{\fg[[t]]}^{\widehat{\fg}_{\kappa}}(\Fun(G(O)))$.
We claim that the above isomorphism 
is an isomorphism of 
$\widehat{\fg}_{-\kappa+2\cdot crit}$-modules as well.

We regard both sides as objects of:

\[
(\widehat{\fg}_{\kappa}\mod \otimes \widehat{\fg}_{-\kappa+2\cdot crit}\mod)^{\heart} \subset 
\widehat{\fg}_{\kappa}\mod \otimes \widehat{\fg}_{-\kappa+2\cdot crit}\mod \simeq 
\TwoEnd_{\DGCat_{cont}}(\widehat{\fg}_{\kappa}\mod).
\]

By construction, 
$\on{CDO}_{G,\kappa}$ corresponds to the endofunctor
$\Oblv\Av_*^{G(O)}: \widehat{\fg}_{\kappa}\mod \to 
\widehat{\fg}_{\kappa}\mod$.

By \cite{methods} Theorem 9.16.1, 
the functor:

\[
\widehat{\fg}_{\kappa}\mod^+ \to 
\widehat{\fg}_{\kappa}\mod
\]

\noindent corresponding to an object:

\[
M \in 
(\widehat{\fg}_{\kappa}\mod \otimes 
\widehat{\fg}_{-\kappa+2\cdot crit}\mod)^{\heart}
\]

\noindent is the functor:

\[
N \mapsto C^{\sinf}(\fg((t)),\fg[[t]];M \otimes N).
\]

\noindent Here the right hand side is the functor
of $G(O)$-integrable semi-infinite cohomology,
which is defined because $M \otimes N$ is 
a Kac-Moody module with level $2\cdot crit$.

By \cite{cdo} Theorem 5.5, 
we have:

\[
C^{\sinf}(\fg((t)),\fg[[t]];
\ind_{\fg[[t]]}^{\widehat{\fg}_{\kappa}}(\Fun(G(O)))
 \otimes N) = 
\Oblv\Av_*^{G(O)}(N)
\]

\noindent as desired.
(More precisely, one needs to upgrade \cite{cdo}
a bit; this is done in \cite{fg2} Lemma 22.6.2,
where we note that the definition of convolution
in \emph{loc. cit}. involves tensoring
and forming semi-infinite cohomology.)

This gives the desired isomorphism of
modules with two commuting Kac-Moody symmetries; 
this isomorphism is readily seen
to coincide with the one constructed earlier.

\step\label{st:km-cdo}

The functor:

\[
\Gamma(G(K),-):D_{\kappa}^*(G(K)) \to 
\widehat{\fg}_{\kappa}\mod \otimes 
\widehat{\fg}_{-\kappa+2\cdot crit}\mod
\]

\noindent canonically upgrades to a functor
between factorization categories.
This induces a canonical morphism of vertex
algebras:

\[
\bV_{\fg,\kappa} \otimes \bV_{\fg,-\kappa+2\cdot crit}
\to \on{CDO}_{G,\kappa}.
\]  

This map coincides with the one constructed
in \cite{cdo}; indeed, both are given by acting
on the unit vector $1 \in \Fun(G(O)) \subset 
\on{CDO}_{G,\kappa}$ using the
Kac-Moody action, and we have shown that our
Kac-Moody action coincides with the one 
in \cite{cdo}.

\step\label{st:ag-actions}

Now suppose $\sF \in D_{\kappa}^*(G(K))^{\heart}$.
By construction, $\Gamma(G(K),\sF) \in \Vect^{\heart}$
carries an action of $\widehat{\fg}_{\kappa}$
and of $\Fun(G(K))$ (considered as a topological
algebra).

These two actions coincide with the ones
considered in \cite{cdo}. 
Indeed, this is tautological for $\Fun(G(K))$.

For $\widehat{\fg}_{\kappa}$, we are reduced
to showing that for $K_n \subset G(O)$ the
$n$th congruence subgroup and 
$\sF \in D_{\kappa}(G(K)/K_n)^{\heart}$, the
two actions of $\widehat{\fg}_{\kappa}$
on $H^0(\Gamma(G(K)/K_n,\sF))$ coincide.

This is a general assertion about Tate Lie algebras:
for $H$ a Tate group indscheme and $S$ a classical
indscheme with an action of $H$, the above
logic defines $\Gamma^{\IndCoh}(S,-):D(S) \to 
\fh\mod$, and we claim that 
$\sF \in D(S)^{\heart}$, this
action of $\fh$ on
$H^0\Gamma(S,\sF)$ coincides with the standard
one. This can be checked element by element in $\fh$,
so reduces to the case where $\fh$
is 1-dimensional. There it follows by
the construction of the 
comparison results in \cite{crystals}.

\step 

Because $\delta_{G(O)}$ is the 
\emph{unit} object in the unital factorization
category $D_{\kappa}^*(G(K))$ (see \cite{cpsii} \S 2), 
\cite{chiral} Proposition 8.14.1 shows that
$\Gamma$ upgrades to a functor:

\[
\Gamma(G(K),-):D_{\kappa}^*(G(K)) \to 
\on{CDO}_{G,\kappa}\mod_{un}^{\on{fact}}.
\]

\noindent Here we use the notation from \cite{chiral},
and are not distinguishing in the notation 
between our factorization
algebra and its fiber at a point.

Comparing with the construction in \cite{cdo}
and applying Step \ref{st:ag-actions},
we find that on abelian categories that the
functor:

\[
D_{\kappa}^*(G(K))^{\heart} \to 
\on{CDO}_{G,\kappa}\mod_{un}^{\on{fact},\heart} \simeq 
\ind_{\fg[[t]]}^{\widehat{\fg}_{\kappa}}(\Fun(G(O)))
\mod_{un}^{\on{fact},\heart}
\]

\noindent coincides with the one constructed
in \emph{loc. cit}. 

Now the assertion follows from Step \ref{st:km-cdo}.

\end{proof}

\begin{cor}\label{c:conv-comparison}

For every $\sF \in D_{\kappa}^*(G(K))^{\heart}$
compact, the functor:

\[
\sF \star -: \widehat{\fg}_{\kappa}\mod^+ 
\to \widehat{\fg}_{\kappa}\mod^+
\]

\noindent coincides with the similarly-named
functor constructed in \cite{fg2} \S 22.

\end{cor}

\begin{proof}

By construction of
$\Gamma(G(K),-)$, the following diagram commutes:

\[
\xymatrix{
D_{\kappa}^*(G(K)) \otimes \widehat{\fg}_{\kappa}\mod
\ar[rr]^(.4){\Gamma(G(K),-) \otimes \id} \ar[drr]_{-\star-}
& & \widehat{\fg}_{\kappa}\mod \otimes  
\widehat{\fg}_{-\kappa+2\cdot crit}\mod \otimes    
\widehat{\fg}_{\kappa}\mod 
\ar[d]^{\id \otimes \langle -,-\rangle} 
\\ 
& & 
\widehat{\fg}_{\kappa}\mod.
}
\] 

By \cite{methods} Theorem 9.16.1, this means
that for $M \in \widehat{\fg}_{\kappa}\mod^+$,
we have:

\[
C^{\sinf}(\fg((t)),\fg[[t]];
\Gamma(G(K),\sF) \otimes
M) \in \widehat{\fg}_{\kappa}\mod.
\]

\noindent Here we tensor and use the
diagonal action mixing the level $-\kappa+2\cdot crit$
action on $\Gamma(G(K),\sF)$ and the given 
level $\kappa$ action on $M$, and then 
we form the semi-infinite cochain complex,
which retains a level $\kappa$ action from
the corresponding action on 
$\Gamma(G(K),\sF)$. 

By Proposition \ref{p:gamma-comparison},
the latter amounts to the definition of
convolution given in \cite{fg2} \S 22.5
(see also \emph{loc. cit}. \S 22.7).

\end{proof}

One can similarly show that this isomorphism
is compatible with the associativity isomorphisms
constructed in \emph{loc. cit}. \S 22.9.

\section{Fully faithfulness}\label{a:ff}

\subsection{}

In this appendix, we present a different proof of 
Theorem \ref{t:ff} (fully faithfulness of 
$\Gamma^{\Hecke}$) than
the one given in \cite{fg2}.

\subsection{}

We have the following general criterion.

\begin{prop}\label{p:sph-+}

Suppose $\sC_i \in G(K)\mod_{crit}$ are given for $i = 1,2$.
Suppose that each $\sC_i$ is equipped with a $t$-structure
such that:

\begin{itemize}

\item The $t$-structure that
is strongly compatible with the $G(K)$-action.

\item The functor 
$\Av_!^{\psi}:\sC_i^{G(O)} \to \Whit(\sC_i)$ is $t$-exact
for $\Av_!^{\psi}$ as in \S \ref{ss:av-!-psi-sph}.
Here $\Whit(\sC_i)$ is equipped with the $t$-structure
coming from \cite{whit} Theorem 2.7.1 and \S B.7.

\item The functor $\Av_!^{\psi}:\sC_i^{G(O),\heart} \to \Whit(\sC_i)^{\heart}$
is conservative.

\end{itemize}

Suppose that $F:\sC_1 \to \sC_2 \in G(K)\mod_{crit}$ is given.
We suppose that the induced functor $\sC_1^{G(O)} \to \sC_2^{G(O)}$
is $t$-exact.

Then if the induced functor $\Whit(\sC_1) \to \Whit(\sC_2)$ is a
$t$-exact equivalence, the functor $\sC_1^{G(O),+} \to \sC_2^{G(O),+}$ is
as well.

\end{prop}

\begin{proof}

For $i = 1,2$, the functor 
$\Av_!^{\psi}:\sC_i^{G(O),+} \to \Whit(\sC_i)^+$ is $t$-exact and
conservative by assumption. Moreover, this functor admits
the right adjoint $\Av_*^{G(O)}$. By \cite{methods} Lemma 3.7.2, 
the functor $\sC_i^{G(O),+} \to \Whit(\sC_i)^+$ is comonadic.

Being $G(K)$-equivariant, the functor $F$ intertwines the comonads
$\Av_!^{\psi}\Av_*^{G(O)}$ on $\Whit(\sC_1)$ and $\Whit(\sC_2)$. 
Therefore, as we have assumed $F$ induces an equivalence
$\Whit(\sC_1)^+ \isom \Whit(\sC_2)^+$, we obtain the result.

\end{proof}

\subsection{}

We now deduce the following result.

\begin{cor}\label{c:ff-sph}

The functor:

\[
\Gamma^{\Hecke,naive}:D_{crit}^{\Heckez}(\Gr_G)^{G(O),+} \to 
\widehat{\fg}_{crit}\mod_{reg,naive}^{G(O),+}
\]

\noindent is a $t$-exact equivalence.

\end{cor}

\begin{proof}

We apply Proposition \ref{p:sph-+} with $\sC_1 = D_{crit}^{\Heckez}(\Gr_G)$,
$\sC_2 = \widehat{\fg}_{crit}\mod_{reg,naive}$, and $F = \Gamma^{\Hecke,naive}$.
It remains to check the hypotheses.

Both $t$-structures are strongly compatible with $t$-structures
by \cite{methods}. 

The functor $\Av_!^{\psi}:D_{crit}(\Gr_G)^{G(O)} \to 
\Whit(D_{crit}(\Gr_G))$ is $t$-exact and an equivalence (in particular,
conservative)
on the hearts of the $t$-structures by Theorem \ref{t:whit-sph}.
We deduce the same for $D_{crit}^{\Heckez}(\Gr_G)$
by \S \ref{ss:ind-hecke}-\ref{ss:hecke-t-str}.

The functor $\Av_!^{\psi}:\widehat{\fg}_{crit}\mod^{G(O)} \to 
\Whit(\widehat{\fg}_{crit}\mod)$ is $t$-exact by \cite{whit} Theorem 7.2.1.
We immediately deduce the same for $\widehat{\fg}_{crit}\mod_{reg,naive}$.
The functor:

\[
\Av_!^{\psi}:\widehat{\fg}_{crit}\mod_{reg,naive}^{G(O),\heart}
\to \Whit(\widehat{\fg}_{crit}\mod_{reg,naive})^{\heart} 
\overset{Cor. \ref{c:crit-whit-naive}}{=} 
\QCoh(\Op_{\ld{G}}^{reg})^{\heart}
\]

\noindent is an equivalence by \cite{fg1} Theorem 5.3. 

Finally, $\Gamma^{\Hecke,naive}$ restricted to 
$D_{crit}^{\Heckez}(\Gr_G)^{G(O)}$ is $t$-exact by 
Theorem \ref{t:iwahori}, and similarly for
$\Whit(D_{crit}^{\Heckez}(\Gr_G))$ by Theorem \ref{t:whit-equiv}.

\end{proof}

\subsection{}

We now prove Theorem \ref{t:ff}. The reductions
follow \cite{fg2}; only the last step differs.

\begin{proof}[Proof of Theorem \ref{t:ff}]

\step 

Recall from \S \ref{ss:ind-hecke}
that $D_{crit}^{\Hecke}(\Gr_G)$ is compactly generated
by objects of the form $\ind^{\Heckez}(\sF)$ for
$\sF \in D_{crit}(\Gr_G)$ compact. Moreover, $\Gamma^{\Hecke}$
preserves compact objects by construction.
Therefore, it suffices to show
that the map:

\[
\ul{\Hom}_{D_{crit}^{\Heckez}(\Gr_G)}(\ind^{\Heckez}(\sF),
\ind^{\Heckez}(\sG)) \to 
\ul{\Hom}_{\widehat{\fg}_{crit}\mod_{reg}}(\Gamma^{\IndCoh}(\sF),
\Gamma^{\IndCoh}(\sG)) 
\]

\noindent is an equivalence for $\sF,\sG \in D_{crit}(\Gr_G)$ compact.

As $\Gamma^{\IndCoh}(\sF) \in \widehat{\fg}_{crit}\mod_{reg}^c \subset
\widehat{\fg}_{crit}\mod_{reg}^+$, it suffices to show that
if we apply $\rho$, then the induced map:

\[
\ul{\Hom}_{D_{crit}^{\Heckez}(\Gr_G)}(\ind^{\Heckez}(\sF),
\ind^{\Heckez}(\sG)) \to 
\ul{\Hom}_{\widehat{\fg}_{crit}\mod_{reg,naive}}(\Gamma^{\IndCoh}(\sF),
\Gamma^{\IndCoh}(\sG))
\]

\noindent is an equivalence.

We will show this below with the weaker assumption that 
$\sG \in D_{crit}(\Gr_G)^+$.

\step 

By Lemma \ref{l:conv-adj} (and its proof), we can rewrite the
above terms as:

\[
\begin{gathered}
\ul{\Hom}_{D_{crit}^{\Heckez}(\Gr_G)^{G(O)}}(\ind^{\Heckez}(\delta_1),
\ind^{\Heckez}(\on{inv}\bD(\sF) \star \sG)) \to \\
\ul{\Hom}_{\widehat{\fg}_{crit}\mod_{reg,naive}^{G(O)}}(\Gamma^{\IndCoh}(\delta_1),
\Gamma^{\IndCoh}(\on{inv}\bD(\sF) \star \sG)).
\end{gathered}
\]

Noting that all the terms that appear here are eventually
coconnective in the relevant $t$-structures,  
the claim follows from Corollary \ref{c:ff-sph}.

\end{proof}

\bibliography{bibtex.bib}{}
\bibliographystyle{alphanum}

\end{document}